\documentclass[11pt]{article}

\RequirePackage[OT1]{fontenc}
\RequirePackage{amsthm,amsmath}
\RequirePackage[numbers]{natbib}
\RequirePackage[colorlinks,citecolor=blue,urlcolor=blue]{hyperref}
\usepackage{graphicx}

\usepackage{amsfonts,amssymb,amstext}
\usepackage{amsmath}
\usepackage{amstext}
\usepackage{amsopn}
\usepackage{amsthm}

\makeatletter

\newcommand{\tr}{\hbox{tr}}
\newcommand{\n}{\noindent}
\newcommand{\q}{\quad}
\newcommand{\E}{\mathbb{E}}
\newcommand{\g}{\gamma}
\newcommand{\I}{\varphi}
\newcommand{\G}{\Gamma}
\newcommand{\de}{\delta}
\newcommand{\De}{\Delta}
\newcommand{\La}{\Lambda}
\newcommand{\al}{\alpha}
\newcommand{\la}{\lambda}
\newcommand{\f}{\infty}
\newcommand{\vs}{\varepsilon}
\newcommand{\cd}{\cdot}
\newcommand{\si}{\sigma}

\newcommand{\be}{\beta}

\newcommand{\Om}{\Omega}
\newcommand{\om}{\omega}

\newcommand{\inl}{\int_{-\pi}^{\pi}}
\newcommand {\ol} {\overline}

\newcommand{\Var}{\mathrm{Var}}

\newcommand {\s} {\section}
\newcommand {\sn} {\subsection}
\newcommand {\ssn} {\subsubsection}

\theoremstyle{plain}

\newtheorem{thm}{Theorem}[section]
\newtheorem{lem}{Lemma}[section]
\newtheorem{pp}{Proposition}[section]
\newtheorem{cor}{Corollary}[section]
\newtheorem{exa}{Example}[section]
\newtheorem{rem}{Remark}[section]
\newtheorem{den}{Definition}[section]

\usepackage{enumerate}

\numberwithin{equation}{section}

\newcommand{\beq}{\begin{equation}}
\newcommand{\eeq}{\end{equation}}
\newcommand{\bea}{\begin{eqnarray}}
\newcommand{\eea}{\end{eqnarray}}
\newcommand{\beaa}{\begin{eqnarray*}}
\newcommand{\eeaa}{\end{eqnarray*}}

\usepackage[active]{srcltx}

\newcommand{\Ph}{\mathbb{P}}
\newcommand{\Exp}{\mathbb{E}}

\begin{document}

\title{Asymptotic behavior of the variance of BLUE
	for the mean of stationary processes}
\author{Mamikon S. Ginovyan\\ e-mail: ginovyan@bu.edu}

\date{}
\maketitle

\begin{abstract}
\noindent
In this paper, we survey results on the asymptotic behavior of the variance of 
the best linear unbiased estimator (BLUE) for the mean of stationary processes.
This behavior is influenced by the regularity and memory structures of the
observed models. 
The results show that the asymptotic behavior of the variance of the BLUE
is determined solely by the behavior of the spectrum near the origin. 
For nondeterministic models, the variance of the BLUE exhibits 
hyperbolic behavior, similar to the power function, while for purely deterministic models, the variance decreases at an exponential rate. 
Specifically, a necessary condition for the variance of the BLUE to approach zero
exponentially is that the spectral density of the model vanishes on a set of positive
Lebesgue measure in any neighborhood of zero.
We also present results on the asymptotic efficiency of various unbiased linear estimators in comparison to the BLUE.

\end{abstract}


\tableofcontents

\section{Introduction}

\subsection{The problem}

Consider the following stochastic model:
\begin{equation}
	\label{1.0}
	X(t)=m+Y(t), \q t\in \mathbb{Z}:=\{0,\pm 1,\pm 2,\ldots\},
\end{equation}
where $m$ is the constant unknown mean of $X(t)$, and the noise $Y(t)$
is assumed to be a zero-mean, wide-sense stationary process with a spectral
density function $f(\lambda)$, $\lambda \in {\Lambda}:=[-\pi, \pi]$
and a covariance function $r(t)$:
\begin{equation}
	\label{1.1}
	r(t) =  \int_{-\pi}^{\pi}e^{it \lambda}f(\lambda)d\lambda, \q  t\in \mathbb{Z}.
\end{equation}

The problem of interest is estimation of the unknown mean $m$ by unbiased
linear estimators $\widehat{m}_n$, based on a random sample $\{X(t), t=0,1,\ldots,n\}$:
\begin{equation}
	\label{1.2}
	\widehat{m}_n =\sum_{k = 0}^n c_kX(k), \quad \sum_{k = 0}^n c_k=1,
\end{equation}
where the condition $\sum_{k = 0}^n c_k=1$ is needed for unbiasedness of $\widehat{m}_n$.

Of particular interest is the {\it best linear unbiased estimator (BLUE)}, 
denoted by
$\widehat{m}_{BLU}:=\widehat{m}_{n,BLU}$, i.e., the estimator of the form (\ref{1.2}),
where the weights $c_k$, $k=0,1,\ldots,n$, are chosen so that the variance
\begin{equation}
	\label{1.3}
	\Var \left(\widehat{m}_n\right)
	= E|\widehat{m}_n-m|^2 = \sum_{j,k=0}^n c_j\bar{c}_kr(j-k)
\end{equation}
is minimal under the condition $\sum_{k = 0}^n c_k=1$.

\sn{Existence and uniqueness of the BLUE}
We assume that the process $Y(t)$ is non-degenerate, that is,
$$E|Y(t)|^2=r(0) = \int_{-\pi}^{\pi}f(\lambda)d\lambda >0,$$ 
which, in view of (\ref{1.1}),  implies that the covariance matrix
\begin{equation}
	\label{2.1cm}
R_n=\{r(j-k); \,  j, k=1,2,\cdots,n\}
\end{equation} 
of $\{Y(t), t=0,1,\ldots,n\}$
is positive definite. Then, as is well known (see, e.g., Grenander and Szeg\H{o} \cite{GS}, Section 11.1), the BLUE $\widehat{m}_{BLU}$ of the unknown mean $m$ exists and has a unique representation of the form (\ref{1.2}). 
Specifically, the row vector c of coefficients in $\widehat{m}_{BLU}$ is given by 
\begin{equation}
	\label{2.1}
	c=(c_0 \cdots c_n) = (\mathbf{1}^TR_n^{-1}\mathbf{1})^{-1}\mathbf{1}^TR_n^{-1}, \quad \mathbf{1}^T=(1,\ldots, 1),
\end{equation}
and the estimator $\widehat{m}_{BLU}$ is given by 
\begin{equation}
	\label{2.1e}
	\widehat{m}_{BLU} = (\mathbf{1}^TR_n^{-1}\mathbf{1})^{-1}\mathbf{1}^TR_n^{-1}X^T, \quad X^T=(X(0), \cdots, X(n)),
\end{equation}
with
\begin{equation}
	\label{2.2}
	\Var( \widehat{m}_{BLU}) = (\mathbf{1}^TR_n^{-1}\mathbf{1})^{-1} > 0.
\end{equation}

Although formulas \eqref{2.1} -- \eqref{2.2} provide explicit expressions for $c$, $\widehat{m}_{BLU}$,
and $\Var( \widehat{m}_{BLU})$, their practical value is moderate. This is primarily because they involve the inverse of the covariance matrix $R_n$, which may not be known in practice. Moreover, even when it is known, inverting it for high dimensions can lead to cumbersome computational challenges. Therefore, adequate approximations are necessary.

%From (\ref{2.1}) and (\ref{2.2}) we note that $cR_n=\mathbf{1}^T \textrm{Var}(\widehat{m}_{BLU})$. Since $R_n$ is positive definite, this relation together with $cu=1$ uniquely determines the coefficient vector $c$ of $\widehat{m}_{BLU}$, as well as the variance $\Var(\widehat{m}_{BLU})$.

\sn{Efficiency of estimators}
Let $\widehat m$ be an unbiased linear estimator of the unknown mean $m$ of 
a process $X(t)$ possessing a spectral density $f$. 
The efficiency of the estimator $\widehat m$ is defined as
\beq
\label{eff1}
e(n,\widehat{m}, f): = \frac{\Var_f(\widehat{m}_{BLU})}{Var_f(\widehat{m})},
\eeq
where $\widehat{m}_{BLU}$ is the BLUE for the unknown $m$. 
The asymptotic efficiency of $\widehat m$ is defined as
\beq
\label{eff2}
e(\infty, \widehat{m}, f): = \displaystyle{\lim_{n \to \infty}} e(n, f).
\eeq

\begin{den}
An estimator $\widehat m$ of $m$ is said to be efficient (or optimal) if $e(n, \widehat{m}, f)=1$.
It is termed asymptotically efficient if $e(\f, \widehat{m}, f)=1$. 
\end{den}

%Clearly, if the observed process is not Gaussian, BLUE may not be the best estimator, and $e(n, s)$ represents the efficiency only relative to the class of unbiased linear estimators.

In this paper, we provide a survey of both established and recent findings related to the following two problems:
\begin{itemize}
	\item[(a)]  Describe the asymptotic behavior of the variance
	$\Var\left(\widehat{m}_{n,BLU}\right)$ as $n\to\f$ for stationary models with various regularity and memory structures.
	\item[(b)]  Analyze the asymptotic efficiency of various unbiased linear estimators relative to the BLUE.
\end{itemize}
The asymptotic behavior of the variance of BLUE is influenced by the
regularity and memory structure of the observed models. 
The results indicate that this asymptotic behavior 
is determined solely by the behavior of the spectral density near the origin. 
For nondeterministic models, the variance of the BLUE exhibits 
hyperbolic behavior (like a power). In contrast, for purely deterministic models, the variance decreases at an exponential rate. 
Specifically, a necessary condition for the variance of the BLUE to approach zero
exponentially is that the spectral density of the model vanishes on a set of positive
Lebesgue measure in any neighborhood of zero.

Analyzing the asymptotic efficiency of various unbiased linear estimators relative to the BLUE for nondeterministic models, we observe that when the model has short memory
and a piecewise continuous spectral density without discontinuities at $\lambda = 0$, the least squares estimator is asymptotically efficient.
However, if the spectral densities do not meet these conditions, the loss of efficiency can be significant.

\sn{A brief history}
The question of the asymptotic behavior of the variance of BLUE and the asymptotic efficiency
of the least squares estimator $\widehat m_{LS}$, relative to the best linear unbiased
estimator $\widehat m_{BLU}$ for short-memory nondeterministic models, traces back to the papers
by Grenander \cite{Gr-1,Gr-2,Gr-3} (see also Grenander and Rosenblatt \cite{GR1}, Grenander and Rosenblatt \cite{GRo}, Sections 7.0 and 7.1, and Grenander and Szeg\H{o} \cite{GS}, Sections 11.1-11.3). 
Grenander demonstrated that if the model exhibits short memory
with a piecewise continuous spectral density that has no discontinuities at 
$\lambda = 0$, then the least squares estimator is asymptotically efficient.
Since then, various authors have explored these questions for models with different memory structures. Vitale \cite{Vit} and Grenander \cite{Gr-4}
analyzed certain classes of anti-persistent models. 
In particular, it was shown that if the spectral density  $f(\lambda)$ is positive
except for a second-order zero at $\lambda = 0$, then the asymptotic efficiency of 
the least squares estimator is zero.  

The asymptotic behavior of the variance of BLUE and the asymptotic efficiency
of the least squares estimator $\widehat m_{LS}$, relative to the best linear unbiased
estimator $\widehat m_{BLU}$ for long-memory nondeterministic models, have been 
considered in the papers by Adenstedt \cite{Ad}, Adenstedt and Eisenberg \cite{AdE},
Beran \cite{Ber}, Beran and K\"unsch \cite{BK}, Samarov and Taqqu \cite{ST}, and 
Yajima \cite{Ya}.
In these papers, for certain classes of spectral densities satisfying the condition
$f(\la) \sim\lambda^\nu L(\la)$ ($\nu> -1 $) as $\lambda \rightarrow 0$, where $L(\lambda)$ is a slowly varying function at the origin with $0 < L(0) < \infty$,
it was shown that the variance of BLUE  decreases  hyperbolically. Specifically,
$$\Var( \widehat{m}_{BLU}) \sim n^{-\nu -1 } \q {\rm as} \q n\to\f.$$
These classes include  ARFIMA$(p,d,q)$ models with parameter $d < {1}/{2}$ and  fractional Gaussian noise with parameter $0 < H < 1$. 

Estimating the mean for processes characterized by the spectral function, 
applicable to both discrete and continuous time models, has been discussed by Adenstedt and Eisenberg \cite{AdE} and Grenander \cite{Gr-4}.
Parzen \cite{P-4} provides a general Hilbert space approach to linear estimation problems (see also Adenstedt and Eisenberg \cite{AdE}).

The concept of pseudo-best estimators for general regression models
was introduced by Yu. Rozanov (see Ibragimov and Rozanov \cite{IR}, Section 7.3). Pseudo-best estimators of the mean for anti-persistent and long memory deterministic models were discussed in the papers by
Kholevo \cite{Kh}, Rasulov \cite{Ras}, and Rasulov and Kholevo \cite{RKh}.

The asymptotic behavior of the variance of the BLUE for deterministic (singular) models 
has been studied by Babayan and Ginovyan \cite{BG2019a,BG2019}. They established the necessary and sufficient conditions for the variance of the BLUE to decrease exponentially. Specifically, it was demonstrated that a key condition for the variance of the BLUE to approach zero exponentially is that the spectral density of the model must vanish on a set of positive Lebesgue measure in any neighborhood of zero.

\sn{Notation and conventions}
Throughout the paper, we will use the following notation and conventions.\\
The symbol '$:=$' denotes 'by definition'; 
d.t.:= discrete-time; c.t.:= continuous-time;
s.d.:= spectral density; c.f.:= covariance function;
BLUE:= best linear unbiased estimator; LSE:= least squares estimator;
ARMA:= autoregressive moving average;
ARIMA: autoregressive integrated moving average;
ARFIMA:= autoregressive fractionally integrated moving average;
OPUC: = orthogonal polynomials on the unit circle.
WN$(0,1)$: = standard white noise.
By $\mathbb{E}_f[\xi]$ and $\Var_f(\xi)$, we denote the expected value and variance 
of a random variable $\xi$ with respect to the measure
corresponding to the spectral density $f$.

The standard symbols $\mathbb{N}$, $\mathbb{Z}$, $\mathbb{R}$, 
and $\mathbb{C}$ denote the sets of natural numbers, integers, real numbers, and complex numbers, respectively.
We also define the following sets:
$\mathbb{Z}_+: = \{0,1,2,\ldots\}$, $\Lambda: = [-\pi, \pi],$
$\mathbb{D}:=\{z\in\mathbb{C}: |z|<1\}$, and $\mathbb{T}:=\{z\in \mathbb{C}: \, |z|=1\}$.
We denote the weighted Lebesgue space with respect to the measure $\mu$ by $L^p(\mu) := L^p(\mathbb{T}, \mu)$ for $p \geq 1$, and 
we denote the norm in $L^p(\mu)$ by $|\cdot|_{p, \mu}$.
In the special case where $\mu$ is the Lebesgue measure,
we will use the notation $L^p$ and $||\cdot||_{p}$, respectively.

For a set $E$, we denote its closure by $\ol E$.
By $E_f$ we denote the spectrum of the process $X(t)$,
i.e., $E_f:=\{e^{i\la}, \,\, f(\la)>0\}$.
For a fixed complex number $\xi\in\mathbb{C}$, we define $\mathcal{Q}_n(\xi)$ as
the set of polynomials $q_n(z)$, $z\in\mathbb{C}$, of degree at most $n$, that satisfy the condition $q_n(\xi) = 1$. In particular,
$\mathcal{Q}_n(1)=\{q_n(z)=c_0z^n+c_1z^{n-1}+\cdots+c_{n-1}z+c_n:\,\, q_n(1)=\sum_{k=0}^nc_k=1\}.$

For a compact set $F$ in the complex plane $\mathbb{C}$, 
by $T_n(z,F)$ and  $\tau(F)$ we denote the Chebyshev polynomial and the Chebyshev constant of $F$, respectively.
For a function $h\geq0$, by $G(h)$ we denote the geometric mean of $h$.
For nonnegative functions $f(\la)$ and $g(\la)$, 
we write $f(\lambda){\sim}g(\lambda)$ as ${\lambda\to\lambda_0}$
if $\lim_{\lambda\to\lambda_0}{f(\lambda)}/{g(\lambda)}=1$.
For sequences
$\{a_n>0, n\in\mathbb{N}\}$ and $\{b_n>0, n\in\mathbb{N}\}$,
we write $a_n\sim b_n$ if $\lim_{n\to\f}{a_n}/{b_n}=1$, and $a_n\simeq b_n$ if $c_1a_n\leq b_n \leq c_2a_n$, where $c_1, c_2$ are positive constants.

The letters $C$, $c$, $M$ and $m$ with or without 
indices, denote positive constants that may vary from line to line.

All considered functions are assumed to be $2\pi$-periodic
and periodically extended to $\mathbb{R}$.
We assume that all relevant objects are defined in terms
of Lebesgue integrals,
making them invariant under changes to the integrand on a null set.

\sn{The structure of the paper}
The remainder of the paper is structured as follows.
In Section \ref{model}, we specify the model of interest—a second-order stationary 
process—and recall some key concepts and results from the theory of stationary processes.
In Section \ref{SzW} we present the Szeg\H{o} extremum problem
and its solution by utilizing results from the theory of orthogonal
polynomials on the unit circle.
Section \ref{formulas} provides formulas and properties of the variance of BLUE.
In Section \ref{ND-BLU} we present results on asymptotic behavior of the variance of the BLUE for nondeterministic models.
In Section \ref{EFF} we examine efficiency of the ordinary least squares estimator compared with the BLUE.
In Section \ref{Ext}, we explore extensions of the results discussed in the previous sections, focusing on continuous time
models, estimation in a Hilbert space setting, and pseudo-best estimators.
Section \ref{Det} examines the asymptotic behavior of the variance of the BLUE for deterministic (singular) models.
Finally, in Section \ref{Tl}, we present some known results and tools that are used to state and prove the results 
presented in Sections \ref{model} - \ref{Det}.

\section{The model. Key concepts and some basic results}
\label{model}
In this section, we introduce the model of interest: a second-order
stationary process. 
We also revisit some key concepts and results from the
theory of stationary processes.

\sn{Second-order (wide-sense) stationary processes}
Let $\{X(t), \ t\in \mathbb{Z}\}$ be a centered real-valued second-order
stationary process defined on a probability space
$(\Omega, \mathcal{F}, \mathbb{P})$ with covariance function $r(t)$. This means that
\[
{\E}[X(t)]=0, \q r(t)={\E}[X(t+s)X(s)], \q s, t\in\mathbb{Z},
\]
where ${\E}[\cdot]$ denotes for the expectation operator with respect
to the measure $\mathbb{P}$.

By the Herglotz theorem (see, e.g., Brockwell and Davis \cite{BD},
p. 117-118),
there exists a finite measure $\mu$ on $\Lambda$ such that
the covariance function $r(t)$
has the following {\it spectral representation}:
\beq
\label{i1}
r(t)=\inl e^{it\la}d\mu(\la), \q t\in \mathbb{Z}.
\eeq
\n
The measure $\mu$ in (\ref{i1}) is called the {\it spectral measure} of the
process $X(t)$.
The function $F(\la):=\mu\{(-\f,\la]\}$ is referred to as the {\it spectral function} 
of the process $X(t)$.
If $\mu$ is absolutely continuous with respect to the Lebesgue measure,
then the function $f(\la):=d\mu(\la)/d\la$ is called the {\it spectral density} of $X(t)$.
We assume that $X(t)$ is a {\it non-degenerate} process, meaning that
${\rm Var}[X(0)]:={\E}|X(0)|^2=r(0)>0$.
Without loss of generality, we can take $r(0)=1$.
To avoid trivial cases, we also assume that the spectral measure
$\mu$ is {\it non-trivial}, meaning that $\mu$ has infinite support.

Notice that if the spectral density $f(\la)$ exists, then
$f(\la)\geq 0$, $f(\la)\in L^1(\Lambda)$, and equation \eqref{i1} becomes
\begin{equation}
	\label{mo1}
	r(t)=\inl e^{it\la}f(\la)d\la, \q t\in \mathbb{Z}.
\end{equation}
Thus, the covariance function $r(t)$ and the spectral
function $F(\la)$ (or the spectral density function $f(\la)$) 
provide equivalent specifications of the second-order properties for a stationary process
$\{X(t), \ t\in\mathbb{Z}\}$.

\begin{rem}
	{\rm The parametrization of the unit circle $\mathbb{T}$ using the formula $z=e^{i\la}$ establishes a bijection between
	 $\mathbb{T}$ and the interval $[-\pi,\pi)$. 
	 Through this bijection, the measure $\mu$ on $\Lambda$ generates a corresponding measure on the unit circle
	 $\mathbb{T}$, which we also denote by $\mu$. 
	 Therefore, depending on the context, the measure $\mu$ may be supported either on $\Lambda$ or on $\mathbb{T}$, so we will identify integrals over both $\Lambda$ and $\mathbb{T}$.
		
	We utilize the standard Lebesgue decomposition of the measure $\mu$:
		\beq
		\label{di1}
		d\mu(\la)=d\mu_{a}(\la) +d\mu_s(\la) =f(\la)d\la +d\mu_s(\la),
		\eeq
		where $\mu_a$ is the absolutely continuous part of $\mu$
		(with respect to the Lebesgue measure)
		and $\mu_s$ is the singular part of $\mu$, which is the sum
		of the discrete (pure point) and continuous singular components of $\mu$.
		A similar decomposition holds for the spectral function $F(\la)$:
		\beq
		\label{di1SF}
		dF(\la)= f(\la)d\la +dF_s(\la).
		\eeq}
\end{rem}

By the well-known Cram\'er theorem (see, e.g., Cram\'er and Leadbetter \cite{CL}), for any stationary process $\{X(t), \, t\in \mathbb{Z}\}$ with spectral
measure $\mu$ there exists an orthogonal stochastic measure $Z=Z(B)$,
$B\in\mathfrak{B}(\Lambda)$, such that for every $t\in \mathbb{Z}$
the process $X(t)$ has the following {\sl spectral representation}:
\begin{equation}
	\label{i18}
	X(t)=\int_\Lambda e^{it\la}dZ(\la), \q  t\in\mathbb{Z}.
\end{equation}
Moreover, ${\E}\left[|Z(B)|^2\right]=\mu(B)$ for every $B\in\mathfrak{B}(\Lambda)$.
Here, $\mathfrak{B}(\Lambda)$ denotes the Borel $\si$-algebra of the sets of $\Lambda$.
For the definition and properties of orthogonal stochastic measures and the stochastic integral in (\ref{i18}), we refer to Cram\'er and Leadbetter \cite{CL}, Ginovyan \cite{G2025}, and Ibragimov and Linnik \cite{IL}.

\sn{Linear processes. Existence of spectral density functions}

We will consider stationary processes that possess spectral density
functions. For the following results, we refer to
Cram\'er and Leadbetter \cite{CL}, Doob \cite{Doob}, and Ibragimov and Linnik \cite{IL}.
\begin{thm} The following assertions hold.
	\begin{itemize}
		\item[(a)]
		The spectral function $F(\la)$ of a stationary process $\{X(t), \,t \in \mathbb{Z}\}$
		is absolutely continuous (with respect to the Lebesgue measure), meaning that $F(\la)=\int_{-\pi}^\la f(x)dx$,
		if and only if it can be represented as an infinite moving average:
		\begin{equation}
			\label{dlp}
			X(t) = \sum_{k=-\f}^{\f}a(t-k)\vs(k), \quad
			\sum_{k=-\f}^{\f}|a(k)|^2 < \f,
		\end{equation}
		where $\{\vs(k), k\in \mathbb{Z}\}\sim$ WN(0,1) is standard white noise, which consists of a sequence of orthonormal random variables.
		\item[(b)]
		The covariance function $r(t)$ and the spectral density $f(\la)$ of $X(t)$ are defined by the following formulas:
		\begin{equation}
			\label{dcv}
			r(t)= {\E} X(t)X(0)=\sum_{k=-\f}^{\f}a(t+k) a(k),
		\end{equation}
		and
		\beq
		\label{dsd}
		f(\la) =
		\frac{1}{2\pi}\left|\sum_{k=-\f}^{\f}a(k)e^{-ik\la}\right|^2
		= \frac{1}{2\pi} |\widehat a(\la)|^2,
		\quad \la\in\Lambda.
		\eeq
		\item[(c)]
		In the case where the innovation $\{\vs(k), k\in \mathbb{Z}\}$ is a sequence of Gaussian random variables,
		the process $\{X(t), \,t \in \mathbb{Z}\}$ is also Gaussian.
	\end{itemize}
\end{thm}

\sn{Dependence (memory) structure of the model}
\label{mem}

Depending on the memory (dependence) structure, we will distinguish
the following types of stationary models:

(a) short memory (or short-range dependent),

(b) long memory (or long-range dependent),

(c) intermediate memory (or anti-persistent).

\noindent
The memory structure of a stationary process is essentially a
measure of the dependence among all the variables in the process,
considering the effects of all correlations simultaneously.
Traditionally, memory structure has been defined in the time domain
in terms of the decay rates of autocorrelations or, in the
frequency domain, in terms of the rates of explosion of low-frequency spectra
(see, e.g., Beran et al. \cite{BFGK}, Ginovyan \cite{G2025}, and references therein).

Note that the definitions of memory in both the time and frequency domains are generally not equivalent (see Proposition \ref{memp1} for details).
	
It is convenient to characterize the memory structure in terms
of the spectral density function.

\ssn{Short memory models}
A stationary process $\{X(t), \,t \in \mathbb{Z}\}$ with a spectral
density function $f(\la)$ is referred to as a \textit{short memory} process if the spectral density $f(\lambda)$ is bounded away from both zero and infinity. Specifically, there exist constants  $C_1$ and $C_2$ such that
\begin{equation*}
	0< C_1 \le f(\la) \le C_2 <\f.
\end{equation*}

A typical example of a short memory model is the stationary
Autoregressive Moving Average (ARMA)$(p,q)$ process $X(t)$, 
defined as a stationary solution to the following difference equation:
\begin{equation*}
	\psi_p(B)X(t)=\theta_q(B)\varepsilon(t), \q t\in\mathbb{Z},
\end{equation*}
where $\psi_p$ and $\theta_q$ are polynomials of degrees $p$ and $q$, respectively, with no zeros on the unit circle $\mathbb{T}$,
$B$ is the backshift operator defined by $BX(t)=X(t-1)$,
and  $\{\vs(t), t\in \mathbb{Z}\}\sim$ WN(0,1) is standard white noise.
The spectral density $f(\la)$ is a rational function
(see, e.g., Brockwell and Davis \cite{BD}, Section 3.1):
\begin{equation}
	\label{arma}
	f(\la) = \frac{1}{2\pi}\cd\frac
	{|\theta_q(e^{i\la})|^2}{|\psi_p(e^{i\la})|^2}.
\end{equation}

\ssn{Long-memory and anti-persistent models}

A {\it long-memory} model is defined as a
stationary process with {\it unbounded} spectral density,
while an {\it anti-persistent} model is a stationary
process with {\it vanishing} spectral density at some fixed 
points (see, e.g., Beran et al. \cite{BFGK}, Brockwell and Davis \cite{BD}, Ginovyan \cite{G2025}, and references therein).

A typical model that exhibits long memory and intermediate
memory (i.e., is anti-persistent) is the Autoregressive Fractionally Integrated Moving Average (ARFIMA)$(p,d,q)$ process $X(t)$.
This process is defined as a stationary solution to the following
difference equation (see, e.g., Brockwell and Davis \cite{BD}, Section 13.2):
\begin{equation*}
	\psi_p(B)(1-B)^dX(t)=\theta_q(B)\varepsilon(t), \q d<1/2,
\end{equation*}
where $B$ is the backshift operator, $\varepsilon(t)$ 
is standard white noise, and $\psi_p$ and $\theta_q$ are
polynomials of
degrees $p$ and $q$, respectively.
The spectral density $f_X(\la)$ of $X(t)$ is given by
\begin{equation}
	\label{AA}
	f_X(\la)=|1-e^{-i\la}|^{-2d}f(\la)=(2\sin(\la/2))^{-2d}f_{ARMA}(\la), \q d<1/2,
\end{equation}
where $f_{ARMA}(\la)$ is the spectral density of an ARMA$(p,q)$ process, as given by equation (\ref{arma}). The condition $d<1/2$ ensures that
$\int_{-\pi}^\pi f(\la)d\la<\f$,
implying that the process $X(t)$ is well-defined because
${\E}|X(t)|^2=\int_{-\pi}^\pi f(\la)d\la.$
By \eqref{AA}, we have
\begin{equation}
	\label{MT0L}
	f(\la)\thicksim (2\pi)^{-1} |\la|^{-2d}\q {\rm as}\q \la\to0.
\end{equation}

Observe that for $ 0<d<1/2$ the model $X(t)$ specified by the spectral density (\ref{AA}) exhibits long memory, indicating that $f(\la)$ blows up at $\la=0$ like a power function, which is characteristic of long memory models.  
For $d<0$, the model $X(t)$ exhibits intermediate memory, where the spectral density in \eqref{AA}
vanishes at $\la=0$. For $d=0$, the model $X(t)$ displays short memory.
For $d \ge1/2$, the function $f_X(\la)$ in (\ref{AA}) is not integrable,
and therefore cannot represent the spectral density of a stationary process.

Another important discrete time long memory model is the fractional
Gaussian noise (fGn),  characterized by the {\it Hurst index} $H$, where $0<H\leq1$.
The fGn is a discrete time centered Gaussian stationary process $X(k)$, $k \in \mathbb{Z}$, and it has a
spectral density function:
\begin{equation}
	\label{MT2}
	f(\la)=c\, |1-e^{-i\la}|^{2}\sum_{k=-\f}^\f|\la+2\pi k|^{-(2H+1)},
	\q-\pi\le\la\le\pi,
\end{equation}
where $c$ is a positive constant.

It follows from (\ref{MT2}) that
\begin{equation}
	\label{MT2L}
	f(\la)\thicksim  |\la|^{1-2H}\q {\rm as}\q \la\to0,
\end{equation}
which indicates that $f(\la)$ blows up if $H>1/2$ and approaches zero if $H<1/2$.
Furthermore, by comparing (\ref{MT0L}) and (\ref{MT2L}), we observe that
the spectral density of fGn exhibits the same
behavior at the origin as ARFIMA$(0,d,0)$, with the memory parameter $d=H-1/2.$

Thus, the fGn $\{X(k), k\in\mathbb{Z}\}$  exhibits long memory when $1/2 < H < 1$ and is anti-persistent when $0 < H < 1/2$. The variables $X(k)$, $k \in \mathbb{Z}$, are independent when $H = 1/2$. For more details, we refer to Ginovyan \cite{G2025}, and Samorodnisky and Taqqu \cite{ST}.

Another notable example of a process that exhibits long memory or anti-persistence is the Fisher-Hartwig model, also referred to as the Jacobian model.
We define a Fisher--Hartwig model as a stationary process $X(t)$ whose spectral density $f(\la)$ has the following form:
\beq
\label{FH1}
f(\la) = f_1(\la)\prod_{k=1}^m|e^{i\la}-
e^{i\la_k}|^{2\al_k}, 	\q-\pi\le\la\le\pi,
\eeq
where $f_1(\la)$ is the spectral density of a short memory process,
the points $\la_k\in [-\pi, \pi]$ are distinct, and the exponents $\al_k>-1/2$, $k=1,\ldots,m$. 
%The numbers $\al_k$ are the Fisher--Hartwig exponents.

Note that for $k=1, \al_1=-d$, and $f_1(\la)=f_{ARMA}(\la)$,
the Fisher--Hartwig model becomes ARFIMA$(p,d,q)$ model.

\sn {Deterministic and nondeterministic processes}
\label{dnd}
In this section, we present Kolmogorov's isometric isomorphism theorem and provide frequency-domain characterizations of deterministic (singular), nondeterministic, and purely nondeterministic (regular) processes.

\ssn{Kolmogorov's isometric isomorphism theorem}
Given a probability space $(\Om, \mathcal{F}, \mathbb{P})$, we define the $L^2$-space 
$L^2(\mathbb{P}):=L^2(\mathbb{P},\Lambda)$ of real-valued random variables $\xi= \xi(\om)$ with
$\E[\xi]=0$:
\begin{equation}
	\label{i19}
	L^2(\mathbb{P}):=\left\{\xi: \, ||\xi||^2:=
	\int_\Om |\xi(\om)|^2dP(\om)<\f\right\}.
\end{equation}
Then $L^2(\mathbb{P})$ becomes a Hilbert space with the following
inner product: for $\xi, \eta\in L^2(\mathbb{P})$,
\begin{equation}
	\label{i20}
	(\xi,\eta)= \E[\xi \eta]=\int_\Om \xi(\om){\eta(\om)}dP(\om).
\end{equation}

For $a, b\in \mathbb{Z}$,  $-\f\le a\le b\le \f,$
we define the space $H_a^b(X)$ as the closed linear
subspace of the space $L^2(\mathbb{P})$ spanned by the random
variables $X(t)=X(t,\om)$, $t\in [a,b]$, $\om\in\Omega$:
\begin{equation}
	\label{i21}
	H_{a}^b(X): = \ol{sp}\{X(t), \,\, a \le t\le b\}_{L^2(\mathbb{P})}.
\end{equation}

Note that the space $H_{a}^b(X)$ consists of
all finite linear combinations of the form $\sum_{k=a}^b c_k X(k)$,
as well as, their $L^2(\mathbb{P})$-limits.

The space $H(X):=H_{-\f}^\f(X)$ is referred to as the \textit{Hilbert space generated by the process} $X(t)$, or the {\it time-domain} of $X(t)$.

Let $\mu$ be the spectral measure of the process $\{X(t), \, t\in \mathbb{Z}\}$.
Consider the weighted $L^2$-space $L^2(\mu):=L^2(\mu,\Lambda)$ of complex-valued functions
$\I(\la), \, \la\in\Lambda$, defined by
\beq
\label{i22}
L^2(\mu):=\left\{\I(\la): \,||\I||^2_\mu:= \int_\Lambda|\I(\la)|^2d\mu(\la)<\f\right\}.
\eeq
Then $L^2(\mu)$ becomes a Hilbert space with the following
inner product: for $\I, \psi\in L^2(\mu)$,
\begin{equation}
	\label{i23}
	(\I, \psi)_\mu= \int_\Lambda\I(\la)\ol{\psi}(\la)d\mu(\la).
\end{equation}

The Hilbert space $L^2(\mu,\Lambda)$ is referred to as the \textit{frequency-domain} of the process $X(t)$.

\begin{thm}[Kolmogorov's isometric isomorphism theorem]
	\label{Kolm}
	For any stationary process $X(t)$, $t\in \mathbb{Z}$, with spectral
	measure $\mu$, there exists a unique isometric isomorphism $V$ between
	the time-domain $H(X)$ and the  frequency-domain $L^2(\mu)$
	such that $V[X(t)]=e^{it\la}$ for any $t\in \mathbb{Z}.$
\end{thm}
According to Theorem \ref{Kolm}, any linear problem in the time-domain $H(X)$ can be translated into a problem in the frequency-domain $L^2(\mu)$, and vice versa. This fact allows for the study of stationary processes using analytic methods.

For more information about Hilbert spaces related to stationary processes, see Section \ref{HSR}.

\ssn {Deterministic (singular), nondeterministic and purely nondeterministic (regular) processes. Characterizations}

We start by stating the linear prediction problem. 
Let $X(t),$ $t\in\mathbb{Z}: = \{0,\pm1,\ldots\}$, be a centered discrete-time
second-order stationary process with spectral measure $\mu$.
Suppose we observe a finite realization of the process $X(t)$:
$\{X(t), \,\, -n\le t\le-1\},$ $n\in\mathbb{N}: = \{1,2, \ldots\}.$
We aim to make a one-step ahead prediction, specifically to predict the
unobserved random variable $X(0)$ using the {\it linear predictor}
defined as $Y=\sum_{k=1}^{n}c_kX(-k).$

The coefficients
$c_k$, $k=1,2,\ldots,n$, are chosen so as to minimize
{\it the mean-squared error}: $\E\left|X(0) - Y\right|^2,$
where $\E[\cd]$ denotes the expectation operator.
If such minimizing constants $\widehat c_k:=\widehat c_{k,n}$
can be found, then the random variable
$$\widehat X_n(0):=\sum_{k=1}^{n}\widehat c_kX(-k)$$
is referred to as  {\it the best linear one-step ahead predictor} of $X(0)$
based on the observed finite past: $X(-n), \ldots, X(-1)$.
The minimum mean-squared error:
\beq
\label{si-pred}
\si_{n}^2(0,\mu): =\E\left|X(0) - \widehat X_n(0)\right|^2\geq0
\eeq
is called {\it the best linear one-step ahead prediction error}.

Observe that $\sigma_{n+1}^2(0,\mu)\leq \sigma_n^2(0,\mu)$, $n\in\mathbb{N}$,
and hence the limit of $\sigma_n^2(0,\mu)$ as $n\to\f$ exists.
We denote $\sigma^2(0,\mu): = \sigma_\infty^2(0,\mu)$ 
as the prediction error based on the entire infinite past: $\{X(t)$, $t\le-1\}$.

From a prediction perspective, it is natural to distinguish the class of
processes for which we have {\it error-free prediction} based on the entire infinite past, i.e.,
$\si^2(0,\mu)=0$. Such processes are termed {\it deterministic} or {\it singular}.
Processes for which $\si^2(0,\mu)>0$ are referred to as {\it nondeterministic}.

Observe that the time-domain $H(X)$ of any non-degenerate
stationary process $\{X(t)$, $t\in\mathbb{Z}\}$ can be
represented as the orthogonal sum: 
$$H(X)= H_1(X)\oplus H_{-\f}(X),$$
where $H_{-\f}(X)$ is the remote past of $X(t)$, defined by $H_{-\f}(X): = \cap _t H_{-\f}^t (X)$,
and $H_1(X)$ is the orthogonal complement of $H_{-\f}(X)$.

Thus, we can provide the following geometric definitions for deterministic (singular), nondeterministic, and purely nondeterministic (regular) processes.
\begin{den}A stationary process $\{X(t)$, $t\in\mathbb{Z}\}$ is called
	\begin{itemize}
		\item
		deterministic or singular if $H_{-\f}(X)= H(X)$, that is,
			$H_{-\f}^t(X)=H_{-\f}^s(X)$ for all $t, s \in\mathbb{Z}$;
		\item
		nondeterministic if the remote past $H_{-\f}(X)$ is a proper subspace of $H(X)$, that is, $H_{-\f}(X)\subset H(X)$;
		\item
		purely nondeterministic (PND) or regular if $H_{-\f}(X)=\{0\}$,
		that is, the remote past $H_{-\f}(X)$ is the trivial subspace consisting of the singleton zero.
	\end{itemize}
\end{den}
\begin{rem}{\rm Every purely nondeterministic process $X(t)$
		is nondeterministic, but the converse is generally not true.
		An example of such a process is given by $X(t)=\vs(t)+\xi$,
		where $\{\vs(t), \,t\in \mathbb{Z}\}$ is a ${\rm WN}(0,\si^2)$
		(a centered white noise with variance $\si^2$), and
		$\xi$ is a random variable such that $Var(\xi)=\si^2_0$ and
		$(\vs(t), \xi)=0$ for all $t\in \mathbb{Z}$
		(see Pourahmadi \cite{Po}, p. 163).
	Note that the only process $X(t)$ that is both
		deterministic and purely nondeterministic is the
		degenerate process. Assuming that $X(t)$ is a non-degenerate
		process, we exclude this trivial case.}
\end{rem}

The following result provides spectral characterizations of deterministic, nondeterministic, and purely nondeterministic
processes (see, e.g., Grenander and Szeg\H{o} \cite{GS}, 
p. 44, and Ibragimov and Rozanov \cite{IR}, pp. 35--36).
%\cite{R}, p. 58, 64,).
\begin{thm}
	\label{th34}
	Let $X(t)$ be a non-degenerate stationary process with spectral
	measure $\mu$ of the form \eqref{di1}. The following assertions hold.
	\begin{itemize}
		\item[(a)]{\rm (Kolmogorov-Szeg\H{o}'s Theorem)}.
		The following relations hold.
		\bea
		\label{c013}
		\lim_{n\to\f}\si_n^2(0,\mu)=\lim_{n\to\f}\si_n^2(0,f)=\si^2(0,f)=2\pi G(f),
		\eea
		where $G(f):=G(0,f)$ is the {\it geometric mean} of $f$, namely
		\beq
		\label{a2}
		G(f): = \left \{
		\begin{array}{ll}
			\exp\left\{\frac1{2\pi}\inl\ln f(\la)\,d\la \right\} &
			\mbox{if \, $\ln f \in {L}^1(\mathbb{T})$}\\
			0, & \mbox { otherwise,} \q
		\end{array}
		\right.
		\eeq
		\item[(b)]
		$
		H_{-\f}^0(\mu_s)=H(\mu_s)
		\Leftrightarrow  \sigma^2(0,\mu)=0 \Leftrightarrow
		X(t) \,\,  is \,\, deterministic,
		$
		\item[(c)]
		{\rm (Kolmogorov-Szeg\H{o}'s alternative)}.
		Either
		\begin{align*}
		H_{-\f}^0(\mu_a)=H(\mu_a)& \Leftrightarrow \inl\ln f(\la)\,d\la = -\f
		\Leftrightarrow  \sigma^2(0,f)=0 \\
		&\Leftrightarrow
		X(t) \,\,  is \,\, deterministic,
		\end{align*}
		or else
		\begin{align*}
		H_{-\f}^0(\mu_a) \neq  H(\mu_a)& \Leftrightarrow\inl\ln f(\la)\,d\la > -\f \Leftrightarrow
		\sigma^2(0,f) >0 \\
		&\Leftrightarrow X(t) \,\, is \,\, nondeterministic.
	\end{align*}
		\item[(d)] The process $X(t)$ is regular (PND) if and only if
		it is nondeterministic and $\mu_s\equiv0$.
	\end{itemize}
\end{thm}

\begin{rem}
	\label{r33}
	{\rm The second equality in (\ref{c013}) was proved by Szeg\H{o} \cite{Sz1921}, while the first equality was established by Kolmogorov \cite{Kol1, Kol2} 
		(see also Hoffman \cite{Hof}, p. 49).
	It is remarkable that the relation in \eqref{c013} is independent of the singular part $\mu_s$.}
\end{rem}

\ssn{Non-determinism. The Szeg\H{o} Condition}

\begin{den}[Szeg\H{o}'s Condition]
The spectral measure $\mu$ of the form \eqref{di1} is said to satisfy the Szeg\H{o} condition if
\beq
\label{S}
\inl\ln f(\la)\,d\la > -\f
\eeq
where $f$ is the density of the absolutely continuous part of $\mu$. By convention, if $f(\lambda)=0$ on a set of positive Lebesgue measure, the integral in \eqref{S} is considered to equal $-\infty$.
\end{den}

Note that the condition $\ln f \in {L}^1(\Lambda)$ in \eqref{a2} is equivalent to the Szeg\H{o} condition \eqref{S}.
This equivalence follows because $-\f\leq\ln f(\la)\le f(\la)$.
The Szeg\H{o} condition \eqref{S} is also referred to as the {\it non-determinism condition}.

If the spectral measure $\mu$ of the process $X(t)$ is absolutely continuous with spectral density $f(\la)$, the singularity and regularity of $X(t)$ depend on whether $f(\la)$ approaches zero. For example, a short memory process $X(t)$ is regular. 
If $f(\la)$ vanishes on some interval
$I\subset [-\pi,\pi]$, then the Szeg\H{o} condition \eqref{S} is violated, resulting in $X(t)$ being singular. In this case, we will describe the process as purely deterministic.
Singularity can also occur if $f(\la)$ vanishes at a single point; the spectral density can produce either nondeterministic or deterministic behavior depending on the flatness at zero. 
Furthermore, note that the Szeg\H{o} condition relates 
to the nature of the singularities (zeros and poles) of the spectral density $f$, and does not depend on the differential properties of $f$.
The following example illustrates this point (see also Pourahmadi \cite{Po}, p. 68).
\begin{exa}[Flat-zero spectral model]
\label{Ex-FZM}
{\rm Consider the stationary process $X(t)$ with spectral density
\beq
\label{FZM}
f_a(\lambda): = \left \{
\begin{array}{ll}
	\exp\left\{-|\lambda|^{-a}\right\}&
	\mbox{if \, $\lambda\neq0$}\\
	0 & \mbox{if \, $\lambda=0$,}
\end{array} \quad a>0.
\right.
\eeq
We first show that the function $f_a(\la)$ is infinitely differentiable ($f_a(\lambda) \in C^\infty([-\pi,\pi])$)
with all derivatives vanishing at $\lambda=0$:
\[
f_a^{(k)}(0) = 0, \quad \forall k\ge 0.
\]
This follows using the formula for derivative of the function $|\lambda|$:
\begin{equation}
	\label{da2}
	|\lambda|' = \left(\sqrt{\lambda^2}\right)'= \frac{\lambda}{|\lambda|}=
	\left \{
	\begin{array}{lll}
		-1 & \mbox{if \, $\lambda<0$}\\
		0 & \mbox{is not defined if \, $\lambda=0$}\\
		1 & \mbox{if \, $\lambda>0$},
	\end{array}
	\right.
\end{equation}
and the following relation:
\begin{equation}
	\label{da3}
	\lim_{t\to\infty}t^be^{-ct}=0 \quad (b\in \mathbb{R},  \,\, c>0).
\end{equation}
Using \eqref{da2} and \eqref{da3}, a direct computation shows that
\begin{equation}
	\nonumber
	f_a^{(k)}(\lambda) = \left \{
	\begin{array}{ll}
	P_k\big(|\lambda|^{-a}\big)\,|\lambda|^{-k}\,e^{-|\lambda|^{-a}},
		&\mbox{if \, $\lambda\neq0$},\\
		0, & \mbox{if \, $\lambda=0$,}
	\end{array}
	\right.
\end{equation}
where $P_k$ is a polynomial depending on $a$ and $k$.
In addition, we have 
	\[
\lim_{\lambda\to 0} f_a^{(k)}(\lambda) = 0.
\]
Thus, $f_a\in C^\infty([-\pi,\pi])$.
Next, the Szeg\H{o} integral is:
\[
\int_{-\pi}^{\pi} \log f_a(\lambda) \, d\lambda = -2 \int_0^\pi \lambda^{-a} \, d\lambda.
\]
\begin{itemize}
	\item For $0<a<1$, $\int_0^\pi \lambda^{-a} d\lambda < \infty$, so Szeg\H{o}'s condition \eqref{S} is satisfied. The asymptotic one-step prediction error $\si_n^2(0,f)>0$, indicating that the process $X(t)$ is \emph{purely nondeterministic}.
	\item For $a \ge 1$, $\int_0^\pi \lambda^{-a} d\lambda = \infty$, which means Szeg\H{o}'s condition \eqref{S} fails, $\si_n^2(0,f) = 0$, and the process $X(t)$ is \emph{deterministic}.
\end{itemize}}
\end{exa}
\begin{rem}
	\label{PSM}
	{\rm In the case where the measure \( \mu \) is absolutely continuous, represented as \( d\mu = f(\lambda) d\lambda \) (with \( \mu_s = 0 \)), we can differentiate between light and strong determinism. We will refer to the process \( X(t) \) as light deterministic if the spectral density  $f(\la)$ is positive everywhere except for a finite or countable number of points, which results in a violation of the Szeg\H{o} condition \eqref{S}. 
	In contrast, if $f(\la)$ vanishes over a set of positive Lebesgue measure, the corresponding process \( X(t) \) is termed strong deterministic (or purely deterministic). An example of a light deterministic process is the flat-zero spectral model discussed in Example \ref{Ex-FZM} for $a\geq1$. Another important example 
	is the process characterized by the Pollaczek-Szeg\H{o} spectral density:
	\beq
	\label{nd4}
	f_a(\la):=\frac{e^{(2\la-\pi)\varphi(\la)}}{\cosh\left(\pi\varphi(\la)\right)},
	%\ |\sin(\la)|,
	\q f_a(-\la)=f_a(\la),\q 0\leq\la\leq\pi,
	\eeq
	where $\varphi(\la)=(a/2)\cot\la$ and $a$ is a positive parameter. 
	Observe that for the function $f_a$ in \eqref{nd4}, we have the following asymptotic
	relation (for details, see Babayan and Ginovyan \cite{BG2023},
	and Szeg\H{o} \cite{S1}:
	\beq \label{tt2}
	f_a(\la)\sim
	\left \{
	\begin{array}{ll}
		2e^a\exp\left\{-{a\pi}/{|\la|}\right\} & \mbox{as $\la\to0$},\\
		2\exp\left\{-{a\pi}/{(\pi-|\la|)}\right\} & \mbox{as $\la\to\pm\pi$}.
	\end{array}
	\right.
	\eeq
Thus, the function $f_a$ in \eqref{nd4} has a very high order of contact
with zero at the points $\la=0,\pm\pi$, due to which the Szeg\H{o} condition \eqref{S}
is violated, indicating that the process with spectral
density $f_a$ is light deterministic.}		
\end{rem}
\begin{rem}[Pure point spectrum]
	\label{Ppoint}
	{\rm The spectral measure $\mu$ is \emph{pure point} if
		\[
		\mu = \sum_{k=1}^{\infty} w_k\,\delta_{\lambda_k},
		\quad
		w_k>0,\quad \sum_{k=1}^{\infty} w_k<\infty.
		\]
From the spectral representation theorem applied to atomic
measures we have (see \eqref{i18}):
\beq
\label{Ppoint1}
X(t) = \sum_{k=1}^{\infty} Z_k e^{i t\lambda_k},
\eeq
where $\{Z_k\}$ are uncorrelated random variables with
$\mathbb{E}|Z_k|^2 = w_k$.
Since there is no absolutely continuous part, Kolmogorov's criterion gives
$\si_n^2(0,\mu)=0$, implying that if the spectral measure is purely atomic, then the process $X(t)$ is deterministic.} %Alternatively, the representation \eqref{Ppoint1}
%shows that the infinite past $\{X(-t), \, t\ge1\}$ determines all coefficients $\{Z_k\}$, and hence determines $X(0)$ exactly.}		
\end{rem}
\begin{rem}[Pure singular continuous spectrum]
\label{PSC}
{\rm The spectral measure $\mu$ is \emph{pure singular continuous} if it is continuous, assigns zero mass to points, and is singular with respect to the Lebesgue measure. If $\mu$ is purely singular continuous, then the process $X(t)$ is deterministic.}		
	\end{rem}
\begin{rem}
	\label{NSR}
	{\rm If the Szeg\H{o} condition \eqref{S} is satisfied, but the spectral measure $\mu$ is either discontinuous or continuous with a non-vanishing singular part ($\mu_s \neq 0$), then the process $X(t)$ is neither singular nor regular.}		
\end{rem}

\s{Szeg\H{o}'s extremum problem on the unit circle}
\label{SzW}

\sn{Setting of the problem} 
\label{SzPr}

We denote by $\mathcal{Q}_n$ the set of monic polynomials $q_n(z)=c_0z^n+c_1z^{n-1}+\cdots+c_{n-1}z+c_n$, $z\in\mathbb{C}$ of degree
$n\in\mathbb{N}$ with leading coefficient $c_0 = 1$.
For a fixed $\xi\in\mathbb{C}$, we denote by $\mathcal{Q}_n(\xi)$ the set of polynomials $q_n(z)$, $z\in\mathbb{C}$, that have a degree of at most $n $ and satisfy the condition $q_n(\xi) = 1$. Specifically,
\beq
\label{L5Q}
\mathcal{Q}_n(\xi):=\{q_n(z)=c_0z^n+\cdots+c_{n-1}z+c_n;
\,\, {\rm deg}\ q_n\leq n, \, q_n(\xi) = 1\}.
\eeq

Let $\mu$ be a nontrivial measure on the unit circle $\mathbb{T}$ with the standard Lebesgue decomposition \eqref{di1}.
Define the $n^{th}$ {\it Christoffel function} associated with a measure $\mu$ by
\beq
\label{L5s}
\la_n(\xi,\mu):=\min_{q_n(z)\in \mathcal{Q}_n(\xi)}
||q_n||^2_\mu,
\eeq
where $||q_n||_\mu$ denotes the norm of the polynomial $q_n$ in the space $L^2(\mu)$:
\beq
\label{L5n}
||q_n||^2_\mu:=\int_{-\pi}^{\pi}|q_n(e^{i\lambda})|^2 d\mu(\lambda).
\eeq
Evidently the sequence $\{\la_n(\xi,\mu), \, n\in\mathbb{N}\}$ is non-increasing. 

The generalized Szeg\H{o} extremum problem can be stated as follows: 
\begin{itemize}
	\item[(a)] Find the polynomial $q_n(e^{i\la})$ of degree at most $n$
	in $e^{i\la}=z$ that solves the minimum problem in \eqref{L5s}.
	\item[(b)] Describe the asymptotic behavior of $\la_n(\xi,\mu)$ as $n\to\f$.
\end{itemize}
\begin{den}
	\label{def3.1}
	The polynomial $p_n(z):=p_n(z,\mu)$ that solves the minimum problem in \eqref{L5s}
	is referred to as the optimal polynomial (of degree $n$) for the measure $\mu$. Therefore, 
		\begin{equation}
		\label{Szop}
		\la_n(\xi,\mu)=\min_{q_n(z)\in \mathcal{Q}_n(\xi)}
		||q_n||^2_\mu=||p_n||^2_\mu, \q \xi\in\mathbb{C}.
	\end{equation}
\end{den}

The extremum problem mentioned above was solved by G. Szeg\H{o}, who demonstrated that the optimal polynomial 
$p_n(z)$ exists, is unique, and can be expressed in terms of the
{\it orthogonal polynomials on the unit circle $\mathbb{T}$} associated with the measure $\mu$.
He also described the asymptotic behavior of the function $\la_n(\xi,\mu)$ as $n\to\f$ for some classes.

To present Szeg\H{o}'s solution,
we will first review some key concepts and results from the theory of orthogonal polynomials on the unit circle  $\mathbb{T}$.

\sn{Orthogonal polynomials on the unit circle (OPUC) and Szeg\H{o}'s generalized theorem} 

Let $\mu$ be a nontrivial measure on the unit circle $\mathbb{T}$ with the standard Lebesgue decomposition \eqref{di1}.
The system of orthogonal polynomials on the unit circle $\mathbb{T}$ associated with the measure $\mu$:
%$\I_n(z): = \I_n(z;\mu)$ ($z=e^{i\la}, \, n\in \mathbb{Z}_+$)
$$\{\I_n(z) = \I_n(z;\mu), \q z=e^{i\la}, \q n\in \mathbb{Z}_+\}$$
is uniquely determined by the following two conditions:
\begin{itemize}
	\item[(i)] \, $\I_n(z) = \kappa_nz^n + \cdots +l_n$
	is a polynomial of degree $n$, in which the coefficient
	$\kappa_n$ is positive;
	\item[(ii)] $(\I_k,\I_j)_\mu = \de_{kj}$ for arbitrary $k,j\in \mathbb{Z}_+$, where
	$\de_{kj}$ is the Kronecker delta.
\end{itemize}
Define the monic ($\Phi_n(z)$) and the reversed ($\Phi^*_n(z)$) polynomials:
%(see, e.g., Simon \cite{Sim-1}, p. 2):
\bea
\label{oppp}
&&\Phi_n(z):=\Phi_n(z,\mu)=\kappa^{-1}_n\I_n(z)=z^n + \cdots +l_n\kappa_n^{-1}, \\
\label{opp*}
&&\Phi^*_n(z):= \Phi^*_n(z,\mu)=z^{n}\ol {p_n(1/\ol z)}=\ol l_n\kappa_n^{-1}z^n + \cdots +1.
\eea
Then, we have
\beq
\label{opp}
||\Phi_n||_{\mu}^2=||\Phi^*_n||_{\mu}^2=\kappa_n^{-2}= \si_{n}^2(0,\mu),
\eeq
where \(\si_{n}^2(0,\mu)\) represents the prediction error, as defined in equation \eqref{si-pred}.

For a nontrivial measure $\mu$ on the unit circle $\mathbb{T}$ 
of the form \eqref{di1} and for a fixed $\xi\in\mathbb{C}$, we consider the {\sl Szeg\H{o} kernel} 
$S_n(z,\xi):= S_n(\xi,z,\mu)$ defined by (see, e.g., Grenander and Szeg\H{o} \cite{GS}, p. 39):
\begin{equation}
	\label{op8}
	S_n(\xi,z) =\sum_{k=0}^{n}\ol{\I_k(\xi)}\I_k(z).
\end{equation}

Observe that $S_n(z,\xi)$ is the reproducing kernel of the space $H_n(\mu)$, which consists of polynomials of order $n$
viewed as a subspace of the space $L^2(\mu)$. 
Specifically, the kernel $S_n(z,\xi)$ satisfies the following conditions:
\begin{itemize}
	\item[(a)] \,
	 For every $\xi\in\mathbb{C}$, $S_n(\cd,\xi)$ is a polynomials of order $n$.
	\item[(b)] \,
	The reproducing property holds: for every $\xi\in\mathbb{C}$ and every
	$h\in H_n(\mu)$,
	\beq\label{rk01}
	(h(\cdot),S_n(\cd,\xi))_\mu=h(\xi).
	\eeq
\end{itemize}

For the definition and properties of reproducing kernel Hilbert spaces, see Section \ref{HSR}.

\begin{exa}
{\rm Let $d\mu=d\la$. Then $\I_n(z)=z^n$. In this case, we have %$S_n(e^{i\la},e^{i\theta})$ becomes the Dirichlet kernel, and
$$|S_n(e^{i\la},e^{it)}|^2
=\frac{\sin^2((n+1)(\la-t)/2)}{\sin^2((\la-t)/2)},$$ 
and hence $[2\pi(n+1)]^{-1}|S_n(e^{i\la},e^{it)}|^2$ is the Fej\'er kernel (see Section \ref{FSI}).}
\end{exa}

For a nontrivial measure $\mu$ on the unit circle $\mathbb{T}$ 
of the form \eqref{di1} with $f(\la)$ satisfying Szeg\H{o}'s condition \eqref{S}, we define the Szeg\H{o} function
(see, e.g., Grenander and Szeg\H{o} \cite{GS}, Section 1.14, and Simon \cite{Sim1}, p. 93).
\begin{equation}
	\label{Szf}
D(z)=D(f,z)=\exp \left \{ \frac 1 {4\pi} \inl
\frac {e^{it}+z}{e^{it}-z} \log f(\la)\,d\la \right \},\q z = e^{i\la}.
\end{equation}

The next proposition outlines important properties of the Szeg\H{o} function $D(f,z)$
(see, e.g., Grenander and Szeg\H{o} \cite{GS}, p. 51, and Simon \cite{Sim1}, p. 93).
\begin{pp}
	\label{pp77}
	Let $D(f,z)$ be the Szeg\H{o} function defined by (\ref{Szf}). 
	We assume that the Szeg\H{o} condition \eqref{S} is satisfied.
	The following assertions hold.
	\begin{itemize}
		\item[(a)]
		$D(f,z)$ is an outer function from the Hardy space ${H}^{2+}(\mathbb{D}).$
		\item[(b)]
		For almost all $\la \in [-\pi, \pi)$
		\beq
		\label{b2}
		|D(f,e^{i\la})|^2 = f(\la).
		\eeq
		\item[(c)]
		$D(f,z)$ is different from zero inside the unit disk
		$\{z; \, |z| <1\}$, the value $D(f;0)$ is real and positive, and
		\beq
		\label{b3}
		\lim_{n\to\f} \kappa_n = [G(f)]^{-1/2}= \kappa =\frac1{D(f,0)}.
		\eeq
		\item[(d)]
		$D(f,z)$ and  the orthogonal polynomials
	$\I_n(z)$ are connected by the following formula: 		%(see  Grenander and Szeg\H{o} \cite{GS}, p. 51):
		\beq
		\label{b4}
		\I_n(z) =\frac{z^n}{\ol D(f,1/z)} + o(1), \q n\to\f; \,\, |z|>1,
		\eeq
		uniformly on compact sets of the domain $|z|>1$.
		\item[(e)]
		$D(f,z)$ and  the kernel function $S_n(w,z)$ are connected by the formula
		\beq
		\label{b5}
	S_n(\xi,z) = 	\frac 1{1-\bar wz}\frac 1{\ol{D(f,\xi)}}\frac 1{D(f,z)}
	+ o(1), \q
		n\to\f; \,\, |\xi|<1, |z|<1
		\eeq
		uniformly for $|\xi|$ and $|z|\leq r<1$.
		\item[(f)]
		$D(f,z)$ and the Fourier coefficients $d_k$ of the function $\log f(\la)$
		are connected by the formula
		\beq
		\label{b6}
		\log D(f,z) = \frac {d_0}2 + \sum_{k=1}^\f d_kz^k, \q |z|<1.
		%d_k =  \frac 1 {2\pi} \inl e^{-ik\la}\log f(\la)\,d\la.
		\eeq
	\end{itemize}
\end{pp}
The relations \eqref{b3} -- \eqref{b5}, known as Szeg\H{o}'s asymptotic formulas, are valid either inside or outside the unit circle \(\mathbb{T}\). However, when estimating the unknown mean of a stationary process, we set \(\xi = 1\). Therefore, we require asymptotic formulas that are specifically applicable to the unit circle \(\mathbb{T}\).

A useful tool for this purpose is presented in the next two propositions, known as Szeg\H{o}'s 
asymptotic formulas on the unit circle (see Szeg\H{o} \cite{Sz}, pp. 280, 297, and 298).

\begin{pp}
	\label{SzL}
	Let $f(\la)$ be a positive weight function on the unit circle $\mathbb{T}$,
	which satisfies the Lipschitz-Dini condition:
	\beq
	\label{b4-loc-1}
	|f(\la+\de)-f(\la)|<C|\ln\de|^{-1-\al},
	\eeq
	where $C$ and $\al$ are fixed positive numbers. Then we have, for $|z|=1$
	\beq
	\label{b4-loc}
	\I_n(z) =\frac{z^n}{\ol D(f,1/z)} + \vs_n(z),
	\eeq
	where $\vs_n(z)=o(1)$ as $n\to\f$ uniformly for $|z|=1$.
\end{pp}
The next proposition provides local sufficient conditions for \eqref{b4-loc}.
\begin{pp}
	\label{SzLC}
	Let $f(\la)$ be a Riemann integrable function, and let it have the form:
	\beq
	\label{b4-loc-2}
	f(\la)= h(\la)\left|\prod_{k=1}^m(z-z_k)^{\be_k}\right|,\q z=e^{i\la},
	\eeq
	where $0<c\leq h(\la)\leq C$ and $z_k=e^{i\la_k}$ are distinct points on the unit circle, $\be_k>0$, $k=1,\ldots,m$. 
	Assume that $f(\la)$ is differentiable at the fixed point $\la=a$, 
	$\xi=e^{ia}\neq z_k$, $k=1,\ldots,m$, and let the following be bounded near $\la=a$:
	$$
	\frac{f(\la)-f(a)-f'(a)(\la-a)}{(\la-a)^2}.
	$$
Then the relation \eqref{b4-loc} holds for $z=\xi=e^{ia}$ in the less informative form:
\beq
\label{b4-loc-3}
\I_n(\xi) =\frac{\xi^n}{\ol D(f,1/\xi)} + \vs_n,
\eeq
where  $\vs_n(z)=o(1)$ as $n\to\f$.
\end{pp}
The next proposition provides an estimate of the orthogonal polynomials corresponding 
to the Fisher--Hartwig weight, which has a singularity at a single point
(see Golinskii \cite{Gol-3}).
\begin{pp}
	\label{Gol-1}
Let $f(\la)$ be a weight function on the unit circle $\mathbb{T}$ defined as:
	\beq
	\label{Bgol-1}
	f(\la)=|e^{i\la}-e^{i\la_0}|^{2\al}g(\la), \q \al>0,
	\eeq
where $g(\la)$ is  positive and satisfies the conditions: $g, g^{-1}\in L^1(\mathbb{T})$, and  $|g(\la)-g(\la_0)|\leq C|\la-\la_0|$
for $|\la-\la_0|\leq\de$. 
Then, for the orthogonal polynomials $\I_n$ corresponding to $f$, the following estimate holds:
	\beq
	\label{Bgol-2}
	|\I_n(e^{i\la},f)| \simeq n^\al\q n\in \mathbb{N}.
	\eeq
\end{pp}

\sn{Solution of Szeg\H{o}'s extremum problem}

The following  proposition presents the solution to part (a) of Szeg\H{o}'s extremum problem (see, e.g., 
Grenander and Szeg\H{o} \cite{GS}, Section 2.2, Nikishin and Sorokin \cite{NS}, Section 3.6, and Simon \cite{Sim1}, p. 124).
\begin{pp}
	\label{L0}
	Let $\mu$ be a nontrivial measure on the unit circle $\mathbb{T}$ 
	of the form \eqref{di1}. Let the set of polynomials $\mathcal{Q}_n(\xi)$ and the kernel $S_n(z,\xi)$
	be defined as in \eqref{L5Q} and \eqref{op8}, respectively. The polynomial
	\beq
	\label{op16}
	p_n(z)=\frac{S_n(z,\xi)}{S_n(\xi,\xi)}
	\eeq
	is the unique solution to the extremum problem in \eqref{L5s}, and the minimum is 
	equal to $S_n^{-1}(\xi,\xi)$.
	Specifically, the following relation holds:
	\beq
	\label{op17}
		\la_n(\xi,\mu):=\min_{q_n\in \mathcal{Q}_n(\xi)}||q_n||^2_\mu=||p_n||^2_\mu
	=\frac{1}{S_n(\xi,\xi)}.
	\eeq
\end{pp}
\begin{proof}
Expand any trial polynomial $q_n(z)$ by formula:
\begin{equation}
	\label{Sz3}
	q_n(z) = \sum\limits_{k = 0}^{n}c_k\varphi_k(z).
\end{equation}
Then the normalization condition $q_n(\xi)=1$ says
\begin{equation}
	\label{Sz4}
\sum\limits_{k = 0}^{n}c_k\varphi_k(\xi)=1,
\end{equation}
while
\begin{equation}
	\label{Sz5}
||q_n||^2_\mu=\sum\limits_{k = 0}^{n}|c_k|^2.
\end{equation}

By \eqref{Sz5} and the Schwarz inequality we obtain 
\begin{equation}
	\label{Sz6}
S_n(\xi,\xi)\sum\limits_{k = 0}^{n}|c_k|^2\geq 1.
\end{equation}
On the other hand, for the the choice
\begin{equation}
	\label{Sz7}
c_k=\frac{\ol{\varphi_k(\xi)}}{S_n(\xi,\xi)}
\end{equation}
that is, for $q_n$ given by the right side of \eqref{op16}, we have
\begin{equation}
	\label{Sz8}
||q_n||^2_\mu=\frac{S_n(\xi,\xi)}{S^2_n(\xi,\xi)}=\frac1{S_n(\xi,\xi)},
\end{equation}
and the result follows.
\end{proof}
One immediate useful consequence of Proposition \ref{L0} is the following:
\begin{pp}
	\label{L0-M}
	Let $\mu_1$ and $\mu_2$ be nontrivial measures on the unit circle $\mathbb{T}$ 
	such that $\mu_1\leq\mu_2$, which means that $\mu_2-\mu_1$ is a nonnegative measure. Then
	\beq
	\label{op16-M}
	S_n(z,z,\mu_2)\leq S_n(z,z,\mu_1).
	\eeq
	\end{pp}

For $\theta,\la\in \Lambda$ and $r\in[0,1)$, define the real Poisson kernel:
\beq
P_r(\theta,\la): = \frac{1-r^2}{1-2r\cos(\la-\theta)+r^2}
\eeq

In the interior of the unit circle in the  $\xi$-plane there is defined a
function $G (\xi,f)$ (cf. \eqref{a2}):
\beq
\label{SGT-1}
G(\xi,f): = \left \{
\begin{array}{ll}
	\exp\left\{\frac1{2\pi}\inl\ln f(\la)P_r(\theta,\la)\,d\la \right\} &
	\mbox{if \, $\ln f \in {L}^1(\mathbb{T})$}\\
	0, & \mbox { otherwise,} \q
\end{array}
\right.
\eeq
where $\xi=r e^{i\theta}$; $|r|<1$.

The next result, known as Szeg\H{o}'s generalized theorem,
provides the solution to part (b) of Szeg\H{o}'s extremum problem
in the unit disk $\mathbb{D}$.
The proof can be found in Grenander and Szeg\H{o} \cite{GS}, 
Section 3.2 (see also Simon \cite{Sim1}, Section 2.3).

\begin{pp}[Szeg\H{o}'s generalized theorem]
	\label{SGT}
	Let $\mu$ be a nontrivial measure on the unit circle $\mathbb{T}$ of the form \eqref{di1}. Then for any fixed $\xi\in \mathbb{D}:=\{z: |z|<1\}$ with $\xi=re^{i\theta}$ the following asymptotic relation holds:
	\begin{equation}
		\label{Sz1}
		\lim_{n\to\f}\la_n(\xi,\mu) =\la(\xi,\mu)=2\pi G(\xi,f)(1-|\xi|^2),
	\end{equation}
where $\la_n(\xi,\mu)$ is the $n^{th}$ Christoffel function defined by \eqref{L5s}.
	In the case when $\ln f \notin {L}^1(\mathbb{T})$, the limit in \eqref{Sz1} must be replaced by zero.	
\end{pp}

\begin{rem}
	\label{SGTr-1}
	\rm {The quantity $G(\xi,f)$ in equation \eqref{SGT-1} is a mean value of $f(\la)$ that depends on $\xi$ and represents a generalization of the geometric mean $G(f)$, defined as in \eqref{a2}. It reduces to this geometric mean when $\xi=0$.  
		In this special case, Proposition \ref{SGT} reduces to the Kolmogorov-Szeg\H{o} Theorem
		(see Theorem \ref{th34}, part (a)).}
\end{rem}
\begin{rem}
	\label{SGTr-2}
	\rm {Taking into account the equality (see Simon \cite{Sim-1}, p. 118):
		\beq
		\label{SGT-100}
		-\frac1{2\pi}\inl P_r(\theta,\la)\ln P_r(\theta,\la)\,d\la =\ln(1-r^2)<0,
		\eeq		
		the relation \eqref{Sz1} can be written in the following form (see Simon \cite{Sim-1}, p. 139):
		\begin{equation}
			\label{Sz1P}
			\lim_{n\to\f}\la_n(\xi,\mu) =\la(\xi,\mu)=2\pi G(\xi,f/P_r).
		\end{equation}
	}
\end{rem}
\begin{rem}
	\label{SGTr-3}
	\rm {Note that the relation \eqref{Sz1} holds already for a finite value of $n$
		\begin{equation}
			\label{Sz2}
			\la_n(\xi,\mu) =2\pi G(\xi,f)(1-|\xi|^2)
		\end{equation}
		if and only if
		\begin{equation}
			\label{Sz02}
			f(\la)=\frac{1}{|1-\ol\xi z|^2}\frac1{\phi(\la)}, \q z=e^{i\la},
		\end{equation}
		where $\phi(\la)$ denotes a positive trigonometric polynomial of order $n$.}
\end{rem}	

Theorem \ref{SGT} provides the asymptotic behavior of the Christoffel function $\la_n(\xi,\mu)$ inside the unit disk $\mathbb{D}$.
However, when estimating the unknown mean of a stationary process, we set $|\xi|=1$, which requires us to examine the asymptotic behavior of $\la_n(\xi,\mu)$ on the boundary of the unit disk.

Note that on the boundary of the unit disk, Proposition \ref{SGT} only gives the trivial result $\lambda(\xi, \mu) =0$ for almost every $\xi$ with $|\xi|=1$. More precisely, it was proven in Máté et al. \cite{MNT} that
\begin{equation}
	\label{Sz21}
	\la_n(e^{i\la},\mu) = \mu(\{\la\})\q \text{for every} \q \la\in[-\pi,\pi).
\end{equation}

Therefore, in view of \eqref{Sz21}, it is more natural to examine a weighted limit of $\la_n(\xi,\mu)$ rather than $\lim_{n\to\f}\la_n(\xi,\mu)$. 
The following result was established by M\'at\'e et al. \cite{MNT}.
\begin{pp}[M\'at\'e's theorem]
	\label{MNT}
Let $\mu$ be a nontrivial measure on the unit circle $\mathbb{T}$ of the form \eqref{di1}. Assume that $\mu$ satisfies the Szeg\H{o} condition \eqref{S}. Then 
	\begin{equation}
		\label{M-Sz22}
		\lim_{n\to\f}(n+1)\la_n(e^{i\la},\mu) =2\pi f(\la)
	\end{equation}
holds for almost every $\la\in[-\pi,\pi)$.
\end{pp}
It is important to note that the limiting relation \eqref{M-Sz22} was originally derived by Szeg\H{o} under more restrictive conditions. Specifically, the measure $\mu$ must be absolutely continuous, and $f=\mu'>0$ needs to be twice continuously differentiable
(see Szeg\H{o} \cite{SzP}, Theorem I', p. 461).

The proof of the following result can be found in Golinskii \cite{Gol-1}.
\begin{pp}%[M\'at\'e's theorem]
	\label{Gol-2}
	Let $\mu$ be a nontrivial measure on the unit circle $\mathbb{T}$ that satisfies the Szeg\H{o} condition \eqref{S}. Assume that $\mu$ is absolutely continuous and the weight $f(\la)$ is continuous and $f(\la)\geq c>0$. Then the following relation holds
	\begin{equation}
		\label{Sz-Gol}
		%\lim_{n\to\f}\frac{n+1}{S_n(z_0,z_0)} =
		\lim_{n\to\f}(n+1)\la_n(e^{i\la_0},\mu) =2\pi f(\la_0), \q z_0=e^{i\la_0}.
	\end{equation}
\end{pp}
\begin{rem}
	\label{Ger-r}
	\rm {In Geronimus \cite{Ger-4}, it was proved that the asymptotic formula \eqref{M-Sz22} holds for almost every $\lambda \in [-\pi, \pi]$ under the weaker condition that $f = \mu' >0$ almost everywhere on $[-\pi, \pi]$. Additionally, it was observed that this condition is satisfied if the measure is from the Szeg\H{o} class. Furthermore, it is also satisfied for the Pollaczek-Szeg\H{o} weight function (see equation \eqref{nd4}), 
	for which the Szeg\H{o} condition \eqref{S} is violated.
	Unfortunately, as noted by Nevai \cite{Nev}, the Geronimus proof of the asymptotic formula \eqref{M-Sz22} under this weaker condition, contains an error. For details, see Nevai \cite{Nev}, pp. 27--28.}
\end{rem}

\s{Formulas and properties of the variance of BLUE}
\label{formulas}

In this section, we present formulas and properties of the variance of the BLUE.
We first utilize Kolmogorov's isometric isomorphism between the time and frequency domains, $V: X(t)\leftrightarrow e^{it\lambda}$, to reformulate the problem of finding the BLUE in the frequency domain. Then, we apply Szeg\H{o}'s results to derive a convenient formula for the variance $\Var\left(\widehat{m}_{n,BLU}\right)$
in terms of orthogonal polynomials on the unit circle with respect to the
spectral density $f(\lambda)$.

%We will use the following notation.
%By $\mathcal{Q}_n:=\mathcal{Q}_n(1)$ we denote the set of monic polynomials $q_n(z)$, $z\in\mathbb{C}$, of degree $n\in\mathbb{N}$ satisfying the condition $q_n(1) = 1$ (see \eqref{L5Q}).

To begin, using equations \eqref{1.1} - \eqref{1.3}, we can write
\begin{align}\nonumber\label{Formula}
	\textrm{Var} \left(\widehat{m}_n\right) &= \sum_{j,k=0}^n c_j\bar{c}_kr(j-k)
	=\int_{-\pi}^\pi\left|\sum_{\nu = 0}^n c_\nu e^{i\nu\lambda}\right|^2f(\lambda) d\lambda\\
	&=\int_{-\pi}^\pi\left|\sum_{\nu = 0}^n c_\nu e^{i(n-\nu)\lambda}\right|^2f(\lambda) d\lambda
	=\int_{-\pi}^{\pi}|q_n(e^{i\lambda}|^2 f(\lambda)d\lambda,
\end{align}
where $q_n(z)=c_0z^n+c_1z^{n-1}+\cdots+c_{n-1}z+c_n$. 

Thus, the problem of finding $\widehat{m}_{n,BLU}$ reduces to the
special case of Szeg\H{o}'s extremum problem on the unit circle where 
$\xi=1$. Specifically (see Section \ref{SzW}), 
\beq
\label{L5}
||q_n||^2_f:=\int_{-\pi}^{\pi}|q_n(e^{i\lambda})|^2 f(\lambda)d\lambda={\rm min},
\q q_n(z)\in \mathcal{Q}_n(1).
\eeq

The polynomial $p_n(z):=p_n(z,f)$, which solves the minimum problem in \eqref{L5}
is referred to as the optimal polynomial (of degree $n$) for $f(\la)$ within the class $\mathcal{Q}_n(1)$. The minimum itself, which represents 
the variance of $\widehat{m}_{n,BLU}$, is denoted by $\sigma_n^2(1,f)$. 
Therefore, for the variance $\sigma_n^2(1,f)$ of the BLUE we have the following formula in terms of the optimal polynomial $p_n(z)$: %(see Definition \ref{def3.1}):
\bea
\label{OSm2}
\sigma_n^2(1,f):=\Var_f(\widehat{m}_{BLU})=\min_{q_n\in \mathcal{Q}_n(1)}||q_n||^2_f=||p_n||^2_f.
\eea

Note that since the optimal $p_n(z)$ is determined by $\widehat{m}_{BLU}$ (and vice versa), it is unique.

For simplicity, we will use the following notation: 
$\widehat{m}_f$ and $\sigma_n^2(f):=\sigma_n^2(1,f)$ will denote the BLUE 
and the variance of the BLUE calculated with respect to the spectral density $f(\lambda)$, respectively.
The optimal polynomials for $f(\lambda)$ will be denoted by $p_n(z)$.

Applying Proposition \ref{L0} with $\xi=1$, we find that 
the optimal polynomial $p_n(z,f)$ for determining the BLUE $\widehat{m}_{n,BLU}$ is given by the following formula:
\begin{equation}
	\label{op}
	\nonumber
	p_n(z, f) = \frac{S_n(z,1)}{S_n(1,1)}
	= \frac{\sum\limits_{v = 0}^{n}\overline{\varphi_v(1)}\varphi_v(z)}{
		\sum\limits_{v =0}^{n}|\varphi_v(1)|^2}.
\end{equation}

Thus, for the variance $\sigma_n^2(f)$ of the BLUE, we have the following formula
in terms of Szeg\H{o}'s kernel and the orthogonal polynomials:
\bea
\label{OSm23}
\sigma_n^2(f)%=\min_{q_n\in \mathcal{Q}_n(1)}||q_n||^2_f=||p_n||^2_f
=\frac{1}{S_n(1,1)}= \left(\sum\limits_{v =0}^{n}|\varphi_v(1)|^2\right)^{-1}.
\eea

Next, by comparing equations \eqref{L5s} and \eqref{OSm2}, we observe that
$\sigma_n^2(f)$ coincides with the $n^{th}$ Christoffel function at the point $\xi=1$. Specifically, we have the following formula for the variance $\sigma_n^2(f)$ in terms of the Christoffel function:
\bea
\label{si-la}
\sigma_n^2(f)=\la_n(1,f).
\eea

Denote by $B_n(f)$ the truncated Toeplitz matrix generated by the spectral density $f(\la)$:
\begin{equation}
	\label{T-Ap0}
	B_n(f)=\|\hat f(k-j)\|, \, \,k,j=1,\ldots, n, 
\end{equation}
where
\begin{equation}
	\label{T-9ss4}
	\hat f(k-j)=\int_\mathbb{T}e^{i(k-j)\la} f(\la)\,d\la, 
\end{equation}
are the Fourier coefficients of the function $f(\la)$.

By comparing equations \eqref{1.1} and \eqref{T-9ss4}, we observe that $\hat f(k-j)=r(k-j)$
for $k,j=1,\ldots, n$, which implies that $B_n(f)=R_n$, where $R_n$ is the covariance matrix of the observed process.
Thus, in view of equation (\ref{2.2}), the variance $\Var(\widehat{m}_{BLU})$ of $\widehat{m}_{BLU}$ can be expressed in terms of Toeplitz matrices as follows:
\begin{equation}
	\label{T-2.2a}
	\sigma_n^2(f)=\Var( \widehat{m}_{BLU}) = (\mathbf{1}^TB_n^{-1}(f)\mathbf{1})^{-1},
\end{equation}
where  $\mathbf{1}^T=(1,\ldots, 1)$.

The next proposition presents two important properties of the variance $\sigma_{n}^2(f)$.
\begin{pp}
	\label{R2}
	For the sequence $\{\sigma_n^2(f), \, n\in\mathbb{N}\}$ the following assertions hold:
	\begin{itemize}
		\item[(a)] $\sigma_n^2(f)$ is non-increasing
		in $n$: $\sigma_{n+1}^2(f)\leq\sigma_{n}^2(f)$.
		\item[(b)] $\sigma_n^2(f)$  is a non-decreasing functional of $f(\lambda)$, that is,
		$\sigma_{n}^2(f)\leq\sigma_n^2(g)$ when
		$f(\lambda)\leq g(\lambda)$ for all $\lambda\in \Lambda$.
	\end{itemize} 
\end{pp}
\begin{proof}
	Assertion (a) follows from the obvious embedding $\mathcal{Q}_n(1)\subset\mathcal{Q}_{n+1}(1)$. Assertion (b) is derived from the definition of an optimal polynomial and the fact that the square norm $\| \psi \|_f^2$ is non-decreasing in $f$ for a fixed function $\psi$. Indeed, let $p_n(z,f)$ and $p_n(z,g)$ denote the optimal polynomials corresponding to the spectral densities $f$ and $g$, respectively (see equation \eqref{OSm2}). We have
	\beaa
	\sigma_n^2(f)=||p_n(f)||^2_f \leq ||p_n(g)||^2_f \leq ||p_n(g)||^2_g =\sigma_n^2(g),
		\eeaa
and the result follows. Note that assertion (b) can also be deduced from Proposition \ref{L0-M} and equation \eqref{OSm23}.
\end{proof}
\begin{rem}
	\label{RMT4}
	{\rm Formula \eqref{OSm2} highlights a similarity between
		the estimation problem discussed here and the best linear mean-squared prediction problem for stationary processes.
	Specifically, the best linear mean-squared one-step ahead prediction error variance $\sigma_n^2(0,f)$, based on a history of length $n$ (see equation \eqref{si-pred}), 
	has representations similar to \eqref{OSm2} and \eqref{OSm23}. 
	The key difference is that, in this case, the minimum is taken over the set 
	$\mathcal{Q}_n$ of monic polynomials of degree $n$, rather than $\mathcal{Q}_n(1)$. 
	For the definitions of the sets $\mathcal{Q}_n$ and $\mathcal{Q}_n(1)$, see Section \ref{SzPr}. For more information regarding the prediction problem, we refer to the papers by Babayan and Ginovyan \cite{BG2023} and Babayan et al. \cite{BabGT}}.
\end{rem}

\s{Asymptotic behavior of the variance of the BLUE for nondeterministic models}
\label{ND-BLU}

\sn{Asymptotic behavior of  $\Var_f(\widehat m_{BLU})$ for short-memory models}

The first results regarding the asymptotic behavior of the variance of the BLUE for short-memory models can be traced back to the  classical works of U. Grenander \cite{Gr-1,Gr-2,Gr-3} (see also Grenander and Rosenblatt \cite{GR1}, and Grenander and Szeg\H{o} \cite{GS}, Sections 11.1-11.3). He demonstrated that the BLUE $\widehat{m}_{BLU}$ has an asymptotic variance of
$2\pi f(0)/n$, provided that $f(\lambda)$ is positive and continuous. Specifically, the following results were established.
\begin{thm}
	\label{Gr1}
Assume that the conditions of Proposition \ref{SzL} (or Proposition \ref{SzLC}) are satisfied. Then, for the variance of $\widehat m_{BLU}$, we have:
	\begin{equation}
	\label{Gr-1f}
			\Var_f(\widehat m_{BLU})\sim 2\pi f(0)/n  \q {\rm as}\q n \rightarrow \infty.
	\end{equation}
\end{thm}
\begin{proof}
By equation \eqref{OSm23}, we have 
\bea
\label{Grf1}
	\Var_f(\widehat m_{BLU})= \left(\sum\limits_{v =0}^{n}|\varphi_v(1)|^2\right)^{-1}.
\eea	
Under the assumptions of the theorem, we have (see equations \eqref{b4-loc} and \eqref{b4-loc-3}):
	\beq
	\label{Grf2}
	\I_n(1) =\frac{1}{\ol D(f,1)} + \vs_n,\q  \vs_n=o(1)\q {\rm as}\q n \rightarrow \infty.
	\eeq
From equations \eqref{Grf1} and \eqref{Grf2}, and Proposition \ref{pp77} (b), the result follows.	
\end{proof}

\begin{thm}
	\label{Gr01}
Assume that the spectral density $f(\la)$ is positive and continuous. 
Then, for the variance of $\widehat m_{BLU}$, we have:
	\begin{equation}
		\label{Gr-2f}
	\sigma_n^2(f):=	\Var_f(\widehat m_{BLU})\sim 2\pi f(0)/n  \q {\rm as}\q n \rightarrow \infty.
	\end{equation}
\end{thm}
\begin{proof}
Consider a $p$-order autoregressive (AR$(p)$) stationary process $X(t)$ with the spectral density $f(\la)$ (see equation \eqref{arma}): 
 \begin{equation}
	\label{ar1}
	f(\la) = \frac{1}{2\pi}\cd |\psi_p(z)|^{-2}
	= \frac{1}{2\pi}\cd \left|\sum_{k=0}^{p}b_kz^k)\right|^{-2}, \q z=e^{i\la},
\end{equation} 
where the polynomial $\psi_p(z)$ has all its zeros in $|z|<1$ and $b_p>0$.
In this case, for the orthogonal polynomials $\I_{\nu}(z)$, we have (see, e.g., Grenander and Szeg\H{o} \cite{GS}, pp. 43 and 210)
 \begin{equation}
 	\label{ar2}
 	\I_{m+p}(z) = \sum_{k=0}^{p}b_kz^{k+m}, \q m=0,1,\ldots.
 \end{equation} 
Therefore,
 \begin{equation}
	\label{ar3}
\frac1n\sum\limits_{k =0}^{n}|\varphi_k(1)|^2 = f^{-1}(0) + o(1) \q {\rm as}\q n \rightarrow \infty.
\end{equation} 
Now let $f(\la)$ be an arbitrary positive and continuous spectral density.
Given an arbitrary $\vs>0$, we can find two polynomials $p_1(z)$ and $p_2(z)$  that satisfy
\begin{align*}
&f_1(\la):=|p_1(z)|^{-2}\leq f(\la)\leq |p_2(z)|^{-2}:=f_2(\la),\\
&f_2(\la)-f_1(\la)<\vs, \q z=e^{i\la}.
\end{align*}
Since, by Proposition \ref{R2} (b), $\sigma_n^2(f)$ is a non-decreasing functional of $f(\lambda)$, the variance $\sigma_n^2(f)$ lies within the interval 
$(\alpha_n, \beta_n)$, where
\begin{equation}
	\label{ar4}
	\al_n\sim 2\pi f_1(0)/n, \,\, \be_n\sim 2\pi f_2(0)/n  \q {\rm as}\q n \rightarrow \infty. 
\end{equation} 
Since $f_1(0)-f(0)$	and $f_2(0)-f(0)$ can be made arbitrarily small by choosing
$\varepsilon>0$ sufficiently small, we find that
\begin{equation}
	\label{ar44}
	\sigma_n^2(f)\sim 2\pi f(0)/n  \q {\rm as}\q n \rightarrow \infty, 
\end{equation} 
and the result follows.	
\end{proof}
As an immediate consequence of Proposition \ref{MNT} and equation
\eqref{si-la} we obtain the following result.
\begin{thm}
	\label{MNT-C}
Let $\mu$ be a nontrivial measure on the unit circle $\mathbb{T}$ of the form \eqref{di1}. Assume that $\mu$ satisfies the Szeg\H{o} condition \eqref{S}. Then 
	\begin{equation}
		\label{Sz22}
		\lim_{n\to\f}n\si_n^2(\mu) =2\pi f(0).
	\end{equation}
\end{thm}
The next result follows from Proposition \ref{Gol-2} and equation \eqref{OSm23}.
\begin{thm}
\label{Gol-T}
Assume that the conditions of Proposition \ref{Gol-2} are satisfied at the point $\la_0=0$. Then, the asymptotic relation \eqref{Sz22} holds.
\end{thm}

\begin{rem}
{\rm While Theorems \ref{Gr1} -- \ref{Gol-T} %{MNT-C} %and \ref{Gr01} 
provide asymptotic expressions for 
the variance of the BLUE of the unknown mean $m$, they do not allow us to compute
$\widehat{m}_{BLU}$ in practice, except in a few special cases.
Therefore, approximations are often necessary.}
\end{rem}

Now let us deal with the {\it least squares estimator} (LSE) $\widehat m_{LS}$ of the mean $m$:
\begin{equation}
	\label{V4.1pls}
	\widehat{m}_{LS}: =\bar X=\frac{1}{n+1}\sum_{k = 0}^n X(k).
\end{equation}
\begin{thm}
\label{Gr2}
Assume that the spectral density $f(\la)$ is continuous at $\la=0$. Then, for the variance of $\widehat m_{LS}$, we have:
	\begin{equation}
		\label{Gr-2}
		\Var_f(\widehat m_{LS})\sim 2\pi f(0)/n  \q {\rm as}\q n \rightarrow \infty.
	\end{equation}
\end{thm}
\begin{proof}
From equation \eqref{Formula} with $c_k=1/(n+1)$, $k=0,1,\ldots,n$, for the variance of  $\widehat m_{LS}$, we have 
\begin{equation}
	\label{Grf3}
\Var_f(\widehat m_{LS})= \frac1{n+1}\inl F_{n+1}(\la) f(\la)d\la,
\end{equation}
where $F_{n+1}(\la)$ is the Fej\'er kernel (see equation \eqref{F14-1}).
By applying Proposition \ref{F-Ker} from \eqref{Grf3}, we can infer \eqref{Gr-2}, leading to the conclusion.
%\begin{equation}
%	\label{s6}
%	F_T(\la):= 
%		\frac1{2\pi (n+1)^2}
%		\cdot\left[\frac{\sin((n+1)\la/2)}{\sin(\la/2)}\right] 
%		\mbox { for  $\la \in \mathbb{T}, \, n \in \mathbb{N}$},
%\end{equation}
\end{proof}

As an immediate consequence of Theorems \ref{Gr01} and \ref{Gr2}, we obtain the following result.
\begin{thm}
\label{Gr3}
Assume that the spectral density $f(\la)$ of the process $X(t)$ is positive and continuous. Then, 
the least squares estimator $\widehat m_{LS}$ is asymptotically as effective as the best unbiased estimator $\widehat m_{BLU}$ for the unknown mean $m$. Specifically, the asymptotic efficiency  
$e(\infty, \widehat m_{LS}, f)$ of $\widehat m_{LS}$ is equal to 1 (see equation \eqref{eff2}).
\end{thm}

The next example demonstrates that the asymptotic behavior of the variance of 
an estimator for the mean $m$ essentially depends on the
dependence (memory) structure of the underlying observed model $X(t)$
(see, e.g., Grenander and Szeg\H{o} \cite{GS}, p. 211).

\begin{exa}
\label{Ex5.1}	
{\rm Consider a first-order moving average MA$(1)$ process $X(t)$:
		$$
		X(t)=\vs (t)-b \cd \vs (t-1),\q \vs (t)\sim WN(0,1),
		\q 0\le b\le1.
		$$
		\n The spectral density is
		$f(\la)=\frac{1}{2\pi}\cd |1-be^{i\la}|^2$
		(see formula \eqref{arma}).
		
		a) First, assume that $X(t)$ has short memory, meaning that $0\le b<1$.
		It is easy to verify that the assumptions of Proposition \ref{SzL} are satisfied. Therefore, according to Theorems \ref{Gr1} and \ref{Gr2}, we have 
			\begin{equation}
			\label{Gr-3}
			\Var_f(\widehat m_{LS})\sim \Var_f(\widehat m_{BLU})\sim 2\pi f(0)/n  \q {\rm as}\q n \rightarrow \infty.
		\end{equation}
		Therefore, in this case, $\widehat m_{LS}$ is asymptotically efficient. 
		
		b) Now, let $b=1$. We have $f(\la)= \frac{1}{2\pi}\cd|1-e^{i\la}|^2,$ and 
		hence $f(0)=0$, indicating that the process $X(t)$ is antipersistent. 
		What is the efficiency of $\widehat{m}_{LS}$ in this case?
		We will analyze this scenario in detail in the next section and will show that $\widehat{m}_{LS}$ is far from being asymptotically efficient.}
\end{exa}

\sn{Asymptotic behavior of $\Var_f(\widehat m_{BLU})$ for a class of antipersistent models}
\label{Vitale}

Denote by $\mathcal{F}_r$ ($r\in\mathbb{N}$) the class of spectral densities 
$f(\la)$ that are even, continuous, and positive,
except for a $2r$th order zero at the origin. 
Specifically, we define
\beq
\label{Fg}
\mathcal{F}_r(g):=\left\{f(\lambda) = (2\pi)^{-1}|1 - e^{i\lambda}|^{2r}g(\lambda),\q r\in N\right\},
\eeq
where $g(\la)$ is the spectral density of a short memory process.
An example of a time series model from the class $\mathcal{F}_r(g)$ is the ARIMA($p,2r,q$)
model, in which case $g(\la)$ is given by equation \eqref{arma}.
\begin{thm}
\label{Vit1}
Let $f(\lambda)\in \mathcal{F}_1(1)$ (with $g(\la)\equiv1$). 
The coefficients $c_k$ and the variance $\Var_f(\widehat m_{BLU})$ of $\widehat m_{BLU}$ are given by the following formulas:
\begin{equation}
		\label{V4.1c}
c_k=\frac{6}{n+2}\left[\frac{k}{n}\left(1-\frac{k}{n+1}\right)\right],\q k=0,1,\ldots,n
	\end{equation}
and 
\begin{equation}
	\label{V4.1}
	\Var_f(\widehat m_{BLU})=\frac{12}{n(n+1)(n+2)}.
\end{equation}
\end{thm}
Formulas \eqref{V4.1c} and \eqref{V4.1} were derived by Vitale \cite{Vit} through direct calculations.
The variance expression in \eqref{V4.1} was previously calculated
in Grenander and Szeg\H{o} \cite{GS}, p. 211), using the following explicit form of the orthogonal polynomials $\varphi_v(z)$
corresponding to the spectral density 
$f(\lambda) = \frac{1}{2\pi} |1 - e^{i\lambda}|^2$ (see also Simon \cite{Sim-1}, p. 76):
$$
\varphi_v(z)= [(\nu+1)(\nu+2)]^{-1/2}[(\nu+1)z^\nu+\nu z^{\nu-1}+\cdots+1].
$$
Then $|\varphi_v(1)|^2= (\nu+1)(\nu+2)/4$, and by equation \eqref{OSm23}, we have
\bea
\label{Grf100}
\Var_f(\widehat m_{BLU})= \left(\sum\limits_{v =0}^{n}|\varphi_v(1)|^2\right)^{-1}
=\frac{12}{n(n+1)(n+2)}.
\eea

Consider the least squares estimator $\widehat m_{LS}$ for the 
model $f(\lambda)\in \mathcal{F}_1(1)$ (see Example \ref{Ex5.1} (b)). We have
\begin{equation}
	\label{V4.1a1}
	\widehat{m}_{LS}: =\frac{1}{n+1}\sum_{k = 0}^n X(k)=\frac{1}{n+1}(\vs_n-\vs_{-1}).
\end{equation}
For the variance $\Var_f(\widehat m_{LS})$, we have  
\begin{align}
	\label{V4.1a2}
	\nonumber
	\Var_f(\widehat m_{LS})&=\frac1{(n+1)^2}\int_{-\pi}^{\pi}
\left| \sum_{k=0}^{n} e^{ikt}\right|^2\frac{1}{2\pi} |1 - e^{i\lambda}|^2	d\la\\
&=\frac1{\pi (n+1)^2}\int_{-\pi}^{\pi}[(1-\cos((n+1)\la)]d\la	=\frac2{(n+1)^{2}}.
\end{align}

From equations \eqref{V4.1} and \eqref{V4.1a2}, we conclude that in this case, 
$\widehat m_{LS}$ is far from being asymptotically efficient. Specifically, we have
$e(\f,\widehat m_{LS}, f) = 0$. 

The next result, obtained by applying Proposition \ref{F-Ker}, describes the asymptotic behavior of $\Var_f(\widehat m_{LS})$ for a general short memory factor $g(\lambda)$.
\begin{thm}
	\label{Vit5}
Assume that $f\in \mathcal{F}_1(g)$. Then, for the variance of $\widehat m_{LS}$, we have
	\begin{equation}
		\label{V4.2LS}
		n^2\Var_f(\widehat m_{LS})\to \frac1\pi\inl g(\la)d\la \q {\rm as}\q n \rightarrow \infty.
	\end{equation}
\end{thm}

\n
{\bf The parabolic estimator for unknown mean $m$}. For the model $\mathcal{F}_1(1)$, consider the so-called {\it parabolic estimator $\widehat m_{p}$}, which is a slight
modification of the BLUE and offers computational 
convenience (see Vitale \cite{Vit} and Grenander \cite{Gr-4}, p. 238).
The coefficients $c_k$ of $\widehat{m}_{p}$ are given by the following formula (cf. equation \eqref{V4.1c}):
\begin{equation}
	\label{V4.1cp}
	c_k=\frac{6n}{n^2-1}\left[\frac{k}{n}\left(1-\frac{k}{n}\right)\right],\q k=0,1,\ldots,n.
\end{equation}
Straightforward calculations yield the following result (see Vitale
\cite{Vit}):
\begin{thm}
	\label{Vit2}
	With the assumptions of Theorem \ref{Vit1}, the variance of $\widehat m_{p}$ is given by formula:
	\begin{equation}
		\label{V4.1p}
		\Var_f(\widehat m_{p})=\frac{12}{n(n^2-1)}.
	\end{equation}
\end{thm}
From equations \eqref{V4.1} and \eqref{V4.1p}, we infer that
the parabolic estimator $\widehat m_{p}$ is asymptotically efficient in the class $\mathcal{F}_1(1)$. Specifically, 
we have $e(\f,\widehat m_{p}, f) = 1$. 

The next result extends Theorems \ref{Vit1} and \ref{Vit2} to a general short memory factor $g(\la$) (see Vitale \cite{Vit} and Grenander \cite{Gr-4}, p. 239).
\begin{thm}
\label{Vit4}
For the model $\mathcal{F}_1(g)$, the following relations hold:
\begin{align}
	\label{V4.2BLU}
	&n^3\Var_f(\widehat m_{BLU})\to 12g(0) \q {\rm as}\q n \rightarrow \infty\\
\label{V4.2LS1}
&n^3\Var_f(\widehat m_{p})\to 12g(0) \q {\rm as}\q n \rightarrow \infty.
\end{align}
\end{thm}
Theorem \ref{Vit4} shows that the parabolic estimator $\widehat m_{p}$ is asymptotically efficient in the class $\mathcal{F}_1(g)$.

Based on Theorems \ref{Vit5} and \ref{Vit4}, we conclude that for the class $\mathcal{F}_1(g)$ of spectral densities, the parabolic estimator $\widehat m_{p}$ outperforms the least squares estimator $\widehat m_{LS}$. Therefore, in practice, the parabolic estimator should be used when $f(0)$ is known to be small or potentially equal to zero,  even though it requires slightly more computation than the least squares estimator. 

What happens when we use the parabolic estimator \(\widehat{m}_p\) in a situation where \(f(0) > 0\), while the least squares estimate \(\widehat{m}_{LS}\) is asymptotically efficient? How much efficiency do we lose in this case? The following theorem demonstrates that the loss in efficiency is not significant, indicating that the result remains quite satisfactory (see Vitale \cite{Vit} and Grenander \cite{Gr-4}, p. 243).
\begin{thm}
	\label{Vit7}
Let the spectral density $f(\lambda)$ be positive and continuous on $[-\pi,\pi]$. Then,
the asymptotic efficiency of the parabolic estimator $\widehat m_{p}$  is 5/6.	
\end{thm}

In Vitale \cite{Vit}, it was conjectured that the
BLUE associated with the spectral density
$f(\la)=(2\pi)^{-1}|1 - e^{i\lambda}|^{2r}$ is
asymptotically efficient over the class $\mathcal{F}_r(g)$
for all $r\in\mathbb{N}$ (see \eqref{Fg}).
This conjecture was confirmed by Adenstedt \cite{Ad}. 
The details are presented in the next section (see Theorem \ref{thm5.2}).

\sn{Asymptotic behavior of $\Var_f(\widehat m_{BLU})$ for long-memory and antipersistent models}
\label{AD}
In the following, we use the notation $\sigma^2_n(f):=\Var_f(\widehat m_{BLU})$. Additionally, we denote $\widehat{m}_f$ as the BLUE of $m$ calculated with respect to the spectral density $f(\lambda)$, and $p_n(z)$ will represent the optimal polynomials for $f(\lambda)$.
We will focus on spectral densities for which the variance $\sigma^2_n(f)$ does not decrease too rapidly.
Specifically, we assume that the sequence $\{\sigma_n^2(f), n\in\mathbb{N}\}$ is weakly varying (see Section \ref{ww}).
%\begin{equation}
%		\label{3.5a}
%		\lim_{n \rightarrow \infty}\sigma^2_{n+1}(f)/\sigma_n^2(f)=1.
%	\end{equation}
Note that if $\sigma_n^2(f)$ is weakly varying, then for any $\nu\in\mathbb{N}$, we have (see Proposition \ref{p4.1} (a)): 
\begin{equation}
	\label{3.5}
\lim_{n \rightarrow \infty}\sigma^2_{n+\nu}(f)/\sigma_n^2(f)=1.
\end{equation}

We also define several classes of functions that will be considered below. All these functions are assumed to be Lebesgue measurable on
 $[-\pi,\pi].$
\begin{den}
	\label{def3.3}
	Let a function $g(\lambda)$ be defined on $[-\pi,\pi]$. We say that this function belongs to the class:
	
	$L_0$: if $g(\lambda)$ is nonnegative, integrable over $[-\pi,\pi]$, and continuous at $\lambda=0$;
	
	$L_0^+$: if $g(\lambda) \in L_0$ and has a positive lower bound;
	
	$B_0$: if $g(\lambda)$ is nonnegative, bounded, and continuous at $\lambda=0$;
	
	$B_0^+$: if $g(\lambda) \in B_0$ and has a positive lower bound.
\end{den}
\noindent Furthermore, suppose that $g(\lambda)$ has the form: 
\begin{equation}
	\label{3.6g}
	g(\lambda)= h(\lambda)|\lambda-\lambda_1|^{\alpha(1)}|\lambda-\lambda_2|^{\alpha(2)} \cdots |\lambda-\lambda_r|^{\alpha(r)},
\end{equation}
where $r$ is a natural number, $\lambda_1,\lambda_2,\cdots,\lambda_r$ are constants in $[-\pi,\pi]$, and $\alpha(1),$ $\alpha(2),$ $\cdots,$ $\alpha(r)$ are non-negative constants. 
\begin{den}
	\label{def3.3+}
Let a function $g(\lambda)$ be defined as in equation \eqref{3.6g}. 
We say that this function belongs to the class $ZL_0^+$  or $ZB_0^+$
depending on whether $h(\lambda)\in L_0^+$ or $h(\lambda)\in B_0^+$, respectively.
\end{den}

To connect the large-sample behavior of the variance  $\sigma^2_n(f)$ with the behavior of the spectral density near the origin, 
we examine the ratio $\sigma_n^2(fg)/\sigma_n^2(f)$, where $\sigma_n^2(f)$ satisfies condition (\ref{3.5}) and $g(\lambda)$ 
is a suitably "nice" function. 
Essentially, we are now considering the spectral density 
of the process $Y(t)$ to be $f(\lambda)g(\lambda)$ instead of just $f(\lambda)$. 
\begin{thm}
	\label{thm4.1}
	Let $f(\lambda)$ be a spectral density such that the sequence $\{\sigma_n^2(f), n\in\mathbb{N}\}$ is weakly varying, and let $g(\lambda)$ belong to the class $ZB_0^+$. Then
	\begin{equation}
		\label{4.1}
		\lim_{n \rightarrow \infty}\sigma_n^2(fg)/\sigma_n^2(f)=g(0).
	\end{equation}
	Moreover, if $g(0) > 0$, then the BLUE  $\widehat{m}_f$, calculated under the assumption that the spectral density is $f(\lambda)$, is asymptotically efficient with respect to $f(\lambda)g(\lambda)$ in the sense that 
	\begin{equation}
		\label{4.2}
		\lim_{n \rightarrow \infty}\sigma_n^2(fg)/\Var_{fg}(\widehat{m}_f)=1.
	\end{equation}
\end{thm}

Before proceeding with the proof, we first establish some preliminary results. 
\begin{lem}
	\label{lem4.1}
	If $t(\lambda)$ is a nonnegative trigonometric polynomial, then
	\begin{equation}
		\label{4.3}
		\liminf_{n \rightarrow \infty}\sigma_n^2(ft)/\sigma_n^2(f) \geq t(0).
	\end{equation}
\end{lem}
\begin{proof}
	The result is trivial if $t(0)=0;$ therefore we assume that $t(0)>0$. By the F\'ejer-Riesz representation theorem (see Proposition \ref{p4.2}(c)), there is a polynomial $p(z)$ of the same degree, say $\nu$, as $t(\lambda)$ for which $t(\lambda) = |p(e^{i\lambda})|^2$. Letting $q_0(z), q_1(z), \cdots$ be the optimal polynomials for $f(\lambda)t(\lambda)$, we define the polynomial $r(z):=p(z)q_n(z)/p(1)$ of degree $n+\nu$, and note that $r(1)=1$. Therefore, we have
	\begin{equation}
		\label{4.4}
		t(0)\sigma_{n+\nu}^2(f) \leq t(0)\|r\|_f^2 = \|q_n\|_{ft}^2 = \sigma_n^2(ft).
	\end{equation}
The proof is completed by dividing the extreme inequality by $\sigma_n^2(f)$, taking the $\liminf$, and employing equation (\ref{3.5}).
\end{proof}
Define the kernel
\beq\label{K-n}
K_n(\lambda) = |p_n(e^{i\lambda})|^2f(\lambda)/\sigma_n^2(f). 
\eeq

In the next lemma we show that the kernel $K_n(\lambda)$ is an approximate identity (see Section \ref{FSI}).
\begin{lem}
	\label{lem4.2}
The kernel $K_n(\lambda)$ defined by \eqref{K-n}  is an  approximate identity. 
\end{lem}
\begin{proof}
Observe first that $K_n(\lambda)$ is nonnegative and satisfies $\int_{-\pi}^\pi K_n(\lambda)d\lambda=1$.
Now we show that for any $0<\de \leq \pi$ the kernel $K_n(\lambda)$ satisfies the relation
\begin{equation}
		\label{4.5}
		\lim_{n \rightarrow \infty}\int_{\de \leq |\lambda| \leq \pi}K_n(\lambda)d\lambda=0.
	\end{equation}
Define a function $g(\lambda)$ as follows: $g(\lambda)=2$ for $0 \leq |\lambda| < \de$ and $g(\lambda)=1$ for $\de \leq |\lambda| \leq \pi$. With this definition, we have
\begin{equation}
	\label{4.5a}
	\int_{\de \leq |\lambda| \leq \pi}K_n(\lambda)d\lambda
	=2-\|p_n\|^2_{fg}/\sigma_n^2(f).
\end{equation}
Therefore, to prove (\ref{4.5}), it suffices to show that 
	\begin{equation}
		\label{4.6}
		\lim_{n \rightarrow \infty}\|p_n\|^2_{fg}/\sigma_n^2(f) = 2.
	\end{equation}
	
	To this end, we select a nonnegative trigonometric polynomial $t(\lambda)$ that satisfies the conditions: $t(\lambda) \leq g(\lambda)$ and $t(0) = 2$. 
	Then, by Proposition \ref{R2} (b), we have 
	$$\|p_n\|^2_{fg} \geq \sigma^2_n(fg) \geq \sigma_n^2(ft).$$ 
	Next, by dividing by $\sigma_n^2(f)$, taking the $\liminf$, and using Lemma \ref{lem4.1}, we obtain 
	$$\liminf_{n \rightarrow \infty}\|p_n\|^2_{fg}/\sigma_n^2(f) \geq t(0) = 2.$$ 
	On the other hand, according to the definition of function $g$, we have $$\|p_n\|^2_{fg}/\sigma_n^2(f) \leq 2.$$ Thus, the relation \eqref{4.6} is established, and the proof is complete.
\end{proof}
\begin{lem}
	\label{lem4.3}
	If $g(\lambda)$ is in the class $B_0$, then
	\begin{equation}
		\label{4.7}
		\lim_{n \rightarrow \infty}\Var_{fg}(\widehat{m}_f)/\sigma_n^2(f) = g(0).
	\end{equation}
\end{lem}
\begin{proof}
	With $K_n(\lambda)$ defined as in \eqref{K-n}, the relation 
	\eqref{4.7} is equivalent to the following: 
	$$\int_{-\pi}^\pi K_n(\lambda)g(\lambda)d\lambda \rightarrow g(0)\q {\rm as}\q n \rightarrow \infty.$$ 
	This relation follows from Lemma \ref{lem4.2} and Proposition \ref{F-Ker}.
\end{proof}
\begin{lem}
	\label{lem4.4}
	If $g(\lambda)$ is in the class $B_0$, then 
	\begin{equation}
		\label{4.8}
		\limsup_{n \rightarrow \infty}\sigma_n^2(fg)/\sigma_n^2(f) \leq g(0).
	\end{equation}
\end{lem}
\begin{proof}
	The result is immediate from Lemma \ref{lem4.3}, since $\sigma_n^2(fg) \leq \Var_{fg}(\widehat{m}_f).$
\end{proof}
\begin{lem}
	\label{lem4.5}
	Let $g(\lambda)$ be in the class $ZL_0^+$. Then, for any $\epsilon >0$ there exists a trigonometric polynomial $t(\lambda)$ with the following properties: $0 \leq t(\lambda) \leq g(\lambda)$ and $t(0) \geq g(0)-\epsilon$.
\end{lem}
\begin{proof}
It is sufficient to prove the result for each factor on the right side of equation (\ref{3.6g}), since the product of the "approximants" will provide an approximant for  $g(\lambda)$ in the required sense. We consider the individual factors.
	\begin{itemize}
		\item [(a)]
	Let $g(\lambda) \in L_0^+$. For $\epsilon < g(0)$, we choose $0<\delta \leq \pi$ such that $g(\lambda) \geq g(0)-\epsilon$ for $|\lambda| \leq \delta$. Then we can use 
	$$t(\lambda): = [g(0)-\epsilon][(1+\cos \lambda)/2]^N,$$ 
	where the integer $N$ is chosen so large that 
	$$[(1+ \cos \delta)/2]^N \leq \inf g(\lambda)/[g(0)-\epsilon].$$ 
	For $\epsilon \geq g(0)$, we use $t(\lambda) \equiv 0.$
	\item [(b)]
	Let $g(\lambda)=|\lambda-\lambda_0|^{2 \nu}$, where $\nu$ is a positive integer. The result follows from part (a) by expressing $g(\lambda)$ as the product of a positive and continuous function and the trigonometric polynomial $[1-\cos (\lambda-\lambda_0)]^\nu$.
	\item [(c)]
	Let $g(\lambda) = |\lambda-\lambda_0|^\alpha,$ where $\alpha>0$ is not an even integer. 
	If $\lambda_0 = 0$, we can use $t(\lambda) \equiv 0$. Therefore, we assume that 
	$\lambda_0 \neq 0$. For any integer $\nu > \alpha/2$ and constant $0<\delta <|\lambda_0|$, the function
	$$h(\lambda) = [\max \{\delta,|\lambda-\lambda_0|\}]^{-2\nu+\alpha}$$ 
	is positive and continuous. Thus, the function $g_1(\lambda)=h(\lambda)|\lambda-\lambda_0|^{2\nu}$ can be approximated according to parts (a) and (b). Since $g_1(\lambda) \leq g(\lambda)$ and $g_1(0) = g(0)$, the approximant for $g_1(\lambda)$ also serves for $g(\lambda)$.
\end{itemize}
\end{proof}

%With these preliminary results, the proof of the theorem stated at the beginning of this section is quite short.

\begin{proof}[Proof of Theorem \ref{thm4.1}]
	Any function $g(\lambda)$ in $ZB_0^+$ is also in $ZL_0^+$ and in $B_0$, and thus satisfies the inequality (\ref{4.8}). Therefore, to prove the relation (\ref{4.1}), it suffices to show that 
	\begin{equation}
		\label{4.9}
		\liminf_{n \rightarrow \infty}\sigma_n^2(fg)/\sigma_n^2(f) \geq g(0).
	\end{equation}
	For any $\epsilon > 0,$ choosing a trigonometric polynomial $t(\lambda)$ with the properties as in Lemma \ref{lem4.5}, and using Proposition \ref{R2} (b), we have  
	$$\sigma_n^2(fg) \geq \sigma_n^2(ft).$$ 
	Hence, with use of Lemma \ref{lem4.1}, we conclude that the left side of (\ref{4.9}) is bounded below by 
	$$\liminf_{n \rightarrow \infty}\sigma_n^2(ft)/\sigma_n^2(f) \geq t(0) \geq g(0)-\epsilon.$$ 
	Letting $\epsilon \rightarrow 0,$ we obtain the desired relation (\ref{4.9}). 
	
	The second assertion of the theorem, specifically the asymptotic relation in (\ref{4.2}), follows directly from relation (\ref{4.1}) and Lemma \ref{lem4.3}. 
	
	Theorem \ref{thm4.1} is proved.
\end{proof}

\sn{Asymptotic behavior of $\Var_f(\widehat m_{BLU})$ for the class $\mathcal{F}_\al(g)$}
\label{AD-1}
Theorem \ref{thm4.1} allows us to determine the rate at which $\Var(\widehat{m}_{BLU})$ decreases and to identify asymptotically efficient estimators for a broad range of spectral densities, particularly those characterized by a zero (or infinity) of fixed finite order at the origin.
As representatives of such spectral densities consider the following class (cf. \eqref{Fg}): 
\beq
\label{Fal-1}
\mathcal{F}_\al(g):=\left\{f(\lambda):=f_\al(\lambda)g(\lambda),\q \al>-1/2, \,\, \la\in\Lambda\right\},
\eeq
where $g(\la)$ is the spectral density of a short-memory process, and
\beq
\label{Fal}
f_\al(\lambda) = (2\pi)^{-1}|1 - e^{i\lambda}|^{2\al},
\eeq
 Here $\alpha$ (not necessarily an integer) is a constant, with $\alpha > -{1}/{2}$ for integrability. 

Note that this class is highly important from an applications standpoint.
For instance, the ARFIMA($p,-\alpha,q$) model, defined by equation \eqref{AA}, 
and the fractional Gaussian noise with Hurst index $H$, defined by equation \eqref{MT2}, both belong to this class.

The problem is in obtaining the BLUE $\widehat{m}_{f}$ for the unknown mean
$m$ and in determining the asymptotic behavior of the variance 
$\sigma_n^2(f): = \Var_{f} (\widehat{m}_{f})$,
calculated with respect to the spectral density 
$f(\lambda)\in \mathcal{F}_\al(g)$. 

We begin by analyzing the special case where $g(\la)\equiv1$, 
specifically the ARFIMA($0,-\al,0$) model. 
In this case, by \eqref{Fal}, we have 
\begin{equation}
	\label{5.1}
f(\lambda) =	f_\alpha(\lambda)  %(2\pi)^{-1}|1-e^{i\lambda}|^{2\alpha} 
= 2^{2\alpha-1}\pi^{-1}(\sin^2 \lambda/2)^\alpha, \q \la\in[-\pi,\pi].
\end{equation}
Note that as $\lambda \rightarrow 0,$
$$f_\alpha(\lambda) \sim {2\pi}^{-1}|\lambda|^{2\alpha},$$
which diverges when $-{1}/{2} < \alpha < 0$. This indicates  
that the ARFIMA($0,-\al,0$) model exhibits long-range dependence
(see Section \ref{mem}).

The covariance function $r_\alpha(t)$ corresponding to the spectral density (\ref{5.1}) is 
given by the following formula (see Adenstedt \cite{Ad}):
	\begin{equation}
		\label{5.2}
		r_\alpha(k) = (-1)^k \frac{\Gamma (2\alpha+1)}{\Gamma(\alpha+k+1)\Gamma(\alpha-k+1)}, \quad k = 0, \pm 1, \cdots,
	\end{equation}
with the convention $1/\Gamma(\alpha) = 0$ when $\alpha$ is a non-positive integer. Here $\G(\cdot)$ denotes the Gamma function.

Note that as k $\rightarrow$ $\infty$,
\begin{equation}
	\label{2-r}
r_\alpha(k) \sim C_\alpha k^{-2\alpha -1},
\end{equation}
where
\begin{equation}
	\label{2}
	C_\alpha = \frac{1}{\pi}\Gamma(2\alpha+1)(-\sin \pi \alpha).
\end{equation}

Theorem 5.1 of Adenstedt \cite{Ad} shows that the BLUE $\widehat{m}_{\alpha}:=\widehat{m}_{f_\alpha}$ of the
unknown mean $m$, based on the sample $\{X(t), t=0,1,\ldots,n\}$, calculated with respect to the spectral density (\ref{5.1}), is given by
	\begin{equation}
	\label{5.7est}
	\widehat{m}_{\alpha} = \sum_{k=0}^{n}c_k(n, \alpha)X(k),
\end{equation}
where
	\begin{equation}
	\label{5.7}
			c_k(n,\alpha)\\
			= {n \choose k}\frac{B(\alpha+k+1,\alpha+n-k+1)}{B(\alpha+1,\alpha+1)}, \quad k=0,1,\cdots,n,
	\end{equation}
	and the variance $\sigma_n^2(f_\alpha)$ is given by
	\begin{equation}
		\label{5.8}
		\sigma_n^2(f_\alpha) = \Var_{f_\alpha} (\widehat{m}_{\alpha}) = B(n+1, 2\alpha+1)/B(\alpha+1,\alpha+1),
	\end{equation}
where $B(\cdot, \cdot)$ denotes the Beta function.

A straightforward application of Stirling's formula for the gamma function yields, from (\ref{5.8}), that
	\begin{equation}
		\label{5.11}
		\sigma_n^2(f_\alpha) \sim n^{-2\alpha-1}\Gamma(2\alpha+1)/B(\alpha+1,\alpha+1)
		\q {\rm as}\q n \rightarrow \infty.
	\end{equation}

We are now ready to describe the asymptotic behavior of \(\Var_f(\widehat{m}_{BLU})\) within the class \(\mathcal{F}_\alpha(g)\) for the general short memory factor \(g(\lambda)\). From equation (\ref{5.11}), we conclude that the sequence \(\sigma_n^2(f_\alpha)\) is weakly varying (see Proposition \ref{p4.1}). This allows us to apply Theorem \ref{thm4.1} to establish the desired result.

\begin{thm}
	\label{thm5.2}
Let the model $X(t)$ be specified by equation \eqref{1.0}, and let the noise $\{Y_t\}$ have spectral density function $f(\lambda)=f_\alpha(\lambda)g(\lambda)$, where $\alpha > -{1}/{2}$ and $g(\lambda) \in ZB_0^+$ (see Definition \ref{def3.3+}) with $g(0)>0$. Then
	\begin{equation}
		\label{5.12}
		\Var_f(\widehat{m}_{BLU}) \sim n^{-2\alpha-1}\Gamma(2\alpha+1)g(0)/B(\alpha+1,\alpha+1) \q {\rm as}\q n \rightarrow \infty.
	\end{equation}
Moreover, the estimator	$\widehat{m}_{f_\alpha}$ given by \eqref{5.7est} and \eqref{5.7} is asymptotically efficient
in the class $\mathcal{F}_\al(g)$. Specifically, we have
	\begin{equation}
	\label{5.12ef}
	\frac{\Var_f(\widehat{m}_{BLU})}{\Var_{f_\alpha} (\widehat{m}_{\alpha})} \rightarrow 1 \q {\rm as}\q n \rightarrow \infty.
\end{equation}

\end{thm}

For $\alpha=0$ and $\alpha=1$ the theorem gives slightly strengthened versions of Theorems \ref{Gr1} and \ref{Vit4},
respectively (see also Grenander \cite{Gr-3} and Vitale \cite{Vit}). Additionally, the theorem established a conjecture of Vitale \cite{Vit} for integral $\alpha$ (see also Section \ref{Vitale}).

\n {\bf The optimal polynomials for $f_\alpha(\lambda)$.}
We denote by $C_n^{(a)}(x)$ the Gegenbauer (or ultraspherical) polynomials, i.e., the polynomials orthogonal on $[-1,1]$ with respect to the weight function $(1-x^2)^{a-1/2}$ and with the standardization $C_n^{(a)}(1) = \Gamma(n+2a)/[n!\Gamma(2a)]$ for $a \neq 0$. For the properties of the polynomials $C_n^{(a)}(x)$
we refer to Erd\'elyi et al. \cite{Erd}, Vol. 2., Sect. 10.9.

The next result shows that the optimal polynomials for $f_\alpha(\lambda)$ can be represented in terms of $C_n^{(a)}(x)$
(see Adenstedt \cite{Ad}).
\begin{lem}
	\label{lem6.1}
	The optimal polynomials
	\begin{equation}
		\label{6.1}
		p_{n,\alpha}(z) = \sum_{k=0}^n c_k(n,\alpha)z^k
	\end{equation}
	 for $f_\alpha(\lambda)$ are given by the following formula:
	\begin{equation}
		\label{6.2}
		p_{n,\alpha}(e^{2i\theta})=e^{in\theta}C_n^{(\alpha+1)}(\cos \theta)/C_n^{(\alpha+1)}(1).
	\end{equation}
\end{lem}

\begin{proof}
	In terms of the coefficients (\ref{5.7}), it is known that 
	\begin{equation}
		\label{6.3}
		C_n^{(\alpha+1)}(\cos \theta)/C_n^{(\alpha+1)}(1) = \sum_{k=0}^n c_k(n, \alpha) \cos (n-2k)\theta.
	\end{equation}
	The right side can be clearly expressed as $e^{-in\theta}p_{n,\alpha}(e^{2i\theta})$, because the coefficients $c_k(n,\alpha)$ are real and satisfy the relation $c_k(n,\alpha) = c_{n-k}(n,\alpha)$.
\end{proof}

The next property of the optimal polynomials
$p_{n,\alpha}(z)$ allows us to strengthen Theorem \ref{thm5.2} 
(see Adenstedt \cite{Ad}).
\begin{lem}
	\label{lem6.2}
	For $\delta >0,n^{\alpha+1}\max_{\delta \leq |\lambda| \leq \pi}|p_{n,\alpha}(e^{i\lambda})|$ is bounded uniformly in n.
\end{lem}

\begin{thm}
	\label{thm6.1}
	The conclusions of Theorem \ref{thm5.2} remain valid if the condition therein that $g(\lambda) \in ZB_0^+$ is replaced by the weaker condition that $g(\lambda)\in ZL_0^+$.
\end{thm}
\begin{proof}
In the proof of Theorem \ref{thm4.1}, the boundedness of the function $g(\lambda)$  is relevant only through Lemmas \ref{lem4.3} and \ref{lem4.4}. These lemmas would still hold for an unbounded $g(\lambda)$ (in $L_0$) if condition (\ref{4.5}) could be replaced by the stronger requirement that 
\begin{equation}
	\label{6.7}
	\lim_{n \rightarrow \infty}\max_{\delta \leq |\lambda|\leq\pi}K_n(\lambda)=0, \quad \delta >0,
\end{equation}
where  the kernel $K_n(\lambda)$ is defined by equation \eqref{K-n}.  
It is straightforward to verify that, based on equations \eqref{K-n}, \eqref{5.11}, and Lemma \ref{lem6.2}, 
the kernel $K_n(\lambda)$ satisfies condition (\ref{6.7}), and thus the result follows.
\end{proof}

\n {\bf Estimating the zero-order.} When using the estimator \(\widehat{m}_{\alpha}\) from Theorem \ref{thm5.2} for a spectral density that approaches zero at the origin, there is a risk of inaccurately estimating the true order \(\alpha\) of the zero. 
In some cases, we may either overestimate or underestimate the zero-order. Below, we discuss these scenarios.

{\em Overestimating the zero-order.} Assume that the true spectral density is $f_\alpha(\lambda)$ but we use the estimator $\widehat{m}_{\alpha+\beta}$ having coefficients $c_k(n, \alpha+\beta)$, given by equation \eqref{5.7}, where $\beta$ is restricted to the nonnegative integers (the case of the non-integral $\beta$ remains open). 
We define the {\it asymptotic efficiency} (if exists) of the estimator $\widehat{m}_{\alpha+\beta}$ relative to $\widehat{m}_\alpha$ 
for the spectral density $f_\alpha(\lambda)$ by (cf. \eqref{eff1}
and \eqref{eff2}):
\begin{equation}
	\label{7.1}
	e(\alpha,\beta) = \lim_{n \rightarrow \infty}\sigma_n^2(f_\alpha)/\Var_{f_\alpha}(\widehat{m}_{\alpha+\beta}).
\end{equation}
The next theorem provides an expression for the asymptotic efficiency $e(\alpha,\beta)$.
\begin{thm}
	\label{thm7.1}
For $\alpha > -{1}/{2}$ and $\beta \geq 0$ an integer, the following equation holds:
	\begin{equation}
		\label{7.8}
		e(\alpha, \beta) = \frac{\Gamma(2\alpha+2)\Gamma^2(\alpha+\beta+1)\Gamma(2\alpha+4\beta+2)}{{2\beta \choose \beta}\Gamma(\alpha+1)\Gamma(\alpha+2\beta+1)\Gamma^2(2\alpha+2\beta+2)}.
	\end{equation}
\end{thm}

The efficiency of $\widehat{m}_{\alpha+\beta}$ for a more general spectral density $f_\alpha(\lambda)g(\lambda)$ may be treated by 
using an integral kernel argument presented in Section \ref{FSI}. Moreover, in this general case, the estimator $\widehat{m}_{\alpha+\beta}$ has the asymptotic efficiency $e(\alpha,\beta)$ as in \eqref{7.8}. 
Specifically, we have the following result (see Adenstedt \cite{Ad}). 
\begin{thm}
\label{thm7.2}
Let the noise $\{Y_t\}$ have spectral density $f(\lambda)=f_\alpha(\lambda)g(\lambda)$, where $g(\lambda) \in ZL_0^+$, $g(0) > 0$, and $\alpha > -{1}/{2}$. Then for any integer $\beta \geq 0$,
	\begin{equation}
		\label{7.9}
		\Var_{f}(\widehat{m}_{\alpha+\beta}) \sim g(0)\sigma_n^2(f_\alpha)/e(\alpha,\beta)\q {\rm as}\q n \rightarrow \infty.
	\end{equation}
and the estimator $\widehat{m}_{\alpha+\beta}$ has the asymptotic efficiency $e(\alpha,\beta)$ as in \eqref{7.8}.
\end{thm}

\begin{rem}
{\rm Since $e(\alpha, \beta)>0$, we see that overestimation still yields an estimator whose variance decreases at the optimal rate. From equation (\ref{7.8}) we observe: (a) $e(\alpha,\beta)$ decreases to $1/{2\beta \choose \beta}$ as $\alpha \rightarrow \infty$, and 
(b) $e(\alpha,\beta) \rightarrow 1$ as $\alpha \rightarrow -{1}/{2}$. Additionally, we note that $e(0,1)={5}/{6}$, which was previously stated in Theorem \ref{Vit7} (see also Vitale \cite{Vit}).}
\end{rem}

\n {\em Underestimating the zero-order.} We now consider the case when the estimator $\widehat{m}_\alpha$, where $\alpha$ is an integer, is used, but the true spectral density has a higher-order zero at the origin. In this scenario, which is the opposite of the one discussed above, we have the following theorem
(see Adenstedt \cite{Ad}).
\begin{thm}
	\label{thm8.1}
	Let $\alpha$ be a nonnegative integer and let the true spectral density be $f(\lambda) = f_{\alpha+1}(\lambda)g(\lambda)$, where $g(\lambda) \in L_0$. Then
	\begin{equation}
		\label{8.1}
		\lim_{n \rightarrow \infty}n^{2\alpha+2} \Var_{f}(\widehat{m}_\alpha)=[(2\alpha+1)!/\alpha!]^2\pi^{-1}\int_{-\pi}^\pi g(\lambda)d\lambda.
	\end{equation}
\end{thm}
Thus, if $f(\lambda)=f_{\alpha+\beta+1}(\lambda)g(\lambda)$ with $\beta \geq 0, g(\lambda) \in ZL_0^+$ and $g(0) >0$, we can conclude (using Theorem \ref{thm6.1}) that the efficiency of $\widehat{m}_\alpha$ is $O(n^{-2\beta-1})$. Consequently, the estimator $\widehat{m}_\alpha$ is far from efficient. 
Note that the relation (\ref{8.1}) 
was previously obtained by Vitale \cite{Vit}) for $\alpha = 0$
(see Theorem \ref{Vit5}).
\begin{rem}
{\rm The types of "multipliers" of the function $g(\lambda)$ allowed in Theorems \ref{thm4.1} and \ref{thm5.2}, as defined in equation \eqref{3.6g}, may appear overly complicated. However, we want to emphasize that only the continuity and behavior of the spectral density near $\lambda=0$ are important. Discontinuities or isolated zeros that occur away from the origin are not relevant.}
\end{rem}

\s{Efficiency of the least squares estimator for long-memory models}
\label{EFF}

In this section, we utilize the results from Sections \ref{AD} and \ref{AD-1} to examine the efficiency of the least squares estimator (LSE) $\bar{X}_n$ when the observed process 
$X(t)$ has a spectral density of the form $\lambda ^{2\alpha}L(\lambda)$ as $\lambda \rightarrow 0$, where $L(\lambda)$ is slowly varying function at the origin with $0 < L(0) < \infty$. 
Recall that a function $L(\lambda)$ is slowly varying at 0 if, for all $\alpha > 0$, the limit of $L(\lambda \alpha)/L(\lambda)$ approaches 1 as $\lambda\to0$ (see Section \ref{S2.5}).
 
The efficiency and asymptotic efficiency of the least squares estimator $\bar X_n$ are defined as follows (see equations \eqref{eff1} and \eqref{eff2}):
\beq
\label{eff1X}
e(n, f,\bar X_n): = \frac{Var_f(\widehat{m}_{BLU})}{\Var_f(\bar X_n)},
\eeq
and
\beq
\label{eff2X}
e(\infty, f, \bar X_n): = \displaystyle{\lim_{n \to \infty}} e(n, f,\bar X_n)),
\eeq
where $\widehat{m}_{BLU}$ is the BLUE for $m$. 
The estimator $\bar X_n$ is considered efficient or asymptotically efficient if $e(n, f, \bar X_n) = 1$ or $e(f, f, \bar X_n) = 1$, respectively.

\sn{Efficiency of the least squares estimator for the ARFIMA($0,-\al,0$) model}

Using formulas \eqref{5.2} and \eqref{5.7est} -- \eqref{5.8}, a straightforward algebraic manipulation yielded the following result, which is useful for numerical evaluations. 
For details, see Samarov and Taqqu \cite{ST}.
%and the expression 
%\begin{equation}
%	\label{6}
%	e(n,f_\alpha)=\frac{n^2B(n, 2\alpha+1)}{B(\alpha+1,\alpha+1)\sum_{j=1}^{n}\sum_{k=1}^{n}r_\alpha(j-k)},
%\end{equation}
 
\begin{thm}
Let the spectral density $f_\alpha(\la)$ be as defined in equation \eqref {5.1} with $\alpha > -{1}/{2}$. The efficiency $e(n, f_\alpha,\bar{X}_n)$ of $\bar{X}_n$ is given by the following expression:
	\begin{equation*}
		\begin{split}
			\label{5}
			e(n, f_\alpha,\bar{X}_n) &=
			\bigg\{\bigg(1+\frac{2\alpha}{2}\bigg) \cdots \bigg(1+\frac{2\alpha}{n}\bigg)\\
			&\times \bigg[1-2\alpha \bigg(\frac{1-1/n}{1+\alpha}+\sum_{k=2}^{n-1}\bigg(1-\frac{k}{n}\bigg)\frac{(k-1-\alpha)\cdots(1-\alpha)}{(k+\alpha)\cdots(1+\alpha)}\bigg)\bigg]\bigg\}^{-1}.
		\end{split}
	\end{equation*}
\end{thm}

\sn{Efficiency of the least squares estimator for the class $\mathcal{F}_\al(g)$}

The next theorem provides sufficient conditions for asymptotic efficiency of the LSE $\bar X_n$ in the class $\mathcal{F}_\al(g)$.  
\begin{thm}
	\label{thm2}
Let the noise $\{Y_t\}$ have spectral density $f(\lambda)=f_\alpha(\lambda)g(\lambda)$,
where $g(\lambda) \in ZL_0^+$ (see Definition \ref{def3.3+}) is slowly varying at 0, and $g(0) > 0$. Let $r_\alpha(k)$ and $R_\al(k)$ be the covariance functions corresponding to the spectral densities $f_\alpha(\lambda)$ and $f(\lambda)$,
respectively. Assume that
	\begin{equation}
		\label{3.2}
		R_\al(k) \sim r_\alpha(k)g(0) \q {\rm as}\q k \rightarrow \infty.
	\end{equation}
	The following assertions hold.
	\begin{itemize}
		\item[(i)] If $\alpha\in(-{1}/{2},{1}/{2})\setminus\{0\}$, then
		\begin{equation}
			\label{3.3}
			e(\infty, f,\bar X_n) = \frac{\pi \alpha(1-2\alpha)}{B(\alpha+1, \alpha+1) \sin \pi \alpha}.
		\end{equation}
			\item[(ii)] If  $\alpha \geq {1}/{2}$, then 
			$e(\infty, f,\bar X_n) = 0.$
		\end{itemize}
\end{thm}

\begin{rem}
{\rm(1) Note that $lim_{\alpha \rightarrow 0}e(\infty, f,\bar X_n) = 1$. If $\alpha = 0$ and $g(\lambda)$ is piecewise continuous, continuous at 0 and $0 < g(\lambda) < \infty$, then $\bar X_n$ is asymptotically efficient, that is, $e(\infty, f,\bar X_n)= 1$ (see also Theorem \ref{Gr3}).\\
	%Grenander and Rosenblatt, 1957, Chapter 7).\\
	(2) Part (ii) of Theorem \ref{thm2} implies that the LSE has a slower rate of convergence than the BLUE when the spectrum has a zero at the origin of order 1 or greater (see also Section \ref{Vitale}).}
\end{rem}

\begin{proof}[Proof of Theorem \ref{thm2}]
	Since $g(0) > 0$, $g \in {Z}L_0^+$, and $\alpha > -{1}/{2}$, we can apply Theorem \ref{thm6.1} to conclude that (see equation \eqref{5.12}):
	\begin{equation}
		\label{3.5-T}
		\Var(\widehat{m}_{f}) \sim n^{-2\alpha-1}\frac{\Gamma(2\alpha+1)}{B(\alpha+1,\alpha+1)}g(0)
		\q {\rm as}\q n \rightarrow \infty.
	\end{equation}
	On the other hand, we have
	$$\Var_f(\bar{X}_n) = \frac{1}{n^2} \sum_{j=1}^{n} \sum_{k=1}^{n} R_\alpha(j-k)=\frac{1}{n^2} \bigg\{R_\alpha(0)+\sum_{j=1}^{n-1}A_\alpha(j) \bigg\},$$
	where $A_\alpha(j):=\sum_{k=-j}^{j}R_\alpha(k)$.
	
	If $-{1}/{2} < \alpha < 0,$  then based on equations  \eqref{2-r} and  \eqref{3.2}, we have
	$$A_\alpha(j) \sim (2C_\alpha /(-2\alpha))j^{-2\alpha}g(0) \q {\rm as}\q j \rightarrow \infty,$$ 
	where the constant $C_\alpha$ is given in equation (\ref{2}).

 If $0 < \alpha< {1}/{2}$, then we have 
	\begin{align}
		\label{3.6}
		\nonumber
		A_\alpha(j) &= \sum_{k=-j}^{j}R_\alpha(k)-\sum_{k=-\infty}^{+\infty}R_\alpha(k)\\
		& = 
		-\sum_{|k|\geq j+1}R_\alpha(k)\sim \frac{2C_\alpha}{-2\alpha}j^{-2\alpha}g(0)  \q {\rm as}\q j \rightarrow \f.
	\end{align}
	
	Thus for $-{1}/{2} < \alpha < {1}/{2}$, $\alpha \neq 0$,
	$$\Var_f(\bar{X}_n) = n^{-2\alpha-1}\frac{-C_\alpha}{\alpha(1-2\alpha)}g(0)\q {\rm as}\q n \rightarrow \infty$$
	
	If $\alpha = {1}/{2}$, the relation in (\ref{3.6}) still holds, and hence 
	$$\Var_f(\bar{X}_n) \sim -2C_{1/2}g(0)n^{-2}\ln n.$$ 
	
	Finally, if $\alpha > {1}/{2}$, the sum $\sum_{j=1}^{n-1}A_\alpha(j)$ converges as $n \rightarrow \infty$. Hence, we have
	$$\Var_f(\bar{X}_n)\sim n^{-2}\bigg\{R_\alpha(0)+\sum_{j=1}^\infty A_\alpha(j) \bigg\},$$
	where the expression in the braces is finite and does not depend on $n$.
	
	We can now evaluate the asymptotic efficiency $e(\infty, f,\bar X_n) $.
	By using equations (\ref{2}), \eqref{eff1X}, and \eqref{eff2X}, we derive the relation in (\ref{3.3}) when $\alpha\in(-{1}/{2},{1}/{2})\setminus\{0\}$, and $e(\infty, f,\bar X_n) =0$ when $\alpha \geq {1}/{2}$.
\end{proof}

To apply Theorem \ref{thm2}, we need to verify condition (\ref{3.2}). This can be accomplished in various scenarios. For example, consider the class \textbf{\textit{Z}} of functions defined as follows:
\begin{center}
	\textbf{\textit{Z}}:=\{$g$: $g$ is of bounded variation and slowly varying at 0 in the sense of Zygmund\}.
\end{center}
For the definition and properties of slowly varying functions, see
Section \ref{S2.5}.

As an immediate consequence of Theorem \ref{thm2}, we have the following result.
\begin{cor}
	\label{Cor1}
	The conclusion of Theorem \ref{thm2} holds in the following cases.\\
	(1) For $-{1}/{2} <\alpha<0$, if $g(0)>0$, $g(\lambda)$ continuous and $g \in \textit{\textbf{Z}}.$\\
	(2) For $0< \alpha <{1}/{2}$, if $g(0) > 0$, g differentiable and its derivative  $g' \in \textit{\textbf{Z}}$.\\
	(3) For $\alpha \geq {1}/{2}$, $g(\lambda)$ is continuous at 0, and
	$0 < c_1 \leq g(\lambda) \leq c_2 < \infty$.
\end{cor}

%\begin{proof}
%The relation (\ref{3.2}) follows from Proposition \ref{Zg-1} in the case (1), and also in the case (2) if one integrates $R_\alpha(k) = \int_{-\pi}^{\pi}e^{i\lambda k}f(\lambda)d\lambda$ by parts. 
%	As for assertion (3), note that 
%	$$c_1\sum_{j=1}^{n} \sum_{k=1}^n r_\alpha (j-k) \leq \sum_{j=1}^{n} \sum_{k=1}^{n}R_\al(j-k)\leq c_2 \sum_{j=1}^{n} \sum_{k=1}^n r_\alpha (j-k).$$
%\end{proof}
\begin{rem}
	{\rm Samarov and Taqqu \cite{ST} conjectured that Theorem \ref{thm2} holds if the condition \(0 < g(0) < \infty\) is replaced by \(g \not\equiv 0\) in some interval \([0, \delta]\), and if the condition (\ref{3.2}) is changed to 
		
		\[
		R_\alpha(k) \sim r_\alpha(k) g(1/k) \quad \text{as} \quad k \rightarrow \infty.
		\]
		
		This modification would allow for the consideration of functions such as \(g(\lambda) = |\log(\lambda)|\) or \(g(\lambda) = |\log(\lambda)|^{-1}\).}
\end{rem}
\begin{rem}
{\rm Conditions (1) and (2) in Corollary \ref{Cor1} hold for ARFIMA$(p,-\al,q)$ and fractional Gaussian noise models (see Section \ref{mem}). The asymptotic efficiency
of the LSE for the values of $\al$ specified in Corollary \ref{Cor1} has been analyzed in the paper by Samarov and Taqqu \cite{ST}.}
\end{rem}

Denote by $\bold{C}^{1,\g}$ ($0<\g\leq 1$) the class of differentiable functions whose derivative is H\"older-continuous of order $\g$.

Beran and K\"unsch \cite{BK}, utilizing the results of Bleher \cite{Bl}
on the inversion of Toeplitz matrices, proved the following result.
\begin{thm}
\label{BK-T1}
Let the spectral density $f(\la)$ belong to the class $\mathcal{F}_\al(g)$ with $-{1}/{2} < \alpha < 0$ (see equation \eqref{Fal-1}). Assume that the short-memory factor $g(\la)$ is
from the class $\bold{C}^{1,\g}$. Then the limit in equation \eqref{eff2X}
exists and 
\beq
\label{ef-BK}
e(\infty, f, \bar X_n) = 1- (1-\pi^2/12)(2\al)^2 + o(\al^2) \q {\rm as}\q \al\uparrow 0. 
\eeq
\end{thm}
Observe that the spectral density of the fractional Gaussian noise with Hurst index $H$, given by equation \eqref{MT2}, satisfies the conditions of Theorem \ref{BK-T1} with $\al =1/2-H$ and $\g=1$.
Additionally, Beran and K\"unsch \cite{BK} observed that the asymptotic relation
in \eqref{ef-BK} remains valid also in the case where $\alpha > 0$
as $\al\downarrow 0$.
In this case, the  the spectral density $f(\la)$ has a zero at
$\la=0$ and is not differentiable there.
\begin{rem}
	{\rm In Sections \ref{ND-BLU} and \ref{EFF}, we have discussed spectra that near $\lambda=0$ behave like $\lambda^{2\alpha}$ with $\al>-1/2$.  It is important to note that spectra corresponding to $\alpha > 1$ are unlikely to occur in practice, except possibly through difference filtering. In contrast, spectra with $-{1}/{2} < \alpha \leq 1$ appear to be practically useful.}
\end{rem}

\begin{rem}
	\label{r6.5}
	{\rm The results presented in Sections \ref{ND-BLU} and \ref{EFF} regarding
		the model $X(t)=m+Y(t)$ can be easily extended to the following simple regression model: 
		$$X(t)=me^{i\la_0 t}+Y(t),\q t\in\mathbb{Z},$$ 
		where $\la_0$ is a fixed constant in $[\pi, \pi]$. 
		
		Indeed, by defining $\widetilde{X}(t):=X(t)e^{-i\la_0 t}$ and $\widetilde{Y}(t): = Y(t) e^{-i\la_0 t}$, this new model reduces to the original: $\widetilde{X}(t) = m+ \widetilde{Y}(t)$. If $Y(t)$ has a spectral density $f(\lambda)$, then $\widetilde Y(t)$ has a spectral density given by $\widetilde{f}(\lambda) = f(\lambda + \la_0)$, where the definition of $f(\lambda)$ is extended outside the interval $[-\pi, \pi]$ by periodicity.
		In this new model, the behavior of \( f(\lambda) \) near \( \lambda = \la_0 \) becomes particularly significant.}
\end{rem}

\s{Extensions}
\label{Ext}

In this section, we present extensions of the results discussed in the previous sections for the model \(X(t) = m + Y(t)\), where \(t \in \mathbb{Z}\) (see \eqref{1.0}). We specifically focus on estimating the mean for continuous time processes and for models characterized by their spectral functions. Furthermore, we explore the estimation procedure within a Hilbert space framework, introduce the concept of pseudo-best estimators, and address their efficiency.

\sn{Estimating the mean for processes specified by the spectral function} 
We revisit the model specified by equation \eqref{1.0}, defined as \( X(t) = m + Y(t) \) for \( t \in \mathbb{Z} \), where \( Y(t) \) is a zero-mean, stationary process. In Sections \ref{ND-BLU} and \ref{EFF}, we assumed that the model's spectrum is absolutely continuous, characterized by the spectral density function \( f(\lambda) \).

But what happens when the spectrum also includes a singular component?

Specifically, we examine the case in which the model is described by the spectral function \( F(\lambda) \), which contains both an absolutely continuous component 
\( f(\lambda) \) and a singular part. This can be expressed as (see equation \eqref{di1SF}):
\beq
\label{di1SF-1}
dF(\lambda) = f(\lambda)d\lambda + dF_s(\lambda), \quad \lambda \in [-\pi, \pi],
\eeq
where \( dF_s(\lambda) \) is singular with respect to \( d\lambda \).

We are particularly interested in the asymptotic behavior of the variance of the BLUE and the efficiency of the LSE $\widehat m_{LS}=\bar X_n$ relative to the BLUE based on the sample $\{X(t), t=0,1,\ldots,n\}$. 

The following theorem demonstrates that the efficiency of the LSE relative to the BLUE, even in this broader context, is strongly influenced by the behavior of the spectral function at the origin (see Adenstedt and Eisenberg \cite{AdE}).
\begin{thm}
	\label{AE-T3}
	Let the noise process $Y(t)$ have spectral function $F(\la)$
	given by equation \eqref{di1SF-1}. The following assertions hold.
	\begin{itemize}
		\item[(a)]  $\lim_{n\to\f}\Var(\widehat m_{LS}) =\lim_{n\to\f}\Var(\widehat m_{BLU})=2\pi dF(0)$.
		\item[(b)] Assume also that $f(\lambda)$ is positive and continuous at $\la=0$, that $F_s(\la)-F_s(-\la)$ is constant in 
		$0<|\la|<\de$ for some $\de>0$, and that
		\beq
		\label{A22E}
		\int_{-\pi}^\pi\frac{|p(\la)|^2}{f(\la)}d\la<\f
		\eeq
		for some trigonometric polynomial $p(\la)$. Then as $n\to\f$,
		\beq
		\label{A23E}
		\Var(\widehat m_{LS})-dF(0)\sim\Var(\widehat m_{BLU})-dF(0)
		\sim 2\pi f(0)/n.
		\eeq
	\end{itemize}
\end{thm}
Note that the proof of Theorem \ref{AE-T3} presented in Adenstedt and Eisenberg \cite{AdE} does not require knowledge of the covariance function of $Y(t)$.
\begin{rem}
	{\rm The condition \eqref{A22E} essentially states that the function $1/f(\lambda)$ is integrable, but allows the spectral density
		$f(\lambda)$ to have polynomial zeros that can be removed by $|p(\la)|^2$. 
		It can be shown that \eqref{A22E} implies that the model process $X(t)$ is nondeterministic.}
		%The condition on singular component $F_s$, certainly allows superposition of a finite discrete spectrum on an absolutely continuous spectrum.}
\end{rem}

\sn{Estimating the mean of continuous time processes} 
In this section, we explore the estimation of the unknown mean for continuous time processes. Specifically, we will now consider the following model:
\begin{equation}
	\label{C-1.0}
	X(t)=m+Y(t), \q t\in \mathbb{R},
\end{equation}
where $m$ is the constant unknown mean of $X(t)$, and the noise $Y(t)$ is assumed to be a centered mean-square continuous stationary process. Additionally, we assume that $Y(t)$ has an absolutely continuous 
spectrum and is nondeterministic. In this context, $Y(t)$ has a spectral density $f(\la)$, $\la\in \mathbb{R}$, satisfying the following condition: 
\begin{equation}
	\label{cdet}
	\int_{\mathbb{R}} \frac{\ln f(\la)}{1+\la^2}d\la >-\f.
\end{equation}
Then, $Y(t)$ can be represented as a linear process with a time-invariant filter $g(t)$ (see, e.g., Gihman and Skorokhod \cite{GiSk}, Section 4.7):  
\begin{equation}
	\label{clp}
	Y(t)=\int_{-\f}^t g(t-s)d\xi(s), \q
	\int_{\mathbb{R}}|g(s)|^2ds < \f,
\end{equation}
where $\{\xi(s), s\in \mathbb{R}\}$ is a real-valued orthogonal increment process satisfying
${\Exp}|d\,\xi(s)|^2 = \si^2ds$.	

The problem of interest is the estimation of the unknown mean $m$ using unbiased linear estimators $\widehat{m}_T$, based on a random sample 
\beq
\label{C-Sam}
\{X(t), \, 0\leq t\leq T\},
\eeq
%$\{X(t), 0\leq t\leq T\}$, 
and the efficiency of the LSE relative to the BLUE. 

The next theorem provides a formula for BLUE of the unknown mean $m$
for a class of nondeterministic processes (see Grenander \cite{Gr-4}, p. 189, and Grenander and Szeg\H{o} \cite{GS}, Section 11.4).
\begin{thm}
	\label{C-GS-1}
	Let the noise process $Y(t)$ be a nondeterministic linear process with a time-invariant filter $g(t)$ and spectral density $f(\la)$
	satisfying the following conditions:
	\begin{itemize}
		\item[(a)] The spectral density $f(\la)$ is positive at $\la=0$, i.e., $f(0)>0$.
		\item[(b)] The filter $g(t)$ is such that $g(t)=(1+t^2)^{-1}\cd O(1)$ as $t\to\f$.
	\end{itemize}
	Then, there exists a constant $\g$ such that
	\beq
	\label{C-G-S-3}
	\widehat m_{BLU}=\g\int_{-\infty}^{\infty}\widehat X(t)dt,
	\eeq
	where $\widehat X(t)$ is the projection of $X(t)$ onto the space
	$L^2_X(0,T)$.
\end{thm}

In the continuous time case, the LSE $\widehat m_{LS}$ 
of $m$ is given by
\beq
\label{C-LSE}
\widehat m_{LS}=\bar X=\frac1T\int_0^TX(t)dt.
\eeq

The next theorem shows that the LSE is asymptotically efficient relative to BLUE (see Grenander \cite{Gr-4}, p. 248, and Grenander and Szeg\H{o} \cite{GS}, Section 11.4).
\begin{thm}
	\label{C-GS-2}
	Under the conditions of Theorem \ref{C-GS-1}, the LSE $\widehat m_{LS}$ of $m$ defined as in \eqref{C-LSE} is asymptotically efficient relative to BLUE $m_{BLU}$ of $m$. Specifically, the following relations hold as $T\to\f$:	
	\beq
	\label{C-G-S-1}
	\Var(\widehat m_{LS})\sim\Var(\widehat m_{BLU})\sim 2\pi f(0)/T.
	\eeq
\end{thm}
\begin{rem}
	\label{r7.2}
	{\rm The results of Theorems \ref{C-GS-1} and \ref{C-GS-2} remain valid if condition (b) in \ref{C-GS-1} is replaced by the following (see Grenander \cite{Gr-4}, p. 248):
	\beq
	\label{C-G-S-12}
	\int_0^\f u^3|g(u)|^2du<\f.
	\eeq
	}
\end{rem}
Now, assume that the noise $Y(t)$ in \eqref{C-1.0} is specified by the spectral function $F(\la)$ of the form \eqref{di1SF-1} with
$\la\in\mathbb{R}$.

The following theorem is a continuous analogue of Theorem \ref{AE-T3}  (see Adenstedt and Eisenberg \cite{AdE}).

\begin{thm}
	\label{C-AE-T3}
	Let the noise process $Y(t)$ have spectral function $F(\la)$
	given by equation \eqref{di1SF-1} with $\la\in\mathbb{R}$. 
	The following assertions hold.
	\begin{itemize}
		\item[(a)]  $\lim_{T\to\f}\Var(\widehat m_{LS}) =\lim_{T\to\f}\Var(\widehat m_{BLU})=dF(0)$.
		\item[(b)] Assume in addition that $f$ is positive and continuous at $\la=0$, that $F_s(\la)-F_s(-\la)$ is constant in 
		$0<|\la|<\de$ for some $\de>0$, and that
		\beq
		\label{C-A22E}
		\int_{-\pi}^\pi\frac{|p(\la)|^2}{(1+\la^2)^nf(\la)}d\la<\f
		\eeq
		for some positive integer $n$ and function $p(\la)$ of the form:
		$$
		p(\la)=\sum_{k=1}^\f b_k\exp(-it_k\la).
		$$
		Then, as $T\to\f$,
		\beq
		\label{C-A23E}
		\Var(\widehat m_{LS})-dF(0)\sim\Var(\widehat m_{BLU})-dF(0)
		\sim 2\pi f(0)/T.
		\eeq
	\end{itemize}
\end{thm}
\begin{rem}
	{\rm The condition \eqref{C-A22E} requires that the spectral density \( f(\lambda) \) does not decrease too rapidly as \( |\lambda| \) approaches infinity. It also allows for the presence of zeros in \( f(\lambda) \) that can be removed by \( |p(\lambda)|^2 \). Additionally, it can be shown that \eqref{C-A22E} implies that the model process \( X(t) \) is nondeterministic.}
\end{rem}

\sn{Estimating the mean within a Hilbert space framework} 
In this section, we will examine the estimation problem within a Hilbert space framework, without imposing restrictions on the index set or the set of observations.
Specifically, we consider the model:
 \begin{equation}
 	\label{E1.0}
 	X(t)=m\be(t)+Y(t), \q t\in \mathbb{U},
 \end{equation}
 where $\mathbb{U}$ is an arbitrary index set, $m$ is the unknown regression constant, $\be(t)$ is a known regression term, and the noise $Y(t)$
 is a zero-mean, second-order process with a covariance function $R(t,s)=\Exp[Y(t)Y(s)]$, $t,s\in \mathbb{U}$.
  
 The problem of interest is finding the BLUE for the unknown regression constant
 $m$ based on a random sample $\{X(t), t\in \mathbb{S}\}$, where $\mathbb{S}$ 
 is a fixed subset of $\mathbb{U}$.
 
Note that the previous sections discussed the cases where $\mathbb{U}=\mathbb{Z}$ or $\mathbb{R}$, and $\mathbb{S}=\{0,1,\ldots,n\}$ or $[0,T]$,  with $\be(t)\equiv 1$. 

Denote by $\Ph_m$ the probability measure on the sample path space $\Omega$ 
for a process that is equal in law to $m \be(t)+Y(t)$. Then, for $\omega \in \Omega$, 
we regard $X(t, \omega) =Y(t, \omega) = \omega(t)$ as the evaluation
of the sample path at point $t$. The processes $Y(t)$ and $X(t)$ differ only in that $Y(t)$ 
is treated as an element of the space $L^2(\Ph_0)$, whereas $X(t)$ is considered 
as an element of the space $L^2(\Ph_m)$. This approach allows us to focus on a single set of sample paths. 

An estimator $\widehat{m}(\omega)$ 
for $m$ is now a function of the sample path. It is unbiased if 
$\int_\Omega\widehat m(\omega)d\Ph_m = m$, and it is linear if it lies within the  
subspace $L^2_\mathbb{S}(X)$ of $L^2(\Ph_m)$, which is spanned by the family 
$\{X(t) : t \in \mathbb{S}\}$. (For the definition of the space $L^2_\mathbb{S}(X)$, we refer to Section \ref{HSR}.)

We first present an extension of Proposition \ref{L0}, which provides the solution 
of Szeg\H{o}'s extremum problem in an abstract Hilbert space setting
(see, e.g., Simon \cite{Sim-1}, p. 17).
\begin{pp}
	\label{AL0}
Let \( H \) be a Hilbert space equipped with an inner product \( (\cdot, \cdot) \) and a norm \( \| \cdot \| \). 
Let $\varphi\in H$ be a non-zero vector. Then, the following holds:
	\beq
	\label{Aop17}
	\min\{||\psi||^2: \,\, (\psi,\varphi)\} =1\}=\frac1{||\varphi||^2}.
	\eeq
The vector \( \psi \) that achieves the minimum in \eqref{Aop17} is given by  $\psi=\varphi/||\varphi||^2$. 
If $\psi$ is restricted to lie in a closed subspace $H_1$, and $P$ is the projection onto $H_1$, then the minimum becomes  $1/{||P\varphi||^2}$.
\end{pp}
\begin{proof}
Let $\psi_0=\varphi/||\varphi||^2$. If $(\psi,\varphi) =1$, then we have 
$(\psi-\psi_0,\varphi) =0$. Consequently, it follows that $||\psi||^2 =||\psi_0||^2+||\psi-\psi_0||^2$. Thus, the minimum occurs at $\psi=\psi_0$, sine $||\psi_0||^2=1/{||\varphi||^2}$, thereby proving the result.

For the case of subspace, make $H_1$ the Hilbert space and note that for $\psi\in H_1$,
we have $(\psi,\varphi) =(\psi,P\varphi) =1$, and $P\varphi\in H_1$. 
This completes the proof.
\end{proof}

Using Proposition \ref{AL0}, we can now state the following result regarding the BLUE for the unknown parameter \(m\) in the model described in \eqref{E1.0} (see also Adenstedt and Eisenberg \cite{AdE}).
\begin{thm}
	\label{AE-T1}
Let the model $X(t)$ be defined as in equation \eqref{E1.0} with the unknown parameter \(m\). 
Then the BLUE $\widehat{m}_{BLU}$ for $m$, based on observations $\{X(t), t\in \mathbb{S}\}$, satisfies the following relationships:
	\beq
\label{A17E}
\widehat{m}_{BLU}=\frac{\varphi(\omega)}{||\varphi||^2}, \q
\Var(\widehat{m}_{BLU})=\frac{1}{||\varphi||^2},
\eeq
where $\varphi$ is the solution of the equation
\beq
\label{A18E}
\Exp[Y(t)\varphi]=\be(t), \q t\in \mathbb{S}, \, \, \varphi\in H_\mathbb{S}(Y).
\eeq
 If the equation in \eqref{A18E} has no solution, then $\widehat{m}_{BLU}=m$, indicating that \( m \) is perfectly estimable.
\end{thm}

\begin{rem}
{\rm Theorem \ref{AE-T1} can be used to show that the BLUE $\widehat{m}_{BLU}$ is also the maximum likelihood estimator when the process $Y(t)$ is Gaussian  (see Adenstedt and Eisenberg \cite{AdE} for details)}.
\end{rem}

Calculating the BLUE $\widehat{m}_{BLU}$, or equivalently solving equation \eqref{A18E}, is often a challenging problem in most practical situations. 
Therefore, appropriate approximations are necessary. 
We outline a procedure that, in the spirit of the Cram\'er-Rao inequality,  
provides a lower bound for the variance $\Var(\widehat{m}_{BLU})$. 

Denote by $L^2_\mathbb{U}(Y)$ the Hilbert space spanned by the family $\{Y(t), t\in\mathbb{U}\}$. Note that $L^2_\mathbb{S}(X)$ is a subspace of $L^2_\mathbb{U}(Y)$. 
While solving the equation \eqref{A18E} may be challenging, it may be easier to construct a function \(\tilde\varphi \in L^2_\mathbb{U}(Y)\) that satisfies the equation:
\beq
\label{A19E}
\Exp[Y(t)\tilde\varphi]=\be(t), \q t\in \mathbb{S}.
\eeq

It is important to note that such a function \(\tilde\varphi\) exists if and only if \(m\) is not perfectly estimable from observations over \(\mathbb{S}\). Furthermore, \(\tilde\varphi\) is not unique.

If $P$ is the projection operator onto  $L^2_\mathbb{S}(Y)$ and $\tilde\varphi\in L^2_\mathbb{U}(Y)$ satisfies equation \eqref{A19E}, then we have
\beq
\label{A20E}
\Exp[Y(t)P\tilde\varphi]=\Exp[Y(t)\tilde\varphi]=\be(t), \q t\in \mathbb{S}.
\eeq
This implies that \( \varphi = P\tilde{\varphi} \) is the solution to equation \eqref{A18E}. 

Furthermore, taking into account that $||\varphi||=||P\tilde\varphi||\leq||\tilde\varphi||$,
we can deduce the following result.	
	\begin{thm}
		\label{AE-T2}
If $\tilde\varphi$ is in $H_\mathbb{U}(Y)$ and satisfies equation \eqref{A19E}, then  
\beq
\label{A21E}
\Var(\widehat{m}_{BLU})\geq{1}/{||\tilde\varphi||^2}.
\eeq
If no such $\tilde\varphi$ exists, then	$\Var(\widehat{m}_{BLU})=0$.
\end{thm}

\begin{rem}
	{\rm Using the reproducing kernel Hilbert space (RKHS) technique and the projection theorem in abstract Hilbert spaces (see Section \ref{HSR}), Parzen \cite{P-4} provided a general solution to the problem of minimum variance unbiased linear estimation for the following general model:
		\begin{equation}
			\label{P-E1.0}
			X(t)=m(t)+Y(t), \q t\in \mathbb{U},
		\end{equation}
		where $\mathbb{U}$ is an arbitrary index set,  $Y(t)$
		is a zero-mean, second-order process with a known covariance function $R(t,s)=\Exp[Y(t)Y(s)]$, $t,s\in \mathbb{U}$, and
		the unknown mean-value function $m(t):=\Exp[X(t)]$.
		
		Specifically, let $H(R)$ be the RKHS generated by the
		covariance function $R(t,s)$ (see Section \ref{HSR}). 
		We assume that the unknown mean function $m(t)$ belongs to a known class $M$, which is a subset of $H(R)$. 
		A random variable $(h,X)_R\in L^2_\mathbb{S}(X)$, with $h\in H(R)$, is said to be an unbiased linear estimator for $m(t)$ if the following condition holds:
		\beq
		\label{A19E-1}
		\Exp[(h,X)_R]=(h,m)_R=m(t) \q \text{for all}\q m\in M.
		\eeq
		
		In Parzen \cite{P-4}, it was proven
		that the BLUE $\widehat{m}_{BLU}(t)$ for $m(t)$ is given by the following formula:
			\begin{equation}
			\label{P-E1.02}
		\widehat{m}_{BLU}(t)=(P_{\bar{M}}R(\cd,t), X)_R,
		\end{equation}
		where \(\bar{M}\) represents the smallest Hilbert subspace of \(H(R)\) that contains \(M\), and \(P_{\bar{M}}R(\cd,t)\) is the projection of \(R(\cd,t)\) onto the space \(\bar{M}\).
		
	In particular, for the model described by equation \eqref{E1.0}, when \( m(t) \) is expressed as \( m(t) = m \beta(t) \), where \( \beta(t) \) is a known function and \( m \) is a constant to be estimated based on a random sample \( \{X(t), t \in [0,T]\} \), the following formulas for the BLUE \( \widehat{m}_{BLU} \) of \( m \) and for the variance of the BLUE were derived:
		\beq
		\label{P-A17E}
		\widehat{m}_{BLU}=\frac{(\be,X)_R}{||\be||_R^2}, \q
		\Var(\widehat{m}_{BLU})=\frac{1}{||\be||_R^2},
		\eeq
		where $(\be,X)_R$ is the random variable from $L^2_\mathbb{S}(X)$ corresponding to $\be\in H(R)$, and  $||\cd||_R$ is the the norm of the RKHS  $H(R)$ (for details, see Parzen \cite{P-4}, pp. 459--462).}
\end{rem}

\begin{rem}
	{\rm Consider the model defined by equation \eqref{E1.0} with $\mathbb{U}=\mathbb{R}$ (continuous time), observed for $t \in [0,T]$. We assume that the process $Y(t)$ is mean-square continuous and that the regression function $\be(t)$ lies within the range of the covariance function $R(t,s)=\Exp[Y(t)Y(s)]$. Thus, for some function \(a(s)\), where \(s \in [0,T]\), we have
		\begin{equation}
			\label{C-E11}
			\be(t)=\int_0^TR(t,s)a(s)ds, \q 0\leq t\leq T.
		\end{equation}
		This relation yields
		\begin{equation}
			\label{C-E12}
			\Exp\left[Y(t)\int_0^Ta(s)Y(s)ds\right] =\be(t), \q 0\leq t\leq T.
		\end{equation}
		Therefore, the function
		\begin{equation}
			\label{C-E13}
			\psi_T=\int_0^Ta(s)Y(s)ds
		\end{equation}
		characterizes the BLUE based on the observations $\{X(t), \, 0\leq t\leq T\}$.
		Observe that 
		\begin{equation}
			\label{C-E14}
			||\psi_T||^2=\int_0^Ta(t)\be(t)dt.
		\end{equation}
		Then, for the BLUE of the unknown regression constant $m$ and its variance we have the following formulas:
		\begin{align}
			\label{C-E15}
			&\widehat m_{BLU}=\int_0^Ta(s)X(t)dt\left[\int_0^Ta(t)\be(t)dt\right]^{-1},\\
			\label{C-E16}
			&\Var(\widehat m_{BLU})=\left[\int_0^Ta(t)\be(t)dt\right]^{-1}.
		\end{align}
		These results can be compared to Parzen's estimator, shown in equation \eqref{P-A17E}, where we have \( a(t) = \beta(t) \).}
\end{rem}

\sn{Pseudo-best estimators}

The concept of pseudo-best estimators for the general model of the form:
\begin{equation}
	\label{Roz1}
	X(t)=m(t)+Y(t), \q t\in \mathbb{U},
\end{equation}
where $m(t)$ is an unknown deterministic function and the noise $Y(t)$ is a centered Gaussian stationary process with a spectral
function $F(\lambda)$, $\lambda \in {\Lambda}$,
was introduced by Yu. Rozanov (see Ibragimov and Rozanov \cite{IR}, Section 7.3).
%, p. 245). 

We are focusing on the special case where $m(t) = m$ is constant,  
and $F(\lambda)$ is absolutely continuous
(with respect to the Lebesgue measure) with a spectral density $f(\lambda)$.

The concept of pseudo-best estimators is based 
on the following observation (see Ibragimov and Rozanov \cite{IR}, Section 7.3, and Grenander and Rosenblatt \cite{GRo}, Section 2.1). 

Let \( F(\lambda) \) and \( G(\lambda) \) be two spectral functions, and let \( L^2(F) \) and \( L^2(G) \) represent the corresponding weighted \( L^2 \)-spaces (see Section \ref{dnd} for definitions). For an index set \( I \), we denote by \( H_I(F) \) and \( H_I(G) \) the closed subspaces of \( L^2(F) \) and \( L^2(G) \), respectively, 
spanned by the family of functions \( \{ e^{ik\lambda} : k \in I \} \). 
Specifically, we define:
$$H_I(F): = \ol{sp}\{e^{ik\la}, \, k\in I\}_{L^2(F)} \q \text{and} \q H_I(G): = \ol{sp}\{e^{ik\la}, \, k\in I\}_{L^2(G)}.
$$
Observe that if $G(\la) \leq F(\la)$, then $L^2(F)\subseteq L^2(G)$ and $H_I(F)\subseteq H_I(G)$. 

Let the function $h(\la)$ be an element of both $L^2(F)$ and $L^2(G)$. Consider the errors of approximation:
$$\si^2_F: = \inf_{t\in H_I(F)}||h-t||_{L^2(F)} \q \text{and} \q \si^2_G: = \inf_{t\in H_I(G)}||h-t||_{L^2(G)}.
$$
Then, we have 
$$\si^2_G\leq\si^2_F,$$
since the function that minimizes the error of approximation in \( L^2(F) \) produces an error that cannot be smaller than the error obtained when it is used as an approximation in the space \( L^2(G) \).

We revisit the model \( X(t) = m + Y(t) \), where \( t \in \mathbb{U} \) with \( \mathbb{U} = \mathbb{Z} \) or \( \mathbb{U} = \mathbb{R} \), and \( Y(t) \) is a centered stationary process that has a spectral density \( f(\lambda) \).
\begin{den}
A spectral density $\hat f$ that satisfies the inequality  $f(\la)\leq c \hat f(\la)$ for 
some positive
constant $c$ is referred to as a {\it pseudo-spectral density}.
A {\it pseudo-best estimator}, corresponding to a pseudo-spectral density $\hat f$, is defined as the best linear estimator with respect to $\hat f$. It is also referred to as an $\hat f$-estimator.

\end{den}
Note that for the least squares estimators, $\hat f$= constant; for the BLUE $\hat f =f$; and for
estimators corresponding to spectral densities from the class $\mathcal{F}_\al(g)$
(see equation \eqref {Fal-1}), $\hat f =f_\al:=(2\pi)^{-1}|1 - e^{i\lambda}|^{2\al}$. 

Let $\si^2(\hat f,f)$ denote the variance of the $\hat f$-estimator for $m$, and let $\si^2(f):=\si^2(f,f)$  represent the variance of the BLUE
$\hat m_{BLU}$.

Pseudo-best estimators for regression coefficients
for discrete time models with spectral densities that exhibit polynomial-type zeros were examined in the paper by
Rasulov \cite{Ras}. Specifically, for the class of spectral densities of the form:
\begin{equation}
	\label{Ras1}
	f(\la) =(2\pi)^{-1}\prod_{k=1}^n|e^{i\lambda_k} - e^{i\lambda}|^{2q_k}g(\la)
	=(2\pi)^{-1}|P(e^{i\lambda})|^{2}g(\la),
\end{equation}
where $q_k$ are non-negative integers and $g(\la)$ is a continuous positive
function, it was shown that the pseudo-best estimators corresponding to the pseudo-spectral density 
$\hat f(\la) =(2\pi)^{-1}|P(e^{i\lambda})|^{2}$
are asymptotically efficient relative to the BLUE.

The asymptotic efficiency of pseudo-best estimators for continuous time stationary models was established in the paper by Rasulov and Kholevo \cite{RKh}.

Consider the continuous time model as defined in \eqref{C-1.0}, where the noise process $Y(t)$ has a spectral density given by: 
\begin{equation}
	\label{RK-1}
	f(\la) =(2\pi)^{-1}|\lambda|^{2\al}g(\la) =f_\al(\la)g(\la), \q \al>-1/2,
\end{equation}
where the function $g(\la)$ is bounded, positive and continuous for $\la=0$ but may have
a finite number of power-type zeros at other points.

It is important to note that the class of spectral densities of the form given in \eqref{RK-1} serves as the continuous counterpart to the class \( \mathcal{F}_\alpha(g) \), as defined in equation \eqref{Fal-1}.

The following theorem describes the asymptotic behavior of the variance of the pseudo-best $f_\al$-estimator for an unknown mean $m$ in the model defined by equation \eqref{C-1.0}.
\begin{thm}
	\label{RK-T01}
	Let $f(\la)$ be defined as in equation \eqref{RK-1}, with $g(\la)$ satisfying the conditions stated therein. Then the following relation holds:
	\begin{equation}
		\label{RK-02}
		\lim_{T\to\f}T^{2\al+1}\si^2_T(f,f_\al)=c(\al)g(0),
	\end{equation}
where the constant $c(\al)$ is given by the following formula:
	\begin{equation}
	\label{RK-022}
	c(\al)=(2\al+1)\G^2(2\al+1)\G^{-2}(\al+1).
\end{equation}

\end{thm}

We define the class $\mathcal{G}$ as the set of functions $g(\lambda)$ that satisfy the following condition uniformly in $\lambda$:
$$|P(i\la)Q^{-1}(i\la)|^2g^{-1}(\la)\leq c <\f,$$
where $Q(z)$ is a polynomial with all its zeros in the left half-plane, and $P(z)=\sum_{k=0}^{n}p_k\exp\{zt_k\}$, with $t_k$, 
$k=0, \ldots, n$, being non-negative numbers,  and $P(0)\neq0$.

\begin{thm}
	\label{RK-T1}
		Let $f(\la)$ be defined as in equation \eqref{RK-1}, with $g(\la)$ satisfying the conditions stated therein. If, in addition, $g(\la)\in \mathcal{G}$, 
		then the following relation holds: 
	\begin{equation}
		\label{RK-2}
		\lim_{T\to\f}T^{2\al+1}\si^2_T(f)=c(\al)g(0),
	\end{equation}
	where the constant $c(\al)$ is defined as in \eqref{RK-022}.
	Moreover, the $f_\al$-estimator is asymptotically efficient:
	\begin{equation}
		\label{RK-3}
		\lim_{T\to\f}{\si^2_T(f)}[\si^{2}_T(f_\al,f)]^{-1}=1.
	\end{equation}
	\end{thm}

The condition in Theorem \ref{RK-T1} that $g(\la)$ be uniformly bounded can be relaxed to allow $g(\la)$ to have integrable singularities. In this case, Theorem \ref{RK-T1} extends Theorem \ref{thm5.2} to continuous time models.

For $\al=0$, Theorem \ref{RK-T1} implies the asymptotic efficiency of the LSE for $m$ with significantly weaker restrictions on the spectral density than those typically assumed in such problems (see Theorem \ref{C-GS-2} and Remark \ref{r7.2}). 
Specifically, the spectral density $f(\la)$ may exhibit power-type zeros.
%may have zeros of power type located outside the regression spectrum. 

\begin{rem}
	\label{r6.5c}
	{\rm The results of Theorems \ref{RK-T01} and \ref{RK-T1} regarding
		the model $X(t)=m+Y(t)$ ($t\in\mathbb{R}$) with a spectral density as in
		\eqref{RK-1} can be easily extended to the following simple regression model: 
		$$X(t)=me^{i\la_0 t}+Y(t),\q t\in\mathbb{R},$$ 
		where $\la_0\in\mathbb{R}$ is a fixed constant (see Remark \ref{r6.5} for the discrete-time case).}
\end{rem}

\s{Asymptotic behavior of the variance of BLUE for deterministic models}
\label{Det}
Summarizing the results presented in Sections \ref{ND-BLU} - \ref{Ext}, regarding the asymptotic behavior of the BLUE for 
the unknown mean $m$ in the model:
$$	
X(t)=m+Y(t), \q t\in \mathbb{Z}:=\{0,\pm 1,\pm 2,\ldots\},
$$
where the noise $Y(t)$ is a centered stationary process with a spectral density function $f(\lambda)$,
we conclude that the variance
$\sigma^2_n(f)$ of the BLUE depends asymptotically only on the behavior of
the spectral density $f$ near the origin. 
Furthermore, for certain classes of nondeterministic models with spectral
densities that satisfy the condition:
$$ f(\la) \simeq|\lambda|^\nu\q (\nu> -1) \q {\rm as} \q\lambda \rightarrow 0,$$
the variance $\sigma^2_n(f)$ decreases hyperbolically (like a power):
$$ \sigma^2_n(f) \simeq n^{-\nu -1}\q (\nu> -1) \q {\rm as} \q n\to\f.$$

The following question arises naturally: 
Is it possible to increase the order of zero of the spectral density $f$ at the origin to achieve exponential decay of the variance $\sigma^2_n(f)$ as $n\to\infty$?

In this section, we address the question from a negative perspective. Specifically, we show that if the spectral density
$f(\lambda)$ vanishes only at the origin, it is impossible to achieve exponential
decay of $\sigma^2_n(f)$, regardless of the order of the zero of
$f(\lambda)$ at the origin. 
Moreover, to achieve exponential decay of the variance
$\sigma^2_n(f)$ as $n \to \infty$, the underlying process $X(t)$ must be purely deterministic. This means that the spectral density
$f(\lambda)$ should vanish on a set of positive Lebesgue measure in any vicinity of zero.

It is worth mentioning that a similar problem in the prediction theory of stationary processes has been considered in the classical paper by Rosenblatt \cite{Ros}
(see also the survey paper by Babayan and Ginovyan \cite{BG2019}).
The key difference is that in the prediction problem, the asymptotic behavior of the prediction error variance is determined
by the spectral density $f(\lambda)$ over the entire interval 
$[- \pi, \pi]$, rather than just at the origin (see also Remark \ref{RMT4}).
Additionally, in Rosenblatt's work \cite{Ros},
the technique of orthogonal polynomials on the unit circle was used to establish the exponential decay of the prediction error variance. 
However, this approach is not applicable in the case of "BLUE variance." Instead, we adopt the method developed by Babayan et al. \cite{BGT}, which employs tools from complex analysis and measure theory.

\sn{Setting and results}
\label{S-R}
Let $E_f$ denote the spectrum of the process $X(t)$:
	\begin{equation}
	\label{Sp}
	E_f:=\{e^{i\la}, \,\, f(\la)>0\},
\end{equation} 
%Thus, the closure $\ol E_f$ of $E_f$ is the support of the spectral density $f$.
and let $\widetilde\tau_f:=\widetilde\tau(E_f)$  represent the generalized Chebyshev constant of the set $E_f$ (for the definition of $\widetilde\tau(E_f)$, see Section \ref{P} ). 

The main results of this section are presented in the following theorems.
\begin{thm}
	\label{bas}
	If the spectrum $E_f$ of the process $X(t)$
is either an arc of the unit circle or the unit circle itself, then the sequence
	$\{\sqrt[n]{\sigma_n(f)}, \, n\in\mathbb{N}\}$ converges to a limit
	$\widetilde{\tau}(E_f)\leq 1$. Specifically, we have 
	\begin{equation}
		\label{km1}
		\lim_{n \rightarrow \infty} \sqrt[n]{\sigma_n(f)} = \widetilde{\tau}(E_f)\leq 1.
	\end{equation}
\end{thm}
\begin{thm}
	\label{MT}
	The following assertions hold:
	\begin{itemize}
		\item[(a)] If the spectral density $f(\lambda)$ is positive almost
		everywhere in some neighborhood  of zero, then
		$$
		\lim_{n \rightarrow \infty} \sqrt[n]{\sigma_n(f)} = 1.
		$$
		\item[(b)] If the spectral density $f(\lambda)$ vanishes almost
		everywhere for $|\lambda| < \alpha$, where $0<\alpha<\pi,$ %and is positive otherwise,
		then $\sigma_n(f)$ decreases at least exponentially. More precisely, we have
		$$
		\lim_{n \rightarrow \infty} \sqrt[n]{\sigma_n(f)}
		\leqslant \cos({\alpha}/{2}).
		$$
	\end{itemize}
\end{thm}

As an immediate consequence of Theorem \ref{MT}(a)
we have the following result.

\begin{cor}
	\label{cor}
	A necessary condition for the variance $\sigma^2_n(f)$ to decrease to zero
	exponentially as $n\to\infty$ is that the spectral density $f(\lambda)$
	vanishes on a set of positive Lebesgue measure in any neighborhood of zero.
\end{cor}
	
	\begin{rem}
\label{RMT3}
{\rm
Corollary \ref{cor} demonstrates that if the spectral density \( f(\lambda) \) only vanishes at the origin, it is impossible to achieve exponential decay of \( \sigma_n^2(f) \), regardless of how high the order of the zero of \( f(\lambda) \) is at that point. In this context, it is important to describe the asymptotic behavior of
$\sigma^2_n(f)$ for the flat-zero and Pollaczek-Szeg\H{o} models (see Example \ref{Ex-FZM} and Remark \ref{PSM}).}
	\end{rem}
	\begin{rem}
	\label{RMT1}
	{\rm It follows from relation \eqref{km1} that the question of the exponential decay of $\sigma_n(f)$ as $n\to\f$ does not  depend on the specific form of $f(\lambda)$ but is determined solely by the value of the generalized Chebyshev constant $\widetilde\tau(E_f)$.
	Denoting $\g_n:=\si_n(f)/\widetilde\tau_f^n $, from \eqref{km1} we infer that $\lim_{n \to \infty} \sqrt[n]{\g_n} =1$ and
	\begin{equation}
	\label{sfg}
	\si_n(f)=\widetilde\tau_f^n\cd \g_n.
	\end{equation}
	Thus, when $\widetilde{\tau}_f <1$, the error $\sigma_n(f)$ can be decomposed into a product of two factors: one, $\widetilde{\tau}_f^n$, is a geometric progression, and the other, $\g_n$, is an exponentially neutral sequence (see Definition \ref{ekd1}).}
	\end{rem}

	\sn{Preliminaries}
	\label{P}
	
To prove Theorems \ref{bas} and \ref{MT}, we will need some auxiliary results involving Chebyshev polynomials and the Chebyshev constant for the class $\mathcal{Q}_n(1)$.
		
First, we recall the definitions of the Chebyshev polynomial and the Chebyshev constant for the class $\mathcal{Q}_n$ of monic polynomials. For the definitions of the classes  $\mathcal{Q}_n$ and $\mathcal{Q}_n(1)$, see Section \ref{SzPr}.	

Let $F$ be a bounded closed (compact) 
set in the complex plane $\mathbb{C}$. We define 
	$$m_n(F): = \inf_{q_n\in\mathcal{Q}_n}\max_{z\in F}|q_n(z)|.$$
	
		It is well-known (see, e.g., Goluzin \cite{GMG}, Section 7.1) that there exists a unique monic polynomial $T_n(z,F)$ from the class $\mathcal{Q}_n$,
	called the {\it Chebyshev polynomial} of $F$ of order $n$, such that
	\beq
	\label{td6}
	m_n(F)=\max_{z\in F}|T_n(z,F)|.
	\eeq
	
	Fekete \cite{F} (see also Goluzin \cite{GMG}, Section 7.1) proved that $\lim_{n\to\f}(m_n(F))^{1/n}$ exists.
	This limit, denoted by $\tau(F)$, is called the {\it Chebyshev constant}
	for the set $F$. Thus,
	\beq
	\label{td7}
	\tau(F): = \lim_{n\to\f}(m_n(F))^{1/n}.
	\eeq
The Chebyshev constant $\tau(F)$ is a metric characteristic of the 
compact set $F$. 
One of the fundamental results of geometric complex analysis is the
classical theorem by Fekete and Szeg\H{o}, which states that for any compact set $F$ in the complex plane  $\mathbb{C}$, the Chebyshev constant, the transfinite diameter, and the capacity of $F$ coincide, even though they are defined from very different perspectives.
Specifically, the Chebyshev constant of $F$ characterizes the minimal uniform deviation of a monic polynomial on $F$, the transfinite diameter of the set $F$ characterizes its asymptotic size,  and the capacity of $F$
describes the asymptotic behavior of the Green function at infinity
(for details, see, e.g., Fekete \cite{F}, Goluzin \cite{GMG}, Chapter 7, and Szeg\H{o} \cite{Sz1}, Chapter 16).	
	
The next lemma is analogous to the assertion regarding the existence of Chebyshev polynomials for the class $\mathcal{Q}_n(1)$. 
\begin{lem}
\label{L1}
Let $F$ be an arbitrary infinite bounded closed set in the complex plane $\mathbb{C}$.
Then, in the class $\mathcal{Q}_n(1)$, there exists a polynomial
$\widetilde T_n(z):= \widetilde T_n(z,F)$ with the least maximum modulus on $F$:
\begin{equation}
	\label{m3}
	\max_{z \in F} |\widetilde T_n(z)|=\min_{q_n\in\mathcal{Q}_n(1)}\max_{z \in F} |q_n(z)|.
\end{equation}
\end{lem}
\begin{proof}
Denote
\begin{equation}
	\label{m4}
	\widetilde{m}_n(F): = \inf_{q_n\in\mathcal{Q}_n(1)}\max_{z \in F} |q_n(z)|.
\end{equation}
For a fixed $n\in \mathbb{N}$, let $\{q_{n,k}(z), \, k \in \mathbb{N}\}$
be the sequence of polynomials from the class $\mathcal{Q}_n(1)$,
whose maxima of moduli on $F$ tend to $\widetilde{m}_n(F)$ as $k\to\f$.
We fix $n + 1$ points $z_1, z_2, \ldots, z_{n+1}\in F$ and represent the polynomials
$q_{n,k}(z)$  according to Lagrange's formula as follows: 
\[
q_{n,k}(z) = \sum_{\nu=1}^{n+1} P_{n}(z,{\bold z}_\nu)\, q_{n,k}(z_\nu),
\]
where ${\bf z}_\nu:=(z_1, \ldots,  z_{\nu-1}, z_{\nu+1},\ldots z_{n+1})$, and 
\[
P_{n}(z,{\bf z}_\nu): = \frac{(z - z_1)(z - z_2) \cdots (z - z_{\nu-1})(z - z_{\nu+1})
	\cdots (z - z_{n+1})}{(z_\nu - z_1)(z_\nu - z_2) \cdots (z_\nu - z_{\nu-1})(z_\nu - z_{\nu+1})
	\cdots (z_\nu - z_{n+1})}.
\]

It follows from this representation that the polynomials $q_{n,k}(z)$ are uniformly bounded in modulus with respect to $k$ on any arbitrary bounded closed set of $\mathbb{C}$.
Hence, based on the known condensation principle (see, e.g., Goluzin \cite{GMG},
Sections 1.2 and 7.1), the sequence $\{q_{n,k}(z), \, k \in \mathbb{N}\}$ contains a
subsequence of polynomials that converges uniformly on every bounded set of $\mathbb{C}$.
Additionally, the coefficients of this subsequence also converge. Therefore, the limiting function 
$\widetilde T_n(z):= \widetilde T_n(z,F)$ is a polynomial from the class $\mathcal{Q}_n(1)$,
and satisfies the condition \eqref{m3}. Thus, we have
\begin{equation}
	\label{km3}
	\max_{z \in F} |\widetilde T_n(z,F)|=\widetilde{m}_n(F)
	=\min_{q_n\in\mathcal{Q}_n(1)}\max_{z \in F} |q_n(z)|.
\end{equation}

Lemma \ref{L1} is proved.
\end{proof}

\begin{rem}
\label{KR2}
{\rm The polynomial $\widetilde T_n(z,F)$ serves as an analog of the Chebyshev polynomial $T_n(z,F)$, which has the least maximum modulus on the set $F$ within the class $\mathcal{Q}_n(1)$. 
 Consequently, we refer to $\widetilde T_n(z,F)$ as the {\it Chebyshev polynomial for the set $F$ with respect to the point} $\xi=1$, 
or simply as the {\it the generalized Chebyshev polynomial}.}
\end{rem}
The next lemma is an analogue, within the class $\mathcal{Q}_n(1)$, of the assertion regarding the existence of the Chebyshev constant for any bounded closed set $F$ in the complex plane $\mathbb{C}$
(see, e.g., Goluzin \cite{GMG}, Section 7.1).
\begin{lem}
\label{L2}
For any bounded closed set $F$ in the complex plane $\mathbb{C}$, the sequence
$$\{ \widetilde{\tau}_n(F): = \sqrt[n]{\widetilde{m}_n(F)}, \, n\in \mathbb{N}\},$$
where $\widetilde{m}_n(F)$ is defined as in \eqref{m4}, converges to some finite limit
$\widetilde{\tau}(F)$. Specifically, we have
$$
\lim_{n \rightarrow \infty} \sqrt[n]{\widetilde{m}_n(F)}=\widetilde{\tau}(F) <\f.
$$
\end{lem}
\begin{proof}
Taking into account that the set $F$ is bounded, we have
$$
R=R_F:= \max_{z \in F} |z|<\f.
$$
Hence, observing that the polynomial $q_n(z)=z^n$ belongs to the class $\mathcal{Q}_n(1)$,
in view of \eqref{km3}, we conclude that the sequence
$\{ \widetilde{\tau}_n(F), \, n\in \mathbb{N}\}$ is bounded:
\begin{equation}
	\label{Km56}
	\widetilde{\tau}_n(F) = \sqrt[n]{\widetilde{m}_n(F)}
	\leq\sqrt[n]{\max_{z \in F}|z^n|} = R <\f. %= \max_{z \in F} |z|
\end{equation}

Define
$$
\liminf_{n \rightarrow \infty} \widetilde{\tau}_n (F)= a \q \text{and} \q
\limsup_{n \rightarrow \infty} \widetilde{\tau}_n(F) = b,
$$
and observe that $a \leqslant b$.
Thus, to complete the proof, we need to show that $b \leqslant a$.
To this end, for a given $\varepsilon > 0$, we choose $n_0\in\mathbb{N}$
such that $\widetilde{\tau}_{n_0}(F) < a + \varepsilon$. Then, we have
$$ |\widetilde T_{n_0}(z)| < (a + \varepsilon)^{n_0}, \q z \in F.$$

Next, observe that for any $q \in \mathbb{N}$ and $r\in \mathbb{Z}_+$,
the polynomial $s_n(z): = z^r[\widetilde T_{n_0}(z)]^q$ of degree $n=n_0q + r$
satisfies the conditions:
\beq
\label{b1}
s_n(z)\in \mathcal{Q}_n(1) \q {\rm and} \q |s_n(z)| < R^r(a + \varepsilon)^{n_0q}, \q z\in F.
\eeq
Therefore, in view of \eqref{km3} and \eqref{b1}, we have
$$
\widetilde{m}_n(F) \leq \max_{z \in F} |s_n(z)|\leqslant R^r (a + \varepsilon)^{{n_0q}},
$$
and
\begin{equation}
	\label{m6}
	\widetilde{\tau}_n(F) \leqslant R^{\frac{r}{n}}(a + \varepsilon)^{\frac{n_0q}{n}}.
\end{equation}

Now, let the subsequence $\{\widetilde{\tau}_{n_\nu}(F),\, n_\nu \in \mathbb{N}\}$
converge to $b$ as $\nu \rightarrow \infty$. We apply the inequality \eqref{m6} for
$n_\nu := n_0q_\nu + r_\nu $ with $0 \leq r_\nu < n_0$ to obtain
\begin{equation}
	\label{Km6}
	\widetilde{\tau}_{n_\nu}(F) \leqslant R^{\frac{r_\nu}{n_\nu}}(a + \varepsilon)^{\frac{n_0q_\nu}{n_\nu}}.
\end{equation}

Finally, by letting $\nu$ tend to infinity and applying inequality \eqref{Km6}, we find that $b \leqslant a + \varepsilon$. Since $\varepsilon$ is arbitrary, we conclude that $b \leqslant a$. 

Thus, Lemma \ref{L2} is proven.
\end{proof}
\begin{rem}
\label{RL2}
{\rm 1. The quantity $\widetilde\tau(F)$ is a metric characteristic of a closed set
$F$, similar to the Chebyshev constant $\tau(F)$ (see Goluzin \cite{GMG}, Section 7.1).
Therefore, we refer to $\widetilde\tau(F)$ as the {\it Chebyshev constant of the 
set $F$} in the class $\mathcal{Q}_n(1)$, or simply as the {\it generalized Chebyshev constant of $F$}.\\
	2. It follows from \eqref{Km56} that $\widetilde\tau(F)\leq R$. In particular,
	for the unit circle $\mathbb{T}$ we have
	\begin{equation}
		\label{K15}
		\widetilde\tau(\mathbb{T})\leq 1.
	\end{equation}
	3. Using \eqref{km3}, it is straightforward to verify that $\widetilde\tau(F)$
	is a non-decreasing set function:
	\begin{equation}
		\label{K1}
		\widetilde\tau(F_1)\leq\widetilde\tau(F_2) \q {\rm if} \q F_1\subset F_2.
	\end{equation}
	Additionally, from  \eqref{K15} and \eqref{K1}, we obtain
	\begin{equation}
		\label{K35}
		\widetilde{\tau}(F) \leq 1  \q \text{ for any } \q F\subset \mathbb{T}.
	\end{equation}
	4. If the set $F$ contains the point $\xi = 1$, then we have
	$\widetilde{\tau}(F) \geq 1$. Hence, according to \eqref{K35}, for any set $F\subset \mathbb{T}$ that contains $\xi = 1$, it follows that 
	$
	\widetilde{\tau}(F) = 1.
	$
	In particular, for the unit circle $\mathbb{T}$ and any of its arcs 
	$$\G_\al:=\{e^{i\la}, \,\, |\la|\leq\al, \, 0<\al<\pi\},$$ we have
	\begin{equation}
		\label{K5}
		\widetilde{\tau}(\mathbb{T}) =\widetilde{\tau}(\G_\al)= 1.
	\end{equation}
	Note that for the unit circle $\mathbb{T}$, the Chebyshev constant $\tau(\mathbb{T})$
	is also equal to 1 (see Goluzin \cite{GMG}, Section 7.1). Thus, we have 
	$\widetilde{\tau}(\mathbb{T}) =\tau(\mathbb{T})= 1.$}
	\end{rem}
	
The following result by Mazurkievicz \cite{Maz} is crucial for proving Theorem \ref{bas}. 
First, we introduce the key concepts and definitions.
A \textit{continuum} is defined as a continuous, rectifiable curve in the complex plane $\mathbb{C}$.
A \textit{linear set} $G$ in $\mathbb{C}$ is defined as any subset of a continuum.
The \textit{linear measure} $\mu(G)$ of a linear set $G$ is defined as the Lebesgue measure generated by the length of an arc of the continuum.
\begin{pp}[The Mazurkievicz Theorem]
		\label{maz}
	For any $\epsilon>0$, there exists $\delta=\delta(\epsilon)>0$ 
		such that for any continuum $\G$ of diameter $d$
		and any closed subset $F\subset\G$, the following inequality holds:
		$$
		M_n:=\max_{x \in \G}|q_n(z)|\leq(1+\vs)^n\max_{z\in F}|q_n(z)|,
		$$
		provided that $\mu(\G\setminus F)<d\delta$, where $q_n(z)$ is an arbitrary
		polynomial of degree $n$, and $\mu(e)$ denotes the linear measure of a set $e$.
	\end{pp}
	
\sn{Proof of Theorems \ref{bas} and \ref{MT}}
\label{PMT}

\begin{proof}[Proof of Theorem \ref{bas}]
Based on the definition of the optimal polynomial $p_n(z,f)$ (see Definition \ref{def3.1}), along with formulas \eqref{OSm2} and \eqref{km3},  we can write
\begin{eqnarray}
	\label{km7}
	\nonumber
	\sigma_n^2(f) %= \int_{-\pi}^{\pi} |p_n(e^{i\lambda}, f)|^2 f(\lambda)d\lambda
	\leq \int_{-\pi}^{\pi}|\widetilde T_n(e^{i\lambda}, E_f)|^2 f(\lambda) d\lambda
	\leq \widetilde{m}_n^2(E_f) \cdot \int_{-\pi}^{\pi} f(\lambda)d\lambda.
\end{eqnarray}
Hence, by Lemma \ref{L2} and \eqref{K35}, we get
\begin{equation}
	\label{km8}
	\limsup_{n \rightarrow \infty} \sqrt[n]{\sigma_n(f)} \leqslant
	\widetilde{\tau}(E_f) \leqslant 1.
\end{equation}

Now we proceed to prove the inequality:
\begin{equation}
	\label{m9}
	\liminf_{n \rightarrow \infty} \sqrt[n]{\sigma_n(f)} \geqslant
	\widetilde{\tau}(E_f).
\end{equation}
To this end, we consider a sequence of subsets $G_n$ of $E_f$, defined by
\beq
\label{kb1}
G_n:=\{z= e^{i\lambda} \in E_f: \,|p_n(z, f)| > \sqrt{\sigma_n(f) \cdot \widetilde{m}_n (E_f)}\},
\eeq
and we denote by $\mu_f$ a measure on the unit circle defined as follows:
$$
\mu_f(G)=\int_{W^{-1}(G)}f(\lambda)d\lambda, \q G\subset E_f,
$$
where $W^{-1}(G):=\{\lambda\in\Lambda: \, e^{i\lambda} \in G\}.$
It is clear that
$$
\limsup_{n \rightarrow \infty} \sqrt[n]{\mu_f(G_n)}\leqslant 1.
$$

Next, for a given sufficiently small number $\rho$ ($0 < \rho < 1$),
we can express the set of natural numbers $\mathbb{N}$ in the form
$\mathbb{N}=J_1\cup J_2$, where
\begin{align*}
&J_1 : = \{n\in\mathbb{N}: \mu_f(G_n)> (1 - \rho)^n\},\\
&J_2 : = \{n\in\mathbb{N}: \mu_f(G_n)\leq (1 - \rho)^n\}.
\end{align*}

It is clear that at least one of the sets $J_1$ and $J_2$ is infinite.
We can assume, without loss of generality, that both $J_1$ and $J_2$ are infinite sets.
For each $n \in J_1$, we have
$$
\sigma_n^2(f) = \int_{E_f} |p_n(z, f)|^2d\mu_f \geqslant
\int_{G_n}|p_n(z, f)|^2 d\mu_f > \sigma_n(f)\widetilde{m}_n(E_f)\mu_f(G_n),
$$
implying that
\begin{equation}
	\label{m10}
	\liminf_{n \rightarrow \infty} \sqrt[n]{\sigma_n(f)}
	\geq \lim_{n \rightarrow \infty} \sqrt[n]{\widetilde{m}_n(E_f)} \cdot
	\liminf_{n \rightarrow \infty}\sqrt[n]{\mu_f(G_n)}
	\geq \widetilde\tau(E_f)(1 -\rho).
\end{equation}

For each $n \in J_2$, we have
\begin{equation}
	\label{m11}
	\lim_{n \rightarrow\infty} \mu_f(G_n)= 0.
\end{equation}
Since $f(\lambda)$ is positive on the set
$W^{-1}(E_f):=\{\lambda\in\La: \, e^{i\lambda} \in E_f\}$,
the linear  measure $\mu$ is absolutely continuous
with respect to the measure $\mu_f$. 
Taking into account that the measure $\mu$ is also finite,  we obtain  from equation \eqref{m11} that
\begin{equation}
	\label{m12}
	\lim_{n \rightarrow \infty} \mu(G_n)  = 0, \q n \in J_2.
\end{equation}

For $n \in J_2$, we define the sets $E_n: = E_f \backslash G_n$ and use
\eqref{kb1} to obtain
\begin{equation}
	\label{m13}
	|p_n(z, f)| \leqslant \sqrt{\sigma_n(f) \cdot \widetilde{m}_n(E_f)}, \q z \in E_n.
\end{equation}
Let $\varepsilon > 0$ be an arbitrary number satisfying
\begin{equation}
	\label{m14}
	(1 + \varepsilon)^{-2}\geq (1 - \rho), %\frac{1}{(1 + \varepsilon)^2} \geqslant 1 - \rho,
\end{equation}
and let $\delta:=\delta(E_f,\varepsilon)$ be chosen according to the Mazurkievicz theorem (see Proposition \ref{maz}).
Then, in view of \eqref{m12}, for sufficiently large  $n \in J_2$, we have
$$
\mu(E_f\backslash E_n) = \mu(G_n) < \delta(E_f, \varepsilon).
$$
Therefore, based on inequality \eqref{m13} and Proposition \ref{maz}, we can write
\[
\widetilde{m}_n(E_f) =\max_{z \in E_f} |\widetilde T_n(z, E_f)| \leqslant
\max_{z \in E_f} |p_n(z, f)|
\]
\[
\leq (1 + \varepsilon)^n \max_{z \in E_n} |p_n(z, f)| \leqslant
(1 + \varepsilon)^n\sqrt{\sigma_n(f)\cdot \widetilde{m}_n(E_f)}, \q n \in J_2,
\]
implying that
$$ \sigma_n(f)\geqslant (1 + \varepsilon)^{-2n}\widetilde{m}_n(E_f), \quad n \in J_2.$$
%\frac{\widetilde{m}_n(E_f)}{(1 + \varepsilon)^{2n}}$ for $n \in J_2.
Letting $n$ tend to infinity and using the inequality \eqref{m14}, we obtain
\begin{equation}
	\label{m15}
	\liminf_{n \rightarrow \infty} \sqrt[n]{\sigma_n(f)}
	\geqslant \widetilde{\tau}(E_f) (1 - \rho).
\end{equation}

By combining \eqref{m10} and \eqref{m15} and taking into account
the arbitrariness of $\rho$, we derive the desired inequality \eqref{m9}.
The combination of \eqref{km8} and \eqref{m9} leads to \eqref{km1},
thereby completing the proof of the theorem.
\end{proof}

\begin{proof}[Proof of Theorem \ref{MT}]
Both assertions (a) and (b) of the theorem are derived from Theorem \ref{bas}.
To prove assertion (a), we first note that the spectral density \( f(\lambda) \) is positive almost everywhere in a neighborhood of zero. Therefore, we can assume that $E_f\supset\G_\al$ for some $0<\al<\pi$, where
$\G_\al=\{e^{i\la}, \,\, |\la|\leq\al\}$.

Let $f_\al(\la)$ denote the contraction of $f(\lambda)$ on the set $W^{-1}(\G_\al)$. Specifically,  
\beq
\label{cont}
f_\al(\la):= \left \{
\begin{array}{ll}
	f(\lambda) & \mbox {if \, $|\la|\leq\al$}\\
	0 & \mbox {if \, $|\la|>\al$}.
\end{array}
\right.
\eeq
Taking into account the obvious inequality $f(\la)\geq f_\al(\la)$
and Proposition \ref{R2}, we conclude that $\si_n(f)\geq \si_n(f_\al)$.
Thus, in light of Theorem \ref{bas} and equation \eqref{K5}, we obtain
\begin{equation}
	\label{kb10}
	\liminf_{n \rightarrow \infty} \sqrt[n]{\sigma_n(f)}
	\geq \lim_{n \rightarrow \infty} \sigma_n(f_\al)=\widetilde\tau(\G_\al)=1.
\end{equation}
By combining relations \eqref{km8} and \eqref{kb10}, we complete the proof
of assertion (a) of the theorem.

To prove assertion (b) of the theorem, we can assume without loss of generality that \(E_f=\G'_\alpha=\{e^{i\lambda}, \, \alpha \leq |\lambda| \leq \pi\}\). Thus, the relation in equation \eqref{km1} holds. 
Therefore, to complete the proof, it remains to show that
\begin{equation}
	\label{m88}
	\widetilde{\tau}(E_f)=\widetilde{\tau}(\G'_\al)  \leqslant \cos({\alpha}/{2}).
\end{equation}
To this end, consider the polynomial of degree $n$:
$$
q_{n}(z) = \left[{(z + 1)}/{2} \right]^{n},
$$
and observe that
$
q_{n}(z) \in \mathcal{Q}_n(1)$ and $ |q_{n}(e^{i\lambda})| = \left(\cos({\lambda}/{2}) \right)^{n}.
$

Next, according to the definition of the polynomial
$\widetilde T_n(z)$, with $E_f=\G'_\al$, we have
$$
\widetilde{m}_n(E_f) =\max_{z \in E_f} |\widetilde T_n(z, E_f)|
\leqslant \max_{z \in E_f} |q_n(z)| = \max_{|\la|\geq\al} |q_n(e^{i\la})|
=\left(\cos({\alpha}/{2}) \right)^{n}.
$$

By taking the $n$-th root and passing to the limit as $n \to \infty$, we obtain \eqref{m88}. The combination of \eqref{km1} and \eqref{m88} completes the proof of assertion (b) of the theorem. Thus, Theorem \ref{MT} is proven.
\end{proof}

\section{Tools}
\label{Tl}
In this section, we present several known concepts, results, and tools that are used 
to state and prove the results presented in Sections \ref{model} - \ref{Det}.

\subsection{Hilbert spaces associated with second-order processes}
\label{HSR}

\ssn{Hilbert space spanned by a family of elements.}
Let $H$ be a Hilbert space with a norm $\Vert\cdot\Vert_H$ and an inner
product $(\cdot,\cdot)_H$. Let ${\bf T}$ be an index set, and $\{u(t),t\in {\bf T}\}$
be a family of elements of $H$. We define the linear manifold {\it spanned} by the family $\{u(t),t\in {\bf T}\}$, denoted by
$$L:={\rm sp}\{u(t), \, t\in {\bf T}\}=L(u(t), \, t\in {\bf T}),$$
as the set consisting of all elements $u(t)\in H$ that can be represented
as a linear combination:
\beq\label{rk3}
u=\sum_{k=1}^n c_ku(t_k), \q n\in\mathbb{N}
\eeq
for some real constants $c_1,\ldots,c_n$ and $t_1,\ldots,t_n\in{\bf T}$.

We define the {\it Hilbert space spanned by the family}  $\{u(t),t\in {\bf T}\}$ 
as the closure of $L$ in $H$:
\beq\label{rk4}
\ol L=\ol L(u(t), \, t\in {\bf T})=\overline{sp}\{u(t), \, t\in {\bf T}\}_H.
\eeq

Note that the space \(\ol L\) consists of all finite linear combinations of the form given in \eqref{rk3}, along with their limits in the norm of \(H\). 

Generally, \(\ol L\) is a subspace of \(H\), and it is possible for \(\ol L\) to be equal to \(H\). In this case, we say that the family $\{u(t), t \in {\bf T}\}$
{\it spans} \(H\).

Note that the family $\{u(t), \, t\in {\bf T}\}$ spans $H$ if and only if the only vector in $H$ that satisfies the condition 
$(g,u(t))_H=0$ for every $t\in {\bf T}$ is $g = 0$ (see, e.g., Parzen \cite{P-3}).

\begin{pp}[The Basic Congruence Theorem]
	\label{BCT}
	Let $H_1$ and $H_2$ be two abstract Hilbert spaces with inner products
	$(\cdot,\cdot)_1$ and $(\cdot,\cdot)_2$, respectively.
	Let  $\{\varphi(t), \, t\in {\bf T}\}$ and $\{\psi(t), \, t\in {\bf T}\}$ be families of vectors that span $H_1$ and $H_2$, respectively.
	Suppose that, for every $s,t\in {\bf T}$,
	\beq\label{rk5}
	(\varphi(s),\varphi(t))_1=(\psi(s),\psi(t))_2.
	\eeq
	Then, the spaces $H_1$ and $H_2$ are congruent (isometrically isomorphic), 
	allowing us to define a congruence $V$ between $H_1$ and $H_2$ with the property that
	\beq\label{rk6}
	V[\varphi(t)]=\psi(t) \,  \mbox{ for all } \, t\in {\bf T}.
	\eeq
\end{pp}
\begin{pp}[The Projection Theorem]
	\label{PrT}
	Let $H$ be an abstract Hilbert space equipped with a norm $||\cdot||$ and an 
	inner product $(\cdot,\cdot)$. 
	Let $M$ be a Hilbert subspace of $H$, and let $h\in H$ and $h^*\in M$. A necessary and sufficient condition that $h^*$ to be the unique
	vector in $M$ that satisfies the following condition:
	\beq\label{proj}
   ||h^*-h||=\min_{g\in M}||g-h||
	\eeq
	is that 
	$$
	(h^*,g)=(h,g)\q \text{for every}\q g\in M.
	$$
	The vector $h^*$ that satisfies the minimization condition \eqref{proj} 
	is referred to as the projection of $h\in H$ onto $M$, and is denoted by $P_Mh$.
\end{pp}

\ssn{Reproducing kernel Hilbert spaces (RKHS's)}

Let $H$ be a class of functions defined on a set ${\bf T}$, forming a Hilbert
space with a norm $\Vert\cdot\Vert$  and an inner
product $(\cdot,\cdot)$, and let $K(t,s)$ be a function defined on
${\bf T}\times{\bf T}$.
\begin{den}
	The function $K(t,s)$, $t,s\in {\bf T}$, is called a reproducing kernel (r.k.)
	of the space $H$  if it satisfies the following conditions:
	\begin{itemize}
		\item[(a)] \,
		For every $s\in {\bf T}$, $ K(t,s)$ as a function of $t$ belongs to $H$; i.e., $K(\cdot,s)\in H$.
		\item[(b)] \,
		The reproducing property holds: for every $s\in {\bf T}$ and every
		$f\in H$,
		\beq\label{rk1}
		(f(\cdot),K(\cdot,t))=f(t).
		\eeq
	\end{itemize}
\end{den}

The Hilbert space $H$ with reproducing kernel $K(t,s)$ is called a {\it reproducing kernel Hilbert space (RKHS)} and is often denoted by $H(K)$. The space $H(K)$ is uniquely determined by conditions (a) and (b).

From the formula \eqref{rk1} with $f(\cdot)=K(\cdot,t)$, we derive the following important equality:
\beq\label{rk2}
(K(\cdot,t),K(\cdot,s))=K(s,t).
\eeq
	\begin{pp}[The Moore-Aronszajn Theorem]
		\label{MAT}
		Let ${\bf T}$ be a given set, and let $K(t,s)$ be a symmetric non-negative definite kernel defined on ${\bf T}\times {\bf T}$. Then, there exists a unique Hilbert space $H:=H(K)$ of functions defined on ${\bf T}$ for which $K(t,s)$ serves as a reproducing kernel. Specifically, we have
		\begin{itemize}
			\item[a)] the system $\{K(\cdot,t), \, t\in {\bf T}\}$ spans the entire space $H$: 
			$$H(K)=\overline{sp}\{K(\cdot,t), \, t\in {\bf T}\}_H;$$
			\item[b)] the inner product in $H$ satisfies the condition:
			$$(f(\cdot),K(\cdot,t))_H=f(t) \mbox{ for any } f\in H.$$
		\end{itemize}
		\end{pp}

\ssn{RKHS representation of $L^2$-processes}

The purpose of the Hilbert space representation of an $L^2$-process is to identify specific Hilbert spaces of functions that are congruent
(isometrically isomorphic) to the Hilbert space generated by that process.

Recall the Hilbert space $L^2(\Ph)$ of all centered random variables $\xi= \xi(\om)$ defined on a probability space $(\Om, \mathcal{F}, \Ph)$ (see Section \ref{dnd}).
Given an index set $\bf T$ and a centered $L^2$-process $\{X(t),t\in {\bf T}\}$
with covariance function $R(t,s),\ t,s\in {\bf T}$, let $L^2_{\bf T}(X)$ denote the Hilbert space spanned by the family $\{X(t),t\in {\bf T}\}$:
\begin{equation}
	\label{rep3}
L^2_{\bf T}(X):=\overline{sp}\{X(t),\ t\in {\bf T}\} _{L^2(\Ph)}.
\end{equation}
Note that the space  $L^2_{\bf T}(X)$ consists of all finite linear combinations of the form
$\sum_{k=1}^n c_k X(t_k)$, $t_k\in{\bf T}$, as well as their $L^2(\Ph)$-limits.

\begin{den}
	Let $H$ be a Hilbert space with inner product $(\cdot,\cdot)_H$.
	A family of vectors $\{f(t),\ t\in {\bf T}\}$ in the Hilbert space $H$
	is said to be a representation of the process $\{X(t),\ t\in {\bf T}\}$ if
	for every $s,t\in {\bf T}$, the following holds:
	\begin{equation}
		\label{rep4}
		(f(s),f(t))_H=(X(s),X(t))_{L^2(\Ph)}=E[X(s)X(t)]=R(s,t).
	\end{equation}
\end{den}

In other words, the family of functions $\{f(t), t \in \mathbf{T}\}$ represents 
the $L^2$-process $\{X(t), t \in \mathbf{T}\}$ if the Hilbert spaces  
$L^2_{\bf T}(X)$ and 
\begin{equation}
	\label{rep5}
	\ol L(f,{\bf T}):=\overline{sp}\{f(t),\ t\in {\bf T}\}_H
\end{equation}
are congruent. The Hilbert space $\ol L(f,{\bf T})$ is also referred to as a representation of the process $\{X(t),\ t\in {\bf T}\}$.

Since the covariance function $R(t, s)$ is a symmetric, non-negative definite kernel, 
we can apply the Moore-Aronszajn theorem (see Proposition \ref{MAT}), to state 
the following result, which shows that an $L^2$-process $\{X(t), t \in \mathbf{T}\}$ can be represented by the RKHS generated by its covariance function. 
\begin{pp}
	\label{RKHS}
	Let $\{X(t), t\in{\bf T}\}$ be a centered $L^2$-process with a covariance
	function $R(t,s)$, $t,s\in{\bf T}$. The following assertions hold.
	\begin{itemize}
		\item[(a)] There exists a unique RKHS, denoted by $H(R)$,
		for which $R(t, s)$ is the reproducing kernel.
		\item[(b)] The space $H(R)$ consists of functions $f(\cd)$ of the form: for some $Y\in L^2_{\bf T}(X)$,
		$$f(t)=\Exp[Y\cdot X(t)] \q \text{for all} \q t\in {\bf T}.$$
		\item[(c)] The inner product in the space $H(R)$ is defined by
		$$(f,f)_R=\Exp[Y^2].$$
		\item[(d)] The space $H(R)$ is spanned by the family $\{R(\cdot,s),\ s\in {\bf T}\}$.
		\item[(e)] The Hilbert spaces \( L^2_{\bf T}(X) \) and \( H(R) \) are isometrically isomorphic. In this context, \( (f,Y)_R \) denotes the random variable in \( L^2_{\bf T}(X) \) that corresponds to the function \( f \in H(R) \).
	\end{itemize}
\end{pp}

\subsection{Fej\'er-type singular integrals}
\label{FSI}

Let $T$ be a parameter that ranges over some set ${\bf T}$, which is
either an interval $(a,b)$ with $-\f\leq a<b\leq\f$ or the set 
$\mathbb{N}$.
A set of functions $\{K_T(\la), \, T\in{\bf T}\}$ is called a
{\it kernel} if $K_T(\la)\in L^1(\Lambda)$ for any $T\in{\bf T}$,
where $\Lambda=\mathbb{T}$ or $\Lambda=\mathbb{R}$, and 
\beq
\label{Ker1}
\int_{\Lambda} K_T(\la)\,d\la = 1.
\eeq

A sequence of kernels  $\{K_T(\la), \, T\in{\bf T}\}$ is called an {\it approximate identity} if, with some constant $C>0$, it satisfies the following conditions
(see, e.g., Butzer and Nessel \cite{BN}, pp. 31, 121):
	\begin{align}
		\label{Ker2}
		&\sup_{T}\int_{{\Lambda}} K_T(\la)\,d\la=C<\f \q T\in{\bf T};\\
		\label{F14}
		&\lim_{T\to \f}\int_{|\la|\geq\de} K_T(\la)d\la
		=0 \q\text{for any} \q\delta>0.
	\end{align}
Apart from \eqref{F14} sometimes we use the following stronger condition:
\beq
\label{F14-1}
\lim_{T\to\f} \sup_{|\la|\geq\de}K_T(\la)\,d\la = 0,\q \de>0.
\eeq
An example of the approximate identity is the classical Fej\'er kernel $F_T(\la)$ defined for both the real line ($\la\in\mathbb{R}$) and the periodic case ($\la\in\mathbb{T}$), given by
\begin{equation}
	\label{s6}
	F_T(\la):= \left \{
	\begin{array}{ll}
		\frac1{2\pi T}
		\cdot\left[\frac{\sin(T\la/2)}{\sin(\la/2)}\right]^2 &
		\mbox { for  $\la \in \mathbb{T}, \, T \in \mathbb{N}$},\\
		\\
		\frac1{2\pi T}
		\cdot\left[\frac{\sin(T\la/2)}{\la/2}\right]^2 &
		\mbox { for  $\la \in \mathbb{R}, \, T \in \mathbb{R}^+$}.
	\end{array}
	\right.
\end{equation}

Let $\psi(\la)\in L^1(\Lambda)$ and $K_T(\la)$ be a kernel.
Then an expression of the form
\begin{equation}
	\label{Ker3}
	\psi_T(\la)	=\psi_T(\psi;\la):=\int_{\Lambda} K_T(x)\psi(\la-x)dx 
\end{equation}
is called a {\it singular integral} (or {\it convolution integral}) corresponding to the function $\psi(\la)$.

If the dependence of the kernel $K_T(\la)$ on the parameter $T$ takes the special form $K_T(\la)=TK(T\la)$, it defines an approximate identity.
Kernels of this type, along with the corresponding singular integrals, are known as {\it Fej\'er-type kernels} and {\it Fej\'er-type singular integrals}, respectively.

Note that if $K_T(\la)=F_T(\la)$ is the Fej\'er kernel, then $\psi_T(\la)$, as defined in equation \eqref{Ker3}, is an entire analytic function of exponential type $T$ when $\Lambda=\mathbb{R}$, and a trigonometric polynomial of degree $T$ when $\Lambda=\mathbb{T}$.

In the following proposition, we present an important and well-known property of the singular integral \eqref{Ker3} 
(see, e.g., Butzer and Nessel \cite{BN}, pp. 32 and 121, and Simon \cite{Sim}, p. 139).
\begin{pp}
	\label{F-Ker} If the kernel $K_T(\la)$ in the integral \eqref{Ker3} is an approximate identity and $\psi(\la)\in \mathbb{C}(\Lambda)$, then	
\begin{equation}
	\label{Ker4}
\lim_{T\to \f}\psi_T(\la)	=\psi(\la).
\end{equation}
\end{pp}

\sn{Some properties of the geometric mean and trigonometric polynomials.}
Recall that a trigonometric polynomial $t(\lambda)$ of degree $\nu$ is a
function of the form:
\beq
\label{tp1}
\nonumber
t(\lambda)= a_0+\sum_{k=1}^\nu(a_k\cos k\la + b_k\sin k\la)
= \sum_{k=-\nu}^\nu c_ke^{ik\la}, \q \la\in \mathbb{R},
\eeq
where $a_0, a_k, b_k \in \mathbb{R}$, $c_0=a_0$, $c_k=1/2(a_k-ib_k)$,
$c_{-k}=\ol c_k$, $k=1,2,\ldots, \nu$.

For a non-negative function $h$, we denote the geometric mean of $h$ by $G(h)$ (see formula \eqref{a2}). In the following proposition, we list some properties of the geometric mean $G(h)$ and trigonometric polynomials (see Babayan et al. \cite{BGT}).
\begin{pp}
	\label{p4.2}
	The following assertions hold.
	\begin{itemize}
		\item[(a)]
		Let $c>0$, $\al\in\mathbb{R}$, $f\geq0$ and $g\geq0$. Then
		\beq
		\label{gt1}
		G(c)=c; \q G(fg)=G(f)G(g); \q G(f^\al)=G^\al(f) \, \, (G(f)> 0).
		\eeq
		\item[(b)]
		$G(f)$ is a non-decreasing functional of $f$: if \ $0\leq f(\la)\leq g(\la)$,
		then $0\leq G(f)\leq G(g)$.
		In particular, if $0\leq f(\la)\leq 1$, then $0\leq G(f)\leq 1$.
		\item[(c)]
		{\rm (Fej\'er-Riesz theorem, see. e.g., Grenander and Szeg\H{o}, Sec. 1.12)}. Let $t(\la)$ be a non-negative trigonometric polynomial of degree $\nu$. Then
		there exists an algebraic polynomial  $s_{\nu}(z)$ $(z\in \mathbb{C})$
		of the same degree $\nu$, such that $s_{\nu}(z)\neq0$ for $|z|<1$, and
		\beq
		\label{ss2}
		t(\lambda)=|s_{\nu}(e^{i\lambda})|^2.
		\eeq
		Under the additional condition $s_{\nu}(0)>0$ the polynomial $s_{\nu}(z)$
		is determined uniquely.
		\item[(d)]
		Let $t(\lambda)$ and  $s_{\nu}(z)$ be as in assertion (c). Then
		$G(t)=|s_{\nu}(0)|^2>0$.
	\end{itemize}
\end{pp}

\sn{Weakly varying sequences}
\label{ww}
We introduce the concept of weakly varying sequences and 
outline some of their properties (see Babayan et al. \cite{BGT}).
\begin{den}
	\label{kd1}
	A sequence of non-zero numbers $\{a_n, \, n \in\mathbb{N}\}$ is said
	to be weakly varying if
	$$\lim_{n\to\f} {a_{n+1}}/{a_{n}} =1.$$
\end{den}
For example, the sequence $\{n^{\alpha}, \, \alpha \in \mathbb{R}, n \in \mathbb{N}\}$
is weakly varying (it is weakly decreasing for $\alpha < 0$
and weakly increasing for $\alpha > 0$). In contrast, the geometric progression 
$\{q^n, \, 0<q<1,  n \in \mathbb{N}\}$ is not weakly varying.

In the following proposition, we outline some simple properties of weakly varying sequences, which can be easily verified.
\begin{pp}
	\label{p4.1}
	The following assertions hold.
	\begin{itemize}
		\item[(a)]
		If $a_n$ is a weakly varying sequence, then
		$\lim_{n\to\f} {a_{n+\nu}}/{a_{n}} =1$ for any $\nu\in \mathbb{N}$.
		\item[(b)]
		If $a_n$ is such that $\lim_{n\to\f}a_n=a\neq0$,
		then $a_n$ is a weakly varying sequence.
		\item[(c)]
		If $a_n$ and $b_n$ are weakly varying sequences, then $ca_n$ $(c\neq0),$
		$a_n^\al \, (\al\in\mathbb{R}, a_n>0)$, $a_nb_n$ and $a_n/b_n$  are also
		weakly varying sequences.
		\item[(d)]
		If $a_n$ is a weakly varying sequence,
		and $b_n$ is a sequence of non-zero numbers such that
		$\lim_{n\to\infty}{b_n}/{a_n}=c\neq 0,$
		then $b_n$ is also a weakly varying sequence.
		\end{itemize}
\end{pp}

We will need the following definition, which characterizes
the rate of variation of a sequence of non-negative numbers compared with a geometric progression
(see Babayan et al. \cite{BGT} and Simon \cite{Sim-1}, p. 91).
\begin{den}
	\label{ekd1}
	(a) A sequence $\{a_n\geq 0, \, n \in\mathbb{N}\}$ %of non-negative numbers
	is called exponentially neutral if
	$
	\lim_{n \rightarrow \infty} \sqrt[n]{a_n} =1.
	$
	(b) A sequence $\{b_n\geq 0, \, n \in\mathbb{N}\}$ %of non-negative numbers
	is called exponentially decreasing if
	$
	\limsup_{n \rightarrow \infty} \sqrt[n]{b_n} <1.
	$
\end{den}
For instance, the sequence $\{a_n=n^\al, \, \al\in\mathbb{R}, \,
n \in\mathbb{N}\}$ is exponentially neutral because
$\log \sqrt[n]{n^\alpha} = ({\alpha}/{n}) \log {n}\rightarrow 0$ as $n\to\f$.
The geometric progression $\{b_n=q^n, \, 0<q<1, \, n \in\mathbb{N}\}$
is exponentially decreasing because
$\sqrt[n]{b_n} = q^{n/n} = q<1$.
The sequence $\{b_n=n^\al q^n, \, \al\in\mathbb{R}, \, 0<q<1, \,
n \in\mathbb{N}\}$ is also exponentially decreasing because
$\sqrt[n]{b_n} = n^{\alpha/n}q \rightarrow q<1.$
It can be easily shown that a sequence $\{b_n\geq0, \, n \in\mathbb{N}\}$
		is exponentially decreasing %, that is, \eqref{em2} is satisfied
		if and only if there exists a number $q$ ($0<q<1$) such that $b_n =O(q^{n})$ as $n\to\f$.

It is well-known that for an arbitrary sequence of positive numbers $a_n$,
the convergence $a_{n+1}/a_n\to a$ implies the convergence 
$\sqrt[n]{a_n}\to a$. Therefore, if $a_n$
is a weakly varying sequence, it is exponentially neutral. 
However, the converse is not generally true (see Remark 3.7 in Babayan et al. \cite{BGT}).

\sn{Slowly varying functions}
\label{S2.5}
We introduce the concept of slowly varying functions and briefly discuss their properties.
For a comprehensive treatment of this subject, see the monographs by Bingham et al. \cite{BGT} and Zygmund \cite{Zyg}.

\begin{den}
	\label{sv1}
A measurable function $u:(c,\f)\rightarrow \mathbb{R}$ $(c\geq 0)$ is called {\it slowly varying at zero (in Karamata's sense)} if it is positive for $x$ large enough, and for any $t>0$
\[
\lim_{x\rightarrow 0} \frac{u(xt)}{u(x)}= 1.
\]
\end{den}

\begin{den}
	\label{sv2}
	A measurable function $u(x)$ is called  {\it slowly varying at 0 in Zygmund’s sense} if for $x$ large enough, it is positive, and for any $\de > 0$, there exists a finite number $x_0(\de) > 0$ such that $x^{\de}u(1/x)$ monotonically increases, and $x^{-\de}u(1/x)$ monotonically decreases for  $x\geq x_0(\de)$.
\end{den}

It is known that the class of functions that are slowly varying in the Zygmund sense is a subclass of those that are slowly varying in the Karamata sense.

Note that a measurable function $u(x)$ is called {\it slowly varying at infinity} if the function $\tilde u(x):=u(1/x)$ is slowly varying at zero.	
	
For the next proposition see Zygmund \cite{Zyg} (Sect. V, Theorem 2.24). 
\begin{pp}
	\label{Zg-1}
Let $0<\be<1$ and $g(\la)$ be a function of bounded variation on 
any interval $(\vs,\pi)$, which is slowly varying at 0. 
Then, for the coefficients $a_n$ and $b_n$ of the function $\la^{-\be}g(\la)$ 
$(0<\la\leq\pi)$, the following relations hold:	
\begin{align*}
		&a_n\sim 2\pi^{-1}n^{\be-1}g(1/n)\G(1-\be)\sin(\pi\be/2),\\
		&b_n\sim 2\pi^{-1}n^{\be-1}g(1/n)\G(1-\be)\cos(\pi\be/2).
\end{align*}
\end{pp}
As noted in Section \ref{mem}, 
the definitions of memory in both the time and frequency domains are generally not equivalent due to the influence of slowly varying functions. However, they can become equivalent depending on the choice of these functions, as demonstrated in the following proposition
(see Beran et al. \cite{BFGK}, Theorem 1.3, and Ginovyan \cite{G2025}, Proposition 3.3.1).
\begin{pp}
	\label{memp1}
	Let $\{X(t),\, t \in \mathbb{Z}\}$ be a stationary process with
	covariance function $r(t)$ $(t \in \mathbb{Z})$ and spectral density $f(\la)$ $(\la\in[-\pi,\pi))$. The following assertions hold.
	\begin{itemize}
		\item[a)]
		If the covariance function $r(t)$ has the form
		\begin{equation}
			\label{mem2}
			r(t)=L_r(t)|t|^{2d-1},
		\end{equation}
		where $L_r(t)$ is slowly varying at infinity in Zygmund's sense, and either
		$d\in(0,1/2)$, or $d\in(-1/2, 0)$ and $\sum_{t \in \mathbb{Z}}r(t)=0$, then
		the process $X(t)$ has a spectral density $f$ that satisfies:
		\beq\label{mem3}
		f(\la)\sim L_f(\la)|\la|^{-2d} \q {\rm as} \q \la\to0
		\eeq
		with
		\beq\label{mem4}
		L_f(\la)=L_r(1/\la)\pi^{-1}\G(2d)\sin(\pi/2-\pi d).
		\eeq
		
		\item[b)] If the spectral density $f(\la)$ has the form
		\beq\label{mem5}
		f(\la)= L_f(\la)|\la|^{-2d}, \q 0<\la<\pi,
		\eeq
		where $d\in(-1/2, 0)\cup(0,1/2)$, and $L_f(\la)$ is slowly varying at zero in Zygmund's sense and is of bounded variation on $(a, \pi)$ for any $a>0$, then
		\begin{equation}
			\label{mem6}
			r(t)\sim L_r(t)|t|^{2d-1} \q {\rm as} \q t\to\f,
		\end{equation}
		where
		\begin{equation}
			\label{mem7}
			L_r(t)=2L_f(1/t)\G(1-2d)\sin(\pi d).
		\end{equation}
	\end{itemize}
\end{pp}

\sn{Approximation of the inverse of covariance matrix}

We revisit the model $X(t)=m+Y(t)$ $(t\in \mathbb{Z})$,
where $m$ is the constant unknown mean of $X(t)$ that needs to be estimated, and the noise $Y(t)$
is a centered, stationary process with a spectral
density  $f(\lambda)$ and a covariance function $r(t)$.

By equation (\ref{2.2}) for the variance $\Var(\widehat{m}_{BLU})$ of the BLUE $\widehat{m}_{BLU}$ we have 
\begin{equation}
	\label{2.2a}
	\Var( \widehat{m}_{BLU}) = (\mathbf{1}^TR_n^{-1}\mathbf{1})^{-1},
\end{equation}
where  $R_n$ is the covariance matrix and $\mathbf{1}^T=(1,\ldots, 1)$.

Thus, to calculate $\Var(\widehat{m}_{BLU})$, we need the explicit form
of the inverse $R_n^{-1}$ of the covariance matrix $R_n$. This is a challenging problem that can be solved for only a few simple models. Consequently, an approximation of 
the inverse $R_n^{-1}$ of the covariance matrix $R_n$ is required.

Observe first that the covariance matrix $R_n$ is a Toeplitz matrix generated by the spectral density $f(\la)$.
Specifically, we  have (see equations \eqref{2.1cm} and \eqref{mo1}):
\begin{equation}
	\label{Ap0}
R_n=B_n(f)=\|r(k-j)\|, \, \,k,j=1,\ldots, n, 
\end{equation}
 where
\begin{equation}
	\label{9ss4}
	r(k-j)=\int_\mathbb{T}e^{i(k-j)\la} f(\la)\,d\la. 
\end{equation}

For large values of \( n \), we anticipate that the Toeplitz matrix \( B_n(1/f) \) will serve as an approximation for the inverse of the matrix \( B_n(f) \), denoted as \( B_n^{-1}(f) \). It is important to note that the inverse of a Toeplitz matrix is generally not a Toeplitz matrix itself.
This expectation is based on the following consideration (see, e.g., Grenander and Szeg\H{o} \cite{GS}, Section 8.1).

Let $g_i\in L^2(\Lambda)$, $i=1,2$. Consider the infinite Toeplitz matrices 
$B(g_i)$ ($i=1,2$) generated by the functions $g_i$. Specifically,
\begin{equation}
	\label{Ap1}
B(g_i)=\|\widehat g_i(k-j)\|, \, \,  k, j \in \mathbb{Z}, \q i=1,2, 
\end{equation}
where $\widehat g_i$ are the Fourier coefficients of functions $g_i$, $i=1,2$:
\begin{equation}
	\label{Ap2}
	\widehat g_i(k)=\int_\mathbb{T}e^{ik\la} g_i(\la)\,d\la, \q i=1,2, \q  k \in \mathbb{Z}. 
\end{equation}

Then, we can form the product $B(g_1)B(g_2)$ and use Parseval's identity to obtain
\begin{equation}
	\label{Ap3}
	\sum_{t=-\f}^{\f}\widehat g_1(k-t)\widehat g_2(t-j) =\frac1{2\pi}\int_\mathbb{T}e^{-i(k-j)\la}g_1(\la)g_2(\la)\,d\la,
\end{equation}
which shows that the product $B(g_1)B(g_2)$ is an infinite Toeplitz matrix generated by the product function $g:=g_1g_2$. Thus, we have 
\begin{equation}
	\label{Ap4}
	B(g_1)B(g_2)=B(g_1g_2).
\end{equation}
If the inverse of $g_1$ is quadratically integrable, then $B(g_1)$ has an inverse generated by the function $1/g_1$. Specifically, we have $B^{-1}(g_1)=B(1/g_1)$.

Now, let $B_n(g_1)$ and $B_n(g_2)$  be the finite (truncated) Toeplitz matrices
generated by the functions $g_1$ and $g_2$, respectively.
The relation in \eqref{Ap4} does not hold for the matrices $B_n(g_1)$ and $B_n(g_2)$. 
 In other words, the product of truncated Toeplitz matrices is not necessarily a Toeplitz matrix.

Thus, it is necessary to find an analogous form of \eqref{Ap4}.
To establish a relation similar to \eqref{Ap4}, Szeg\H{o} proposed approximating
the trace of the product of Toeplitz matrices by the trace of a single Toeplitz 
matrix generated by the product of their respective generating functions. 
This approach is known as the trace approximation problem (see, e.g., Grenander and Szeg\H{o} \cite{GS}, Section 8.1).

More specifically,
let $\mathcal{H}=\{g_1, g_2,\ldots,g_m\}$ be a collection of integrable
real symmetric functions defined on $\Lambda$. Define
$$S_n:=\frac1n\tr\left[\prod_{i=1}^m B_n(h_i)\right]
\q {\rm and} \q
M:=\frac1n\tr\left[B_n\left(\prod_{i=1}^m g_i(\la)\right)\right],
$$
and let $$\De_n:=S_n - M.$$

The trace approximation problem involves identifying sufficient conditions that ensure
$\De_n\to 0$ as $n\to\f$ and estimating the rate of convergence of the approximation error $\De_n$ to zero as $n \to \infty$.

For a comprehensive treatment of the trace approximation problem, we refer to
Chapter 4 in Ginovyan \cite{G2025} and Chapter 8 in Grenander and Szeg\H{o} \cite{GS}. Additionally, see the papers by Dahlhaus \cite{D,D6}.

Below, we discuss the idea that the Toeplitz matrix $B_n(1/f)$ can approximate the inverse $B_n^{-1}(f)$ in a non-rigorous manner.

Denote the entries of the matrix $B_n(1/f)$ by $a(k-j)$, $k,j=1,\ldots, n$, which are the Fourier coefficients of the function $f^{-1}/(2\pi)$:
\begin{equation}
	\label{9ss3}
	a(k-j)=\frac 1{(2\pi)^2}\int_\mathbb{T}e^{i(k-j)\la} \frac{1}{f(\la)}\,d\la. 
\end{equation}

For large $n$, by using Parseval's identity, we can write 
\begin{align*}
	c_{k,j}&= \left(B_n(1/f)B_n(f)\right)_{k,j}
	=\sum_{t=1}^{n}a(k-t)r(t-j)\\	\nonumber
	&\approx \sum_{t=-\f}^{\f}a(k-t)r(t-j)=\frac1{2\pi}\int_\mathbb{T}e^{-i(k-j)\la}\widehat a(\la)
	\ol{\widehat r(\la)}\,d\la\\
	&=\frac1{2\pi}\int_\mathbb{T}e^{-i(k-j)\la}{f(\la)}\frac1{f(\la)} \,d\la=\de_{k,j}, 
\end{align*}
where $\de_{k,j}$ is Kronecker's delta, and $\approx$ indicates an approximation for large $n$.
This leads to the following approximation of the quadratic form in \eqref{2.2a}:
\begin{equation}
	\label{9ss5}
	\frac1{n}\mathbf{1}^T[B_n(f)]^{-1}\mathbf{1}\approx \frac1{n}\mathbf{1}^TB_n(1/f)\mathbf{1}
	= \frac{1}{2\pi}\int_{-\pi}^{\pi} f(t) F_n(t)\,dt,
\end{equation}
where $\mathbf{1}^T=(1,\ldots, 1)$ and $F_n(t)$ is the Fej\'er kernel (see \eqref{s6}). 
If $1/f$ is continuous at $0$, then by Proposition \ref{F-Ker}, we have
\begin{equation}
	\frac{1}{n}\mathbf{1}^\top B_n(1/f)\mathbf{1} \longrightarrow 1/f(0).
\end{equation}

Results on the inversion of Toeplitz matrices with generating functions
from the class $\mathcal{F}_\alpha(g)$, where $\alpha> -\frac{1}{2}$
and $g$ is a sufficiently smooth and positive factor, were obtained 
by Bleher \cite{Bl} and in a series of 
papers by Rambour and Seghier (see \cite{RaSe} and references therein).

%\s{Remarks,	complements and bibliographical notes}

%\section*{Acknowledgments}
%The author would like to thank the associate editor and two anonymous referees for their careful review of the manuscript 
%and valuable remarks, comments and suggestions.

%\section*{Acknowledgements}
%The author would like to thank the referees for their constructive
%comments and suggestions.

%\section{Proofs}
%\label{Proofs}

\end{document}